\documentclass[11pt]    {article}        
\usepackage{amsmath,amsthm,amsfonts,graphicx}
\usepackage{multirow}
\usepackage[T1]{fontenc} 
\usepackage{amsmath,amsthm,amsfonts,graphicx}        
\usepackage{tikz}
\usepackage{caption}
\usepackage{subcaption}
\def\smallddots{\mathinner{\raise7pt\hbox{.}\raise4pt\hbox{.}\raise1pt\hbox{.}}} 
\def\smallsdots{\mathinner{\raise1pt\hbox{.}\raise4pt\hbox{.}\raise7pt\hbox{.}}}

\DeclareMathOperator{\diag}{diag}

\newtheorem{theorem}{Theorem}[section]

\numberwithin{equation}{section} 
\numberwithin{table}{section}
\newtheorem{lemma}{Lemma}[section]

\newtheorem{corollary}{Corollary}[section] 
\newtheorem{proposition}{Proposition}[section]

\newtheorem{algorithm}{Algorithm}[section]
\newtheorem{example}{Example}[section]
\newtheorem{definition}{Definition}[section]

\newtheorem{observation}{Observation}[section]

\newtheorem{remark}{Remark}[section]

\begin{document}
 
\title{Polynomial Root-Finding and \\ Algebraic Eigenvalue Problem} 
\author{Victor Y. Pan} 
\author{Victor Y. Pan$^{[1,2,3],[a]}$ \\
$^{[1]}$ Department of Computer Science \\
Lehman College of the City University of New York \\
Bronx, NY 10468 USA \\ 
$^{[2]}$ Ph.D. Program in   Mathematics and \\
$^{[3]}$ Computer Science\\
The Graduate Center of the City University of New York \\
New York, NY 10036 USA \\ 
$^{[a]}$ victor.pan@lehman.cuny.edu \\ 
http://comet.lehman.cuny.edu/vpan/  
}
\date{}

\maketitle


\begin{abstract}
Univariate polynomial root-finding has been studied for four millennia and very intensively in the last decades. Our  new {\em near-optimal} root-finders  approximate all complex zeros of a polynomial 
almost as fast as one accesses its coefficients  with the precision required for the solution within a prescribed error bound.  The root-finders work even for a {\em black box polynomial}, defined by an  oracle (black box subroutine) 
for its evaluation rather than by its coefficients. Hence they (i) support
approximation of the eigenvalues of a matrix in a record Las Vegas expected bit operation time and (ii)
are particularly fast for polynomials that can be evaluated fast 
 such as the sum of a few shifted monomials
 of the form $(x-c)^d$ or
  a Mandelbrot-like polynomial defined by a recurrence. 
 Our divide and conquer algorithm of 
ACM STOC 1995 is the  only other known near-optimal polynomial
root-finder, but it extensively uses the coefficients, is quite  involved, and has never been implemented, while according to extensive numerical experiments with standard test 
polynomials already a slower initial implementation of our new root-finders competes with user's choice package of root-finding subroutines MPSolve and supersedes it more and more significantly as the degree of a polynomial grows large.
 We elaborate upon design and analysis of our algorithms, comment on their potential  heuristic acceleration, and briefly cover polynomial root-finding by means of functional iterations.
 Our techniques can be of independent interest.   
\end{abstract}

  


\clearpage

\noindent{\bf \Large PART I: State of the Art and  Our  Progress (Outline)}

\section{State of the Art and our Main Results}\label{s1}
 

\subsection{Polynomial
Root-finding: Problems, Impact, and Earlier Studies}\label{prbs}
According to \cite{B40,P97,P98,MB11},
{\em univariate polynomial root-finding}, that is, the solution of a polynomial equation $p(x)=0$,  had been the  central problem  of Mathematics and Computational Mathematics for about four millennia since Sumerian times; it was still a central problem well into the 19th century. In a long way from its rudimentary level to  modern version, its study 
has been responsible for the introduction of mathematical rigor and notation as well as the concepts of irrational, negative, and complex numbers, algebraic groups, rings, and fields  etc.  

The existence of  complex roots of the equation $p(x)=0$ is called the {\em Fundamental Theorem of Algebra}. The attempts of its proof  by d'Alembert (1746),  Euler (1749), de Foncenex (1759), Lagrange (1772), Laplace (1795), and Gauss (1799) had gaps (cf. \cite{S81}), and
the first rigorous proof is due to an amateurish mathematician Argand (1806, revised in 1813). 
 
Already
 about 500 BC it was  proved in Pythagorean
school that the roots of the equation $x^2=2$
were irrational radicals, 
while even with the help of radicals  one cannot express the roots  
rationally if the degree $d$ of $p(x)$ exceeds 4,  according to Ruffini (1799,1813) and Abel (1824); moreover, one cannot do this already for specific  equations such as $x^5-x-1=0$
or $x^5-4x-2=0$,   according to Galois theory of 1830, 1832.

Nevertheless, in his  {\em constructive proof} of the Fundamental Theorem of Algebra in \cite{W24}, Herman Weyl has {\em rationally approximated}   within any fixed  positive error tolerance all roots of $p(x)=0$ lying in a fixed square in the complex plane, that is, all roots if the square is large enough.

This solved the following {\bf Problem 0} of  modern polynomial root-finding:  

{\bf INPUT:}   $d+1$ complex coefficients $p_0,p_1,\dots,p_d$
of a 
 polynomial \begin{equation}\label{eqpoly}
p=p(x):=\sum_{i=0}^dp_ix^i:=
p_d\prod_{j=1}^d(x-z_j),
\end{equation}
 a positive error tolerance $\epsilon=1/2^b>0$, and a bounded domain  $\mathbb D$ in the
complex plane, e.g., a disc or a rectangle.
 
{\bf OUTPUT:}
 approximations within
$\epsilon$
 to every zero of $p(x)$ (aka every root of $p(x)=0$)\footnote{Hereafter we freely say just ``root" or ``zero'' for short, always count them with their {\em multiplicities}, and do not distinguish multiple roots from pairwise isolated root clusters of diameters less than $\epsilon$.} lying in the domain $\mathbb D$
 such that 
\begin{equation}\label{eqrts} 
|\tilde z_j- z_j|\le \epsilon~{\rm if}~z_j\in \mathbb D.
 \end{equation}
  
\begin{example}\label{exallrts} {\rm [Approximation of all 
$d$ roots.]} For
a real \begin{equation}\label{eqR}
R\ge \max_{j=1}^d\{|x_j|\}
\end{equation}
(see \cite{H74} for simple expressions for such bound $R$) and $\mathbb D$ denoting the disc  $D(0,R):=\{x:~|x|\le R\}$,  Problem 0 amounts  to approximation of all $d$ roots $z_j$ within $\epsilon$.
\end{example}

\begin{example}\label{exrlnrrl} {\rm [Numerical real root-finding.]} Root-finding in  the 
 real line or its segment (real root-finding)  is highly important because in many applications, e.g., to optimization in algebraic geometry and geometric modeling,  
   only real roots  of a polynomial are of interest, and  typically they are much less numerous than all $d$ complex roots \cite{K43,EK95}. Now consider Problem 0 where $\mathbb D$ denotes   the minimal  rectangle
   in the complex plane   that covers the $\epsilon$-neighborhood 
 of a fixed real line segment. This task  addresses approximation of both real and {\em near-real roots}, that is,  their perturbations due to small input and rounding errors. The known real polynomial root-finders
 (cf. \cite[Chs. 2 and 3]{M07}) ignore this important issue.
\end{example}

Solution  of Problem 0 can be extended to the   eigenvalue problem for  matrices 
and  matrix polynomials  (see our Secs. \ref{srteig} and \ref{scmplests} and Part (C) of Sec. \ref{sdrctacc}) and is involved into the solution of a multivariate polynomial system of equations \cite{CLOS97}; both of them are  much studied areas   
 of modern computations.
\cite[Introduction]{MP13} specifies some applications of polynomial root-finding to  
computer algebra, signal processing, geometric modeling, control, and financial mathematics. Furthermore, Problem  0 for a domain $\mathbb D$ containing all roots is closely linked to polynomial factorization (see Secs. \ref{scprtf} and \ref{4.15}), which has applications to  
time series analysis, Wiener filtering, noise variance estimation, co-variance matrix
 computation, and the study of multichannel systems \cite{W69,BJ15,B83,DM89,DM90,VD94}.

 Hundreds of efficient polynomial root-finders have been proposed  \cite{M07,MP13}
 and keep appearing
  (see   \cite{PT13/16,BR14,EPT14,KRS16,LV16,PZ17,P17b,BSSXY16,BSSY18,IPY18,P19,P20,LPKZ20,IP20,IP21,M21,IP22,IM23,PGLZ23}, and the bibliography therein).  In \cite{S81}  Smale proposed  
to study Problem 0 in terms of its arithmetic complexity and  computational precision. Together they would define its bit operation complexity, but Smale deduced no complexity estimates.  \cite{S82}    Sch{\"o}nhage did this in \cite{S82}. His much involved 
root-finder, presented with terse complexity analysis on  46 pages, approximates all $d$ roots by using 
$\tilde O(bd^3)$ bit operations, that is, 
$bd^3$ bit operations
up to a  factor in $bd$.\footnote{Error amplification factors $2^{cd}$ are inevitable 
in polynomial root-finding, e.g., in routine shift of the variable, and assume that $b\ge d$
throughout our paper; 
alternatively one can replace $b$ with $b+d$
in our upper bounds on the computational cost.} 

His work prompted intensive  study  of root-finding by researchers in the Theory of Computing 
 in the 1980s and 1990s,
 culminated at ACM STOC 1995  with our {\em near-optimal}   
 divide and conquer root-finder
of \cite{P95}
 (also see its journal versions \cite{P96,P02} and surveys \cite{P97,P98}). Both upper
 bound in $\tilde O(bd^2)$ and
  lower bound of $0.25(d+1)bd$ bit operations of \cite{P95,P02} decrease
  by a factor of $d$ 
 for  polynomials 
 of Def. \ref{defclstr}
of Sec. \ref{scprtf}, whose all zeros and tiny clusters of zeros are pairwise sufficiently well isolated
(see Thms. \ref{thbscfct}   and \ref{thfrmfctr} 
 and Cor.  \ref{cofrmfctr0} in Sec. \ref{scprtf}). 
 
\subsection{Reduction to the case of the unit disc}\label{srdcttountdsc}

 Hereafter 
  $D(c,\rho):=\{x:~ |x-c|\le \rho\}$ and
 $C(c,\rho):=\{x:~ |x-c|=\rho\}$  denote the disc of radius $\rho>0$ centered at a complex point  $c$ and its boundary circle, respectively, and
{\bf Problem $\widehat 0$} denotes Problem 0 where the domain $\mathbb D$
is the unit disc $D(0,1)$; we can extend its solution  to root-finding in the disc  $D(c,\rho)$ for any pair of complex $c$ and positive $\rho$ by means of applying the linear map 
 \begin{equation}\label{eqshft}
x\mapsto   
\frac{x-c}{\rho},~p(x)\mapsto  t(x)=p\Big(\frac{x-c}{\rho}\Big),~D(c,\rho)\mapsto  D(0,1),
\end{equation}
which shifts the variable $x$ by 
$c$ and then scales $x-c$ by $1/\rho$. 

Next we reduce Problem 0  for any domain $\mathbb D$ to Problem $\widehat 0$.

{\bf (i) Two ways from the roots in the unit disc to all roots.} 
Given a value $R$ of 
 (\ref{eqR}), we can scale the variable, $x\mapsto y=\frac{x}{R}$, $p(x)\mapsto p(y)=p(\frac{x}{R})$, and then  
 approximate every root $x_j$, for
 $j=1,\dots,d$,  within
   $\epsilon R$ by solving Problem $\widehat 0$ for $p(y)$.
   
 Here is Sch{\"o}nhage's alternative reduction 
 in \cite{S82}  of Problem 0  for all $d$ roots to Problem $\widehat 0$: 
  \noindent Given an algorithm $\mathcal A$ that
solves Problem $\widehat 0$,  compute also approximations $\tilde y_j$ within $\epsilon$ to the reciprocals $y_j=1/z_j$ of all roots lying outside the disc $D(0,1)$  --  simply by applying this algorithm  to the
reverse polynomial
 \begin{equation}\label{eqpolyrev}
p_{\rm rev}(x):=x^dp\Big(\frac{1}{x}\Big)=\sum_{i=0}^dp_ix^{d-i},~
p_{\rm rev}(x)=p_0\prod_{j=1}^d\Big(x-\frac{1}{z_j}\Big)~{\rm if}~p_0\neq 0. 
\end{equation}
Write  $\Delta_j:=\tilde y_j-y_j$ and deduce that
$\frac{1}{y_j}-\frac{1}{\tilde y_j}=\frac{\Delta_j}{\tilde y_j~y_j}$, and so
$\frac{\tilde y_j}{y_j}-1=
\frac{\Delta_j}{y_j}=\Delta_jz_j$. 
Recall that $|\Delta_j|\le \epsilon$
by assumption 
and obtain that
$|\frac{\tilde y_j}{y_j}-1|\le \epsilon |z_j|$, thus arriving at the following

\begin{proposition}\label{proprev} 
The reciprocal of an approximation within $\epsilon$ to a zero of $p_{\rm rev}(x)$  
in the disc $D(0,1)$
approximates the associated zero  of $p(x)$ lying outside  
that disc within a relative error at most $\epsilon$.\end{proposition}

{\bf (ii) From roots in
  the unit disc to
  roots  in a bounded domain.} Given a bounded domain $\mathbb D$
 on the complex plane,
 we can successively (a) cover it with a minimal disc  $D(c,\rho)$,
(b) map this disc into the disc $D(0,1)$, (c) solve Problem $\widehat 0$ for the resulting polynomial $t(x)$ of (\ref{eqshft}), and finally (d) use the map (\ref{eqshft}) again -- this time to  approximate all roots lying in   the intersection of the disc  $D(c,\rho)$ with the domain $\mathbb D$. In this way we also approximate extraneous roots lying in $D(c,\rho)$ but not in  $\mathbb D$. In Sec. \ref{ssbdrdei} we curtail such an extraneous work 
(see  Prop. \ref{propmk} and Remark \ref{remk}).

\subsection{Black box root-finders and approximation of matrix eigenvalues}\label{srteig}
 
   A nontrivial variant of Newton's iterations by Louis and Vempala  at FOCS 2016 \cite{LV16} approximates within  
a fixed positive
 $\epsilon\le 1/2$
 an absolutely largest zero of a real-rooted  polynomial 
$p$ of a degree $d$  
at the  cost of its  evaluation with a precision in $O((b+\log(R))\log(R~d)))$
 at $O((b+\log(R))\log(d))$ points for $b:=\log_2(R/\epsilon)$
 and a fixed  $R$ satisfying (\ref{eqR}).   While the  previous polynomial root-finders have been devised for polynomials given in the {\em monomial basis}, that is, with their coefficients,  Louis and Vempala in \cite{LV16} never involve them, working for a {\em black box polynomial} $p$ -- given with an  {\em  oracle} (black box  subroutine) for its evaluation, and then    extend their root-finder to efficient approximation of the eigenvalues of a matrix.
Next we state the relevant tasks.\\
{\bf Problem 1 (Black Box Polynomial Root Approximation)} is Problem 0 for a black box polynomial $p$.
 \smallskip 

\noindent{\bf Problem 2: Approximation of  the Eigenvalues of a Matrix.}\footnote{As in \cite{LV16}, the problem covers or can be extended from a matrix to a {\em polynomial matrix}, aka {\em matrix polynomial}.}\\
INPUT:  $\epsilon=
1/2^{\tilde b}>0$, an  integer $d>0$,  a $d\times d$  matrix $M$, and a domain $\mathbb D$
on the  complex plane.
 \\
 OUTPUT: 
 Approximations 
  within $\epsilon$ to  every eigenvalue of $M$ 
lying in the domain $\mathbb D$.

 The output of Problems 0, 1, and 2  may include multiple approximations 
within $\epsilon$ to the same root or eigenvalue, but the next modifications 
exclude such a  confusion.
 
\noindent{\bf Problems $0^*$, $1^*$, $2^*$:  Isolation of Clusters of Roots and    Eigenvalues.} 
Under the assumptions of 
Problems 0, 1, and 2 compute at most $d$ disjoint discs of radii at most
 $\epsilon$ (hereafter $\epsilon$-{\em discs}) that cover all roots or eigenvalues lying in the domain $\mathbb D$;
 also compute the number of the roots or eigenvalues lying in each $\epsilon$-disc. 
 
 \begin{remark}\label{reclsttrisl} If  we collapse all zeros of $p$ lying in every $\epsilon$-disc into a single multiple zero  placed in the center of the disc, then these zeros will lie at the distances 
  greater than $2\epsilon$
  from each other.
  If a polynomial $p$ is 
 square-free (having only simple zeros) and has integer coefficients, then Problem 0$^*$ turns into the well-known Problem of Root-isolation (see Problem ISOL and Remark \ref{rertisl} in Sec. \ref{sisol}).  These comments can be immediately extended to Problems  1$^*$ and 2$^*$.
\end{remark}
  
\subsection{More on  Black Box Polynomial Root-finders}\label{smbb}
\subsubsection{Why we should use them}\label{swhy} 
 
 Namely, the algorithm of
Fiduccia \cite{F72}
 and  Moenck and Borodin
 \cite{MB72} (see \cite[Sec. 4.5]{BM75} for  historical account) evaluates a polynomial of degree  $m$ at $m$ points in  $O(m\log^2(m))$ ops if its coefficients are given, versus $2m^2$ ops involved in Horner's algorithms. Numerically stable algorithms of \cite{M21,IM23} yield such acceleration also in terms of bit operation cost.

One can perform black box 
evaluation at $q>d$ points $x_1,\dots,x_q$ and then interpolate to $p$.
The bit operation
 cost of interpolation, however,
 is unbounded unless
the $q$ points are well isolated from  the roots, but it is a challenge to yield such isolation.
Furthermore, even if the isolation is ensured, 
the cost of  interpolation 
 exceeds the cost of the solution of Problems 0, 0$^*$,
 1, 1$^*$, 2, and 2$^*$ in the cases where domain $\mathbb D$  contains only a small number of roots as well as where a polynomial $p$ can be evaluated fast because root-finding for a polynomial is routinely reduced to its evaluation at sufficiently many points \cite{M07,MP13}.  
 
 For example, one only needs $O(\log(d))$  ops for evaluation of
 a Mandelbrot-like polynomial
$p:=M_k(x)$, where 
$M_0:=x,~
M_{i+1}(x):=xM_i^2(x)+1$, $i=0,1,\dots,k$, $d=2^k-1$, or   the sum of a bounded number of shifted
monomials, such as $p:=\alpha (x-a)^d+\beta(x-b)^{d}+\gamma(x-c)^{d}$ for six constants $a$, $b$, $c$, $\alpha$, $\beta$, and $\gamma$.
For comparison, if we compute the coefficients
of such a polynomial and then use them  to  evaluate it at  a point $x$ by applying Horner's algorithm, then we would need  $2d$ 
ops, and this bound is optimal (see \cite{P66,K81/97} or
\cite[Section ``Pan's  method"]{S74}). 

In computational practice numerical stability of polynomial evaluation 
 is highly important,
 and it is ensured when  a polynomial
 is represented in the Bernstein basis  \cite{L53}. Transition to it from the monomial basis is costly but is not needed where a root-finder works for a black box polynomial. 
 
One can benefit from black box root-finders   even where an input polynomial is  given with its coefficients.

(i) Application of a black box root-finder to the characteristic polynomial  $p(x)=\det(xI-M)$
of a $d\times d$ matrix  $M$  only involves the values of this polynomial but not its coefficients,
whose overall size exceeds the size (a norm)
of the matrix $M$ by a factor of $d$. Hence extension of such a root-finder to Problem 2 saves 
a factor of $d$ in  bit precision  and bit operation  count versus the known solutions of Problems 2  and 2$^*$, based on \cite[Sec. 21]{S82}. 

 (ii) Map (\ref{eqshft}) 
 is a valuable  
 root-finding tool, but its application is limited because it destroys sparsity and blows up computational precision and bit operation cost. Similar comments can be applied to maps (\ref{eqdndgnr})  and (\ref{eqdnd})  
of Sec. \ref{seiIII}.
   Black box polynomial root-finders keep all benefits of using those maps but avoid 
    the latter notorious  problems by  involving no coefficients of  $t(x)$ and  $p_{h}(x)$
  of these maps.
  
   (iii) Deflation of a factor of a polynomial is a well-known cause of the coefficient swell, e.g., for degree 300  factors of 
 $x^{1000}+1$, which blows up  precision and bit complexity of computations.   For an exception,  one can readily deflate a small degree factor $f(x)$ and approximate its zeros, but the coefficients of   the high degree factor $g(x)=p(x)/f(x)$ still tend to blow up. We, however, avoid their computation when we apply a black box
   root-finder to the polynomial $g(x)$ represented by the pair of $p(x)$ and $f(x)$.

\subsubsection{Progress at FOCS 2016 and its limitations}\label{sadtd}
          The known polynomial root-finders, except \cite{LV16}, extensively use the coefficients of $p$, 
and we must be careful 
in  adopting their techniques for Problems 1 and 1$^*$.

Besides proposing novel root-finder for  a black box polynomial with extension to approximation of an eigenvalue of a matrix, Louis and Vempala in \cite{LV16} presented non-trivial techniques for bounding precision of root-finding computations. Extension of these techniques helped us simplify dramatically the estimation of the bit operation complexity of our root-finders compared to those of
 \cite{S82,P95,P02,K98}, which bypass estimation of computational precision.

On the other hand, direct extension of the root-finder of
\cite{LV16}  cannot help solve even Problem 1.  Indeed, 
 repeatedly deflating $p$, by
 peeling off large roots, can accumulate errors that grow too quickly, together with the coefficient length  of the factors: notice, e.g., the swell of the coefficients of $p(x)=x^d+1$ in transition to its factors of degree $d/2$, say. This is true
 even if one performs deflation in the black box setting implicitly, based on Eqn.  (\ref{eqimpldfl})
 of Sec. \ref{sseitsts}. Indeed,  the high order derivatives involved in 
 the algorithm of \cite{LV16} are
  proportional to the
  coefficients and together with them 
  swell in deflation.
  
 Even if deflation succeeded, the algorithm of \cite{LV16}  would have only solved the well-known and much easier problem of root-finding for a real-rooted polynomial (cf.   \cite{P89,BT90,DJLZ97}).
 Moreover, the resulting algorithm for that much easier problem would involve more than $b^2d^2$ bit operations, exceeding by a factor of $b$  the upper bound of \cite{P95,P02} for  general polynomial 
root-finding.

Furthermore, we accelerate the solution of Problem 0 by a factor of $m/\log^2(m)$ by incorporating 
fast multipoint polynomial
evaluation of \cite{F72,MB72,K98,M21},
but we would lose this chance
if we  successively approximated the zeros of $p$  or  matrix eigenvalues by using deflation.  
  
  Thus, instead of applying  
   high order Newton's iterations of 
  \cite{LV16} and deflation, we recall and advance subdivision iterations of \cite{W24,H74,R87,P00}.

  \subsection{Our complexity estimates}\label{scmplests}
  
 Unlike the root-finders of \cite{S82,P95,P02,K98},
 seeking all $d$ roots, of \cite{LV16},
 seeking a single absolutely largest root and of  \cite{P95,P02}, seeking all $d$ roots or all roots in a disc well isolated from the external roots, {\em we approximate the roots  in any domain
 $\mathbb D$  in the complex plane.} 
Root approximation within $\epsilon$ in a domain $\mathbb D$   involves the roots in  
  its $\epsilon$-neighborhood; our complexity estimates grows with the number of roots in a   little larger but very small  neighborhood of $\mathbb D$.
 
 \begin{definition}\label{defepsbt}
 For a polynomial $p=p(x)$,
 a  domain $\mathbb D$
 in the complex plane
  having a diameter $\Delta$, and a positive $\beta$ such that $\beta\Delta>2\epsilon$, write $m=m(p,\mathbb D,\beta)$
  if precisely $m$ roots of $p(x)$  lie in 
 the  $\beta\Delta$-neighborhood of 
 $\mathbb D$. 
 \end{definition}
  
We will prove our Main Theorems 1-4 below  for
$m$ of Def. \ref{defepsbt},  any fixed constant $\beta>2\epsilon/\Delta$, and either  $p(x)=\det(xI-M)$ in the case of Problems 2 and 2$^*$ or  $p(x)$ of
(\ref{eqpoly}) otherwise.\footnote{Hereafter, whenever we say  that subdivision iterations are applied to a input domain $\mathbb D$ of 
 diameter $\Delta$ that contains $m$
roots, we mean that so does its $\beta$-neighborhood.}
  
\begin{theorem}\label{thhlacovrl0} {\bf (Main Theorem 1.)}        
One can solve Problems 1 and 1$^*$ by querying the polynomial evaluation oracle at 
$\tilde O(m^2)$ locations with a computational precision  $O(b\log(d))$.
\end{theorem} 
 
\begin{theorem}\label{theigen} {\bf (Main Theorem 2.)}
One can solve Problems 2 and 2$^*$ at a Las Vegas randomized
 bit operation cost
 \begin{equation}\label{eqblneig}
 \mathbb B_{\rm eigen}=\tilde O(\tilde b d^{\omega}m^2)
 \end{equation} 
  for $\tilde b:=\log_2(||M||_F/\epsilon)$,
  the Frobenius norm $||M||_F$ of a  matrix $M$, and $\omega$  denoting an exponent of {\em feasible} or {\em unfeasible} matrix multiplication
(MM).\footnote{The current records are about 2.7734 for exponents of  
   feasible MM \cite{P82,P17a,P18,KS20,HS23}, unbeaten since 1982, and about 2.372 for exponents of   unfeasible MM \cite{WXXZ24}.
 The  exponents below 2.77 have only been reached at the expense of blowing up the  size of MM by  immense factors.}   
\end{theorem}  

 
  \begin{theorem}\label{thblntm} {\bf (Main Theorem 3.)}
  One can solve Problems 0 and 0$^*$ at a Las Vegas randomized
 bit operation cost $\tilde O(bdm))$
for a polynomial $p$ of 
  (\ref{eqpoly})
 normalized such that 
 $\sum_{i=1}^d|p_i|=1$.
\end{theorem}
 
   {\bf Random Root Model}: The zeros of $p(x)$, for
   $p(x)=\det(xI-M)$ in Problems 2 and 2$^*$ and  for $p(x)$ of 
 (\ref{eqpoly}) otherwise, are  independent identically distributed {\em (iid)}  random variables
    sampled under the uniform probability distribution on the disc
    $D(0,R)$ 
    for a large real $R$ such that
    the value $\frac{R^2}{d}-1$
    exceeds a positive constant.\footnote{(i) The  assumption of the model is a little stronger  than that of our Def. \ref{defclstr}, supporting such a decrease  in \cite{P95,P02}. (ii) Random coefficient models fails at the basic step of the 
shift of the variable $x$, which no known probability distribution withstands.}
    \begin{theorem}\label{thrndrtmd} {\bf (Main Theorem 4.)}
     Under the Random Root Model one can decrease by a factor of $m$ the upper bounds of  Thms.  \ref{thhlacovrl0}-\ref{thblntm} on the computational precision and bit operation complexity.
\end{theorem}
    
In view of  Remarks \ref{reclsttrisl} and \ref{rertisl} we can readily extend our Main Theorems  to {\em Polynomial Root  Isolation.}
    
 
\section{Fast subdivision  
root-finding: overview and soft e/i tests}\label{sbckbrf}

We support  Main Theorems 1 -- 4 with
our accelerated subdivision  
root-finders. They are  reduced to recursive application of soft exclusion/inclusion (e/i) tests, which  decide whether  a fixed complex disc or its small neighborhood
contains roots.
  In Sec. \ref{sseitsts} we outline our  two approaches 
 to devising e/i tests for a black box polynomial and the unit disc  $\mathcal D=D(0,1)$. We expand these two outlines 
  in  Part II,  made up of Secs. \ref{sbcgr}--\ref{sblnrtfdg}, and Part III, made up of Secs. \ref{sbdef}--\ref{sblnrtfnd}, respectively. 
 Our root-finders of Parts II and III
 only approximate (rather than isolate)
 the roots and support    complexity estimates  that are inferior 
 to those of our Main Theorems  by a factor of $b$.
  In Part IV, made up of Secs. \ref{sextrrk} -- \ref{sextr} and outlined in Sec. \ref{saccIV}, we fix both deficiencies, thus   
 proving the Main Theorems.  
 In Sec. \ref{sorg} we further comment about the contents of our paper, including its Part V,
made up of Secs. \ref{scprtf} -- \ref{sexclexc2}. 

\subsection{Subdivision root-finders and reduction to e/i tests}\label{ssbdrdei}

We depart from  popular subdivision
 root-finders,
 traced back to \cite{W24,H74,R87,P00}. They extend root-finding in a line segment by means of bisection 
 to polynomial root-finding in a square on the complex plane and have been 
extensively used in Computational 
  Geometry under the name of {\em Quad-tree Construction}. 
 Next we outline them.

(A) First consider  root-finding in a square on the complex plane.
 Call that square {\em suspect}, recursively  divide it and every other suspect square into four congruent sub-squares, and to each of them apply {\em exclusion/inclusion (e/i)} test to decide whether the  square   contains a root or does not. If it does,  call the square  suspect and process it in the next iteration;  otherwise discard it
(see Fig. \ref{figsbd}a).
 In Def. \ref{defsft}                                                                                                                      
  we generalize e/i tests to $\ell$-{\em tests}, which decide whether a disc in the complex plane contains  $\ell$ or fewer roots provided that it contains at most $m$ roots.  E/i test is a 1-tests, and whenever we need $\ell$-tests for $\ell>1$ (e.g., in Sec. \ref{sextrrk} and Part 
IV), we reduce them to e/i tests (see
 Sec. \ref{smvia1} and \ref{seiann}).

\begin{figure} 
\centering
\begin{subfigure}[b]{0.4\textwidth}
\centering
\resizebox{!}{0.15\textheight}{
\begin{tikzpicture}
\draw (-4, -4) -- (-4, 4) -- (4, 4) -- (4, -4) -- (-4, -4);
\draw (-4, 0) -- (4, 0);
\draw (0, -4) -- (0, 4);
\draw (-4, 2) -- (0, 2);
\draw (0, -2) -- (4, -2);
\draw (-2, 4) -- (-2, 0);
\draw (2, 0) -- (2, -4);
\draw (-1, 4) -- (-1, 2);
\draw (-2, 3) -- (0, 3); 
\draw (0, -3) -- (2, -3);
\draw (1, -2) -- (1, -4);
\draw node at (-1.2, 2.6) {*};
\draw node at (-1.7, 2.2) {*};
\draw node at (1.3, -2.6) {*};
\draw node at (1.7, -2.3) {*};
\end{tikzpicture}}
\caption{Four roots marked by asterisks lie in  suspect sub-squares; the other (empty) sub-squares are discarded.}
\end{subfigure}
\hfill
\begin{subfigure}[b]{0.4\textwidth}
\centering
\resizebox{!}{0.2\textheight}{
\begin{tikzpicture}
\draw[red] (0, 0) circle (1.414);
\draw (-1, 1) -- (1, 1) -- (1, -1) -- (-1, -1) -- (-1, 1);
\draw[red] (0, 0) circle (1); 
\draw (-0.707, 0.707) -- (0.707, 0.707) -- (0.707, -0.707) -- (-0.707, -0.707) -- (-0.707, 0.707);
\draw[red] (0, 0) circle (0.01); 
\draw node at (-1.2, 0) {*};
\draw node at (0, 1.2) {*};
\draw node at (0.9, -0.9) {*};
\end{tikzpicture}}
\caption{Both    
 exclusion and $\sqrt 2$-soft inclusion criteria hold for  e/i test applied to the disc bounded by  internal  circle.}     
\end{subfigure}
\caption{}
\label{figsbd}
\end{figure} 
     
   \begin{observation}\label{obsubsid}
A subdivision iteration applied to a square  with $m$ roots in its $\beta$-neighborhood 
(cf. Def. \ref{defepsbt}  and Remark \ref{remk})
(i)  decreases by twice the diameter  of a suspect square and 
   (ii) processes at most $4m$ suspect squares, (iii) whose centers approximate every root within the distance of one half of their diameter.                                                                                                                                                       
\end{observation}
 
\begin{proof}
Readily verify claim (i). Claim (ii) holds because a root can lie  in at most  four suspect squares, and even in at most two unless it is the shared vertex of the four. Claim (iii) holds because a subdivision iteration never discards a square containing a root.
\end{proof}

\begin{corollary}\label{corsubsid}
Let  subdivision iterations
be applied to a square $S$ 
with side length $\Delta$ and diameter  $\Delta':=\Delta\sqrt 2$ containing precisely $m$ roots. Then  (i)  the $k$-th iteration  outputs at most $4m$ suspect squares whose centers  approximate within  $\epsilon:=\Delta'/2^{k+1}$ all the $m$ roots lying in $S$, so that $k+1\le
\lceil\log_2(\frac{\Delta'}{\epsilon})\rceil$, and (ii) the first
$k$ iterations together involve at most  $4mk$ e/i tests.
\end{corollary}
 
 (B) Seeking roots in any fixed complex domain  $\mathbb D$,
we apply subdivision iterations to its minimal covering square
$S$ with diameter $\Delta'$ and while performing an iteration also discard all squares not intersecting the domain $\mathbb D$, even if they contain roots. The following proposition bounds the overall number of e/i tests performed in $k$ iterations.

\begin{proposition}\label{propmk}
Let $k$ subdivision iterations be applied to a complex domain $\mathbb D$, let precisely  $m_i$  roots lie in the $\Delta_i'$-neighborhood of $\mathbb D$
for  $\Delta'_i$ denoting the diameter of a suspect square at the $i$-th subdivision iteration,
 and let precisely $m_i$ e/i tests be performed at the $i$-thiteration, so that
\begin{equation}\label{eqmi}
 m_i\le 4^{i+1},~i=0,1,\dots,k-1.
\end{equation}  
 
Then 

(i) the $(\Delta_h')$-neighborhood of $\mathbb D$  covers  the union 
of all suspect squares processed in the $h$-th
subdivision iteration, while $\Delta_h'=\Delta_0'/2^h$ fast converges to 0 as $h$ increases; 
  in particular, for any fixed $\beta$ of Def. \ref{defepsbt} we have  $\Delta'_h\le 0.5\beta\Delta'_0$ 
already for $h=O(1)$.

(ii) The first $h$ iterations involve at most $ \frac{4^{h+1}}{3}$ e/i tests, that is, $O(1)$ tests if $h=O(1)$.
\end{proposition}
\begin{proof}
We only prove claim (ii).   Eqn. (\ref{eqmi})
implies that
$\sum_{i=0}^{h-1} m_i\le \sum_{i=0}^{h-1}
4^{i+1}=\frac{4^{h+1}}{3}$. 
\end{proof}
\begin{remark}\label{remk} 
Every root is defined up to within $\epsilon$, and 
so  at the $h$-th subdivision iteration we decide whether we should discard a square or call it suspect and keep subdividing it depending on the roots that lie  in  the $(\Delta_h'+\epsilon)$-neighborhood of $\mathbb D$. 
 Def. \ref{defepsbt}  for $\Delta:=\Delta_0'$ implies that   $\Delta'_h\epsilon\le \beta\Delta$ already for some $h=O(1)$, and so,  the outputs of all e/i tests, except for  $h=O(1)$ initial ones, depend only  on the roots  that lie in the $\beta \Delta$-neighborhood of $\mathbb D$.
\end{remark}



\begin{remark}\label{rerlrtf} {\em [Real and near-real root-finding.]}
In the case of Problem 0 of  Example \ref{exrlnrrl},  {\em generalize bisection root-finding} as follows:
 (i) recursively partition (subdivide) the real line segment defining  every suspect rectangle into two halves and accordingly subdivide the rectangle into two congruent
 rectangles; to each  of them apply e/i test, and then discard a rectangle if it contains no roots but call it suspect and  subdivide it otherwise, (ii) stop where a rectangle has diameter  at most $2\epsilon$, so that its center approximates within $\epsilon$ all roots lying in it. 
\end{remark}

\subsection{Subdivision iterations with soft e/i tests}\label{ssftei}
  
Instead of testing whether a suspect square $S$ contains some roots, Henrici in \cite{H74} tested whether the  minimal disc $D(c,\rho)$  covering $S$ does this. 
 Furthermore, in 
the presence of rounding errors no  e/i test can decide whether a root lies in or outside a disc if it actually lies on or very close to its boundary circle. We address both problems by fixing $\sigma>1$ and applying {\em soft e/i 
tests}, which are a special case for $\ell=1$ of the following more general  soft $\ell$-tests.\footnote{Besides 1-tests (that is, e/i tests) we only use $\ell$-tests for sub-discs of  discs containing at most $\ell$ roots and well-isolated from the external roots.}
 
\begin{definition}\label{defsft}                                                                                                                      
 For  $p(x)$ of (\ref{eqpoly}), 
 $\sigma>1$, and an integer $\ell$, $1\le \ell\le d$,
a $\sigma$-{\bf soft} $\ell$-{\bf test}, or just   $\ell$-{\bf test} for short (in particular an e/i test turns into a 1-test),
either outputs 1 and stops if it detects that $r_{d-\ell+1}\le \sigma$, that is, $\#(D(0,\sigma))\ge\ell$, or outputs 0 and stops if it detects that $r_{d-\ell+1}> 1$, that is,\footnote{Bounds $r_{d-\ell+1}\le \sigma$ and $r_{d-\ell+1}>1$ can hold simultaneously (cf. Fig. \ref{figsbd}b), but as soon as an $\ell$-test  verifies any of them, it stops without checking if the other bound also holds.}   $\#(D(0,1))< \ell$. An
$\ell$-test applied to the polynomial $t(y)$ of (\ref{eqshft}) is said to be 
an  $\ell$-{\em test for the disc} $D(c,\rho)$ and/or {\em the circle} $C(c,\rho)$ and also to be an $\ell$-{\em test}$_{c,\rho}$. 
\end{definition} 
  
\begin{observation}\label{obsgmcvrs}
Let $\gamma(\sigma)$
denote the maximal number of $\sigma$-covers at a fixed subdivision step
that share a complex point, e.g., a root. Then $\gamma(\sigma)\le 4$
if $\sigma<\sqrt 2$ and  $\gamma(\sigma)$ is in $O(1)$ if so is $\sigma$.  
\end{observation}

\begin{corollary}\label{cosgmz}
Cor. \ref{corsubsid} and the observation of Remark \ref{remk} still hold if e/i tests are replaced with $\sigma$-soft e/i tests for  $\sigma<\sqrt                                                                                                         2$, while a subdivision root-finder using $\sigma$-soft e/i  tests  involves   
$O(km)$  such tests
if $\sigma$ is a constant.
\end{corollary}

\section{Fast soft e/i tests}\label{sseitsts}
 
\subsection{Background: Newton's Inverse Ratio}\label{snir} 
 
 We reduce e/i tests of
 subdivision root-finders to  multipoint evaluation   of the ratio 
 \begin{equation}\label{eqratio}
{\rm NIR}(x)=\frac{p'(x)}{p(x)}=\sum_{j=1}^d\frac{1}{x-z_j},
\end{equation}
said to be
{\em Newton's inverse ratio}.
To prove the latter well-known equation, differentiate  
$p(x)=p_d\prod_{j=1}^d(x-z_j)$.

 The ops involved in our evaluation of  NIR$(x)$ are more numerous than the other ops involved in our  root-finders, even in the case where $p(x)=x^d$.

By following \cite{LV16}
we fix a complex 
 $\delta$, for $|\delta|$ small in context, 
  and approximate $p'(x)$ with the first order divided difference
$\frac{p(x)-p(x-\delta)}{\delta}$, 
 equal to $p'(y)$ for some $y\in [x-\delta,x]$,
by virtue of
 Taylor-Lagrange's formula. Hence
  $y\mapsto x$ and
  $p'(y)\mapsto p'(x)$
  as $\delta\mapsto 0$,
  and  
    
\begin{equation}\label{eqNRNIR} 
{\rm NIR}(x)\approx\frac{p(x)-p(x-\delta)}{p(x) \delta}=\frac{p'(y)}{p(x)}~{\rm for~some}~y\in [x-\delta,x].  
 \end{equation}

    
\subsection{Background: error detection and correction} \label{smrerrdtc}
 
A randomized root-finder  can lose some roots, although with a low probability, but we can
 detect such a loss at the end of  root-finding process, simply by observing that among the $m\le d$ roots lying in an input disc $D$
  only $m-w$ {\em tame}
  roots, say, $z_{w+1},z_{w+1},\dots,z_{m}$ have been closely 
approximated,\footnote{We call a complex $c$  a tame root for an error tolerance $\epsilon$ if $c$ lies in an isolated  disc $D(c,\epsilon)$, whose index 
$\#(D(c,\epsilon))$ we can compute fast by applying Thm. \ref{thpwrsm0} or the algorithm in  Sec.   \ref{secmprstris}.}  while $w>0$ {\em wild} roots,  say, $z_{1},z_{2},\dots,z_{w}$ remain at large. 
We can compute the coefficients of the factor $g(x):=\frac{p(x)}{f(x)}$ for $f(x):= \prod_{j=w+1}^m(x-z_j)$,
say,  by applying  evaluation-interpolation techniques, then apply to $g(x)$ a black box root-finder, and repeat this computation recursively
to approximate the $w$  wild roots.\footnote{By this recipe  we detect any output error at the very end of computations in any root-finder that can only lose roots but cannot wrongly generate them.  Our subdivision root-finders belong to this class, but with them we can detect the loss of a root even earlier -- whenever in a subdivision process applied to a square with $m$ roots we detect that  $m$ exceeds the sum of the indices of all suspect squares, that is, the overall number of roots contained in  them.} We 
readily  bound  the complexity and error probability
of this recursive process (although we will not use these estimates in the following): 
 
\begin{proposition}\label{thrtfrrdtccrc11}
Let the above  recipe for error  detection and  correction be applied to
 the output
of a  randomized root-finder 
performing  at the cost $\alpha$  
 (under a fixed model),  let it be correct
with a probability $1/\beta$ for $\beta>1$ and let  the  algorithm be applied recursively, each time  with error detection
at the cost $\delta$. Then in $v$ recursive applications, the output is certified
to be correct at the cost of less than $(\alpha+\delta)v$
with a probability at least
$1-1/\beta^v$. 
\end{proposition}

Explicit deflation can  blow up the coefficient length and 
destroy sparsity unless the degree $w$ is  sufficiently small, but 
 we avoid these mishaps  by combining a fixed black box polynomial root-finder with implicit 
deflation (cf. example (iii) in Sec. \ref{swhy}), that is,  by applying that root-finder  to the quotient polynomial $q(x)$ without computing its coefficients -- based on  the following expression,
implied by (\ref{eqratio}):
\begin{equation}\label{eqimpldfl} 
 \frac{q'(x)}{q(x)}=
\frac{p'(x)}{p(x)}-\frac{t'(x)}{t(x)}=\frac{p'(x)}{p(x)}-
\sum_{j=w+1}^{m}\frac{1}{x-z_j}.
\end{equation}
\begin{remark}\label{refmm}
One can accelerate low precision numerical approximation of the sum 
$\sum_{j=w+1}^{m}\frac{1}{x-z_j}$  simultaneously at many points $x$ by applying {\em Fast Multipole Method (FMM)}
by Greengard and Rokhlin \cite{GR87,CGR88,DGR96}
(see more in Sec. \ref{snopth2}).
\end{remark}

\subsection{Background: root-lifting}\label{sdnd}

Our e/i tests of Part II  involve $h$-lifting of roots:
  
  \begin{equation}\label{eqdndgnr}
 p_0(x):=\frac{p(x)}{p_d},~p_h(x):=\prod_{g=0}^{h-1}  p_0(\zeta_h^gx)=(-1)^d\prod_{j=1}^{d}(x-z_j^h),~ \zeta_h:=\exp\Big(\frac{2\pi\sqrt {-1}}{h}\Big),
\end{equation} 
 for $h=1,2,\dots$
 and $t_{2^k}(x)=p_k(x)$ for all $k$.  $h$-lifting maps roots to their $h$th powers; this  keeps the unit disc $D(0,1)$ and its boundary  circle $C(0,1)$ intact but 
 moves the external roots  away from the circle, thus {\em strengthening its isolation} from those roots.  
 
\begin{remark}\label{rednd}
For $h$ being a power of 2, the $h$-lifting generalizes the
 {\em DLG classical root-squaring iterations} of Dandelin 1826, Lobachevsky 1834, and Gr{\"a}ffe 1837 
(cf.
\cite{H59}):      
\begin{equation}\label{eqdnd} 
 p_0(x):=\frac{p(x)}{p_d},~p_{2^{i}}(x^2):=(-1)^ dp_{2^{i-1}}(x)
p_{2^{i-1}}(-x)=\prod_{j=1}^d(x-z_j^{2^i}),~i=1,2,\dots.  
\end{equation}
 Polynomial root-finding in the 19-th
 and well into the 20-th century relied on  iterations (\ref{eqdnd}) combined with Vieta's formulas, which link
 the polynomial coefficients $p_0,p_1,\dots,p_d$ to the elementary symmetric polynomials of the roots.  The combination enabled efficient approximation of the minimal and maximal distance of the roots 
 from the origin and,
 due to the map (\ref{eqshft}), from any complex point (see \cite{MZ01,P22b,PGLZ23}
 for details). The approach became obsolete because of severe problems of numerical stability of the DLG iterations, but
 they disappeared in the case of root-finders that work for black box polynomials. Extensive tests
  with standard test polynomials  have consistently  demonstrated efficiency 
 of the particular 
implementation  of this 
approach in \cite{PGLZ23}, but even the straightforward implementation of 
iterations (\ref{eqdnd})
promises to be highly efficient as well. The paper \cite{GPS23} explores direct root-squaring for NIR of a black box polynomial:
$${\rm NIR}_{p_{h+1}}(x) =  \frac{1}{2\sqrt x}\left({\rm NIR}_{p_h}(\sqrt{x}) {\rm NIR}_{p_{h}}(-\sqrt{x})\right),~x\neq 0,~~h=0,1,\dots.$$
\end{remark}  
  
  \subsection{Soft e/i test with root-lifting: generic algorithm}\label{seiIII}
 In  Sec. \ref{sei} we  elaborate upon the following generic soft e/i 
test applied to the unit disc $D(0,1)$.
  
\begin{algorithm}\label{algexclIIr} {\em A generic soft e/i 
test}.    
     
INPUT:   $\sigma>1$, an integer $d>0$, and  a $d$-th degree black box polynomial 
$p(x)$  given by an oracle for its evaluation. 

INITIALIZATION: Fix two   positive integers $h$ and $q$ and  $q$ points $x_1,\dots,x_q$  on the circle $C(0,1)$.  

  COMPUTATIONS:  By extending (\ref{eqNRNIR}) to $p_h(x)$ of  (\ref{eqdndgnr}) (see (\ref{eqNRNIRh}))
 approximate the values
\begin{equation}\label{eqvg'}  
  v_g:=|{\rm NIR}_{p_h}(x_g)|=\Big|\frac{p'_h(x_g)}{p_h(x_g)}\Big|,~g=1,\dots,q.
\end{equation}  
   If   $1+d/v_g\le \sigma^h$ for some $g\le q$, then certify that   \#$(D(0,\sigma))>0$ ($\sigma${\rm -soft inclusion}).
Otherwise, claim that \#$(D(0,1))=0$ ({\rm exclusion}).
\end{algorithm} 
 
\begin{remark}\label{reeil} 
(\ref{eqNRNIRh}) and (\ref{eqdndgnr}) together express $v_g$ for all $g$ through $2hq$ values of $p(x)$, by-palink to the values of $p_h(s)$ helps 
us o prove correctness  of the algorithm.
\end{remark}

 \subsection{Soft e/i test with root-lifting: correctness of inclusion claim}\label{seiIIIincl}

Eqn. (\ref{eqratio})
 implies two useful corollaries: 
\begin{corollary}\label{coshftrh}
Map
(\ref{eqshft})
 scales NIR by $
\rho$.
\end{corollary}
\begin{proof}
Eqn. (\ref{eqratio}) 
 implies that
NIR$(x-c)=$NIR$(x)$ and that NIR$(x/a)=a$NIR$(x)$.
\end{proof}
\begin{corollary}\label{coincl}
It holds that
 $r_d(c)\le d~\Big|\frac{p(c)}{p'(c)}\Big|$, that is,
 $\#\Big(D\Big(c,d~\Big|\frac{p(c)}{p'(c)}\Big|\Big)\Big)>0$, for any complex $c$.
\end{corollary}
\begin{proof}
Let $x=c$ in Eqn. (\ref{eqratio}).
If $|c-z_j|< \rho$ for all $j$, then $$\Big|\frac{p'(c)}{p(c)}\Big|=\Big|\sum_{j=1}^d\frac{1}{c-z_j}\Big|\le \sum_{j=1}^d\frac{1}{|c-z_j|}<\frac{d}{\rho},~
{\rm and~so}~\Big|\frac{p(c)}{p'(c)}\Big|>\frac{\rho}{d}.$$
\end{proof}
We  use Cor. \ref{coshftrh} later but next apply Cor. \ref{coincl} to $p_h$ for a sufficiently large $h$
 of order $\log(d)$ to prove correctness of the inclusion claim of Alg. \ref{algexclIIr}.

 \begin{lemma}\label{leincl}
 Let $\rho:=d\Big|\frac{p_h(x)}{p'_h(x)}\Big|$ and $\sigma:=(|x|+\rho )^{1/h}$
 for  a positive integer $h$. Then \#$(D(0,\sigma))>0$.
  \end{lemma}
  \begin{proof}
 Apply Cor. \ref{coincl}  
to the polynomial $p_h(x)$ and obtain that
the equation $\rho=d\Big|\frac{p_h(x_)}{p'_h(x)}\Big|$
implies that
the disc $D(x,\rho)$
contains a zero of $p_h(x)$. Hence
 the disc $D(0,|x|+\rho)$ contains a zero of $p_h(x)$ 
by virtue of triangle inequality. Therefore,
the disc $D(0,\sigma)$
 contains 
a zero of $p(x)$ for
 $\sigma=(|x|+\rho)^{1/h}$.
  \end{proof}
  
 Apply this lemma for $x=x_i$,  $|x_i|=1$, $\rho=d/v_i$, and $v_i=|{\rm NIR}_{p_h}(x_i)|$
 (cf. Alg. \ref{algexclIIr}). Deduce that \#$(D(0,\sigma))>0$ for $\sigma^h=1+d/v_i$,  that is, Alg. \ref{algexclIIr} is correct when it claims $\sigma$-soft inclusion.
   
\begin{observation}\label{exsftinc} {\rm  
$\sigma$-soft inclusion test by means of Alg. \ref{algexclIIr} is correct where}
 $$d\ge 2,~x_1=1,~\frac{1}{|{\rm NIR}(1)|}=4\sqrt d-\frac{1}{d},~
h=40\log_2(d),~{\rm and}~
\sigma<1.1.$$
\end{observation} 
\begin{proof} 
Indeed, in this case
$$1+\rho=1+\frac{d}{|{\rm NIR}(1)|}=4d\sqrt d=2^{2+1.5\log_2(d)}.$$  
Apply Lemma \ref{leincl} and obtain that \#$(D(0,\sigma))>0$ for
$\sigma=(1+\rho)^{1/h}=2^{\nu}~{\rm and}~
\nu=\frac{2+1.5\log_2(d)}{40\log_2(d)}$. 

Deduce that $\nu \le \frac{7}{80}$
because $\frac{2}{40\log_2(d)}\le 0.05$ for $d\ge 2$ and 
because $\frac{2+1.5\log_2(d)}{40\log_2(d)}<\frac{3}{80}$.

Hence $\sigma\le 2^{\frac{7}{80}}=1.062527 \dots<1.1$.
\end{proof}

\subsection{Soft e/i test with root-lifting: further analysis (outline)}\label{seiIIIexcl}

 We  prove {\bf (i) correctness of the exclusion claim}  of
 Alg. \ref{algexclIIr}  separately for two specifications, both strengthening isolation of the unit circle $C(0,1)$ from the roots by means of root-lifting (\ref{eqdndgnr}).

(a) In Sec. \ref{smdfd} 
  we choose any integer $q$  exceeding  $d$
  and then
prove
 correctness of   exclusion
claim,
 provided that
$|{{\rm NIR}_{p_h}(x_i)}                                                                                                                                                                                                                                                                                                                                                                                                                                                                                                                                                                                                                                                                                                                                                                                                                                                                                                                                                                                                                                                                                                                                                                                                                                                                                                                                                                                                                                                                                                                                                                                                                                                                                                                                                                                                                                                                                                                                                                                                                                                                                                                                                                                                                                                                                                                                                                                                                                                                                                                                                                                                                                                                                                                                                                                                                           |<1/(2\sqrt q)$ for all $x_i$ being $q$th roots of unity. This bound is compatible with the bound of Observation  \ref{exsftinc} because
$4\sqrt d-\frac{1}{d}>3\sqrt d$ for $d\ge 2$.
We also extend this 
e/i test
by allowing any $q>m$
(rather than $q>d$)
provided that  the unit disc contains $m$ roots 
 and
$|{{\rm NIR}_{p_h}(x_i)}                                                                                                                                                                                                                                                                                                                                                                                                                                                                                                                                                                                                                                                                                                                                                                                                                                                                                                                                                                                                                                                                                                                                                                                                                                                                                                                                                                                                                                                                                                                                                                                                                                                                                                                                                                                                                                                                                                                                                                                                                                                                                                                                                                                                                                                                                                                                                                                                                                                                                                                                                                                                                                                                                                                                                                                                                           |<1/(3\sqrt q)$ for all $x_i$ being $q$th roots of unity.

(b) In Sec. \ref{seirndalg} 
 we sample $q=O(\log(d))$ iid random values of $x_i$ under the uniform probability distribution on the circle $C(0,1)$ and then prove that the exclusion claim is correct whp under the Random Root Model  if 
 $|{{\rm NIR}_{p_h}(x_i)}|\le 1/d^a$ 
 for a sufficiently large constant $a$.
  
 {\bf (ii) Counting query locations.}
Eqn. (\ref{eqdndgnr}) reduces approximation of  $p_h(x)$ at $x=c$
  to approximation of $p(x)$
  at all  $h$th roots of $c$. We can fix an integer $h>1$  and apply (\ref{eqdndgnr})
   or  fix $h>1$ being a power of two  and apply (\ref{eqdndgnr}). In both cases Alg. \ref{algexclIIr} evaluates polynomial $p(x)$ either at $2jh$ points for $j\le q$ when it claims inclusion  or at $2qh$ points when it claims exclusion. We always choose   $h=O(\log(d))$ and then ignore a factor of $h$ in our $\tilde O$ estimates. 
   
In our 
soft e/i test based on  Alg. \ref{algexclIIr}  we evaluate polynomial $p(x)$ at $\tilde O(m)$
 points and  then
 combine these estimates with Cor. \ref{cosgmz}, and bound by
$\tilde O(bm^2)$ the                                              overall number of  evaluation queries in  the resulting root-finders. 

Under the {\em Random Root Model} of Sec. \ref{scmplests} the roots 
are reasonably well  isolated whp from the unit circle $C(0,1)$. We strengthen this isolation by applying  
$O(\log(d))$ root-lifting  steps of (\ref{eqdndgnr}) and 
 then whp  bound the overall number of  queries   by  $\tilde O(bm)$. 

{\bf (iii) Computational precision and bit-operation complexity.}    
  We deduce  the same bound $O(b\log(d))$ of \cite{LV16}
 on computational precision of our root-finders; our derivation 
   is much simpler
 because unlike  \cite{LV16} we only use
 the first order derivatives and
  divided differences.

 We combine this estimate and our count  of evaluation  queries in our root-finders with (i) Storjohann's algorithm for Las Vegas computation of the determinant \cite{S05}
to obtain record fast
solution of Problems 2 and $2^*$ of  matrix eigenvalue approximation (see alternative extension to eigenproblem in part 
(C) of Sec. \ref{sdrctacc}) and
(ii) the complexity estimates  of Guillaume Moroz  \cite{M21} for his fast  multipoint
 evaluation of a polynomial represented in monomial basis
to arrive at
 the near-optimal complexity bounds of 
our Main Theorems 1 -- 4 within a factor of $b$, to be removed in Part IV. 

 
\subsection{Cauchy-based soft e/i tests: outline}\label{scchei} 
 
 
Next we outline our alternative soft e/i test of Part III.  
  Let $\#(\mathcal D)$ denote the number 
of roots  lying in a complex domain $\mathcal D$ with  the boundary $\Gamma$ and
recall from \cite{A00} that  
   
\begin{equation}\label{eqchint}
 s_0=\#(\mathcal D)=\frac{1}{2\pi\sqrt {-1}}\int_{\Gamma}\frac{p'(x)}{p(x)}~dx,
  \end{equation} 
   that is, $\#(\mathcal D)$ is the average of NIR$(x)=\frac{p'(x)}{p(x)}$ on the boundary $\Gamma$.
   
 Having an approximation of integral  (\ref{eqchint}) within  less than 1/2, we
 can round it to a closest integer and obtain $s_0=\#(D(0,1))$. This root-counter defines  $\ell$-tests
 for all $\ell\le d$ and in particular an e/i test, for $\ell =1$.


In Part III we approximate
the integral  (\ref{eqchint}) with finite sums  of the  values of $\frac{x}{q}$NIR$(x)$ at  the scaled $q$th roots of unity.   
Namely, let 
$\zeta:=\zeta_q$  denote a primitive $q$th root of unity
(cf. (\ref{eqdndgnr}))
and devise randomized $\sigma$-soft e/i tests  as follows.

\begin{algorithm}\label{algeic} {\rm [Cauchy e/i test].}

(i) Fix a constant $\gamma>1$, say, $\gamma=2$, and choose a real $\rho$ at random in a fixed range in
 $(1,\sigma)$. 

(ii) Fix a sufficiently large  integer $q$ of order $m\log(d)$.

 (iii) Compute   approximation $s_{0,q}=\frac{\rho}{q}\sum_{g=0}^{q-1}\zeta^g\frac{p'(\rho\zeta^g)}{p(\rho\zeta^g)}$ to the integral (\ref{eqchint})
for the boundary $\Gamma=C(0,\rho)$.
 Do this by means of approximation of NIR$_{p(\rho x)}(x)=\frac{p'(\rho x)}{p(\rho x)}$ at the points $x=\zeta_q^g$ for $g=0,1,\dots,q-1$; by
applying (\ref{eqNRNIR}) reduce  that task to  the evaluation of $p(x)$ at $2q$ points $x$.
Compute an  integer $\bar  s_0$  closest to $s_{0,q}$. 

(iv) Unless  $\bar s_0=0$
certify inclusion.

Otherwise, repeat the test $v$ times for a fixed integer $v>0$ (at the cost of the 
evaluation of 
$p(x)$ at $O(mv\log(d))$ points). Unless inclusion is certified in these tests,  claim exclusion (correctly with a probability at least  $1-1/\gamma^v$).  
 \end{algorithm}
 

 \section{Faster subdivision with compression: outline}\label{saccIV} 
 
 In Part IV we  accelerate the algorithms of both Parts II and III by a factor of $b/\log(b)$  
 to ensure pairwise isolation of clusters of roots and eigenvalues and to prove Main Theorems 1 -- 4; we prove the first three of them in two ways -- by devising two distinct supporting root-finders based on Parts II and III, respectively.
  Next we outline that acceleration.

{\bf (i) Component tree and compact components.}
   For fixed integer $m$,  positive $\epsilon$, and $\sigma$,  $1<\sigma < \sqrt 2$, apply subdivision process
to Problem 1 and  at every subdivision step partition the union of  the suspect squares  into connected components.

\begin{figure}
\centering
\resizebox{!}{0.15\textheight}{
\begin{tikzpicture}
\draw (0, 0) circle (0.74);
\draw (0, 0) circle (1.42);
\draw (-2, 2) -- (2,2) -- (2, -2) -- (-2, -2) -- (-2, 2);
\draw (-0.5, 0.5) -- (0.5,0.5) -- (0.5, -0.5) -- (-0.5, -0.5) -- (-0.5, 0.5);
\draw (-1, 2) -- (-1, -2);
\draw (1, 2) -- (1, -2);
\draw (2, -1) -- (-2, -1);
\draw (2, 1) -- (-2, 1);
\draw (0, 2) -- (0, -2);
\draw (-2, 0) -- (2, 0);
\draw[red] (-0.1, 0.15) node [left]{$*$}; 
\draw[blue] (-0.1, -0.15) node [left]{$*$};
\draw[red] (-0.1, 0.15) node [right]{$*$}; 
\draw[blue] (-0.1, -0.15) node [right]{$*$}; 
\end{tikzpicture}}
\caption{Roots  (asterisks)  define compact components. A subdivision step halves their diameters  and  accordingly strengthens isolation of their minimal covering squares and discs.}\label{fig3}
\end{figure} 


Represent all  partitions 
 of components with a {\em component tree} whose
 at most $m$ leaves 
 represent components covered by isolated $\epsilon$-discs.  
The other (at most $m-1$) vertices of the tree  represent 
components-parents, each broken by a subdivision step into at least two components-children. The initial suspect square is the  root of the tree. The tree has at most $2m-2$ 
edges,
 connecting parents with their children and representing
 partition of the components but not
 showing modification of components between  partitions.

Every component contains $O(m)$
suspect squares whose  diameter is halved at
 every subdivision step. Hence in $O(b)$  subdivision iterations, for $b:=\log_2(2\Delta'/\epsilon)$ and  the diameter $\Delta'$ of the initial suspect square,  every component has diameter     below $\epsilon/2$ and hence lies in an $\epsilon$-disc. 

{\bf (ii) Faster subdivision  with  compression.} 
Call a component  $\mathcal C$ {\em compact} if all   
 (at most four) its squares share  a vertex; subdivision steps fast decrease  the diameter of  $\mathcal C$ but never decrease its distance from external roots
(see Fig. \ref{fig3}). 

Subdivision  becomes  inefficient  where a  component stays unbroken and  compact in many  subdivision steps (see Fig. \ref{fig4}), and we  replace them with  its {\em compression} with no loss of roots, of which the component very soon  becomes strongly isolated. 

The compression outputs either an  $\epsilon$-disc,  and then we stop, or a rigid disc $D$, that   is, minimal up to a  fixed constant factor, and then we  subdivide a minimal superscribing square $S$ of  $D$. Since  the disc $D$ is rigid,  the iterations   very soon output  more than one component
or a sub-domain of  an $\epsilon$-disc.

 \begin{figure}  
\centering
\begin{tikzpicture}
         \draw (0, 0) circle (1.5 cm) ;
         \draw (-2, 2) -- (2,2) -- (2, -2) -- (-2, -2) -- (-2, 2);
        
         \draw (-2, 0) -- (2, 0);
         \draw (0, 2) -- (0, -2);
        
         \draw (-2, -1) -- (0, -1);
         \draw (-1, 0) -- (-1, -2);
        
         \draw (-2, -1.5 ) -- (-1, -1.5);
         \draw (-1.5, -1) -- (-1.5, -2 );
        
         \draw (-1.25, -1.5) -- (-1.25, -2);
         \draw (-1.5, -1.75) -- (-1, -1.75);
        
         \draw (-1.375, -1.75) -- (-1.375, -2);
         \draw (-1.5, -1.875) -- (-1.25, -1.875);
        
         \draw[red] (-1.75, 2.90) node [left]{$*$};
         \draw[red] (-1.15, -1.65) node [left]{$*$};
         \draw[green] (-1.23, -1.80) node [left]{$*$};      
         
         \draw[blue] (2.225, 2) node [left]{$*$};
         \draw[blue] (0.225, 0) node [left]{$*$};
         \draw[blue] (-0.775, -1) node [left]{$*$};
         \draw[blue] (-1.275, -1.5) node [left]{$*$};
         \draw[blue] (-1.0, -1.750) node [left]{$*$};
\end{tikzpicture}
\caption{Five blue marked centers of suspect squares converge to a green-marked root cluster with linear rate; two red marked  Schr{\"o}der's iterates
of (\ref{eqSCHR}) converge with quadratic rate.}\label{fig4}
\end{figure}

 This way    we  solve  Problem 1 by applying at most
 $2m-2$  compression steps and in addition
$O(m)$ soft e/i tests,
as we deduced  in  Sec. \ref{scntei}.
 
{\bf (iii) Compression of components.} 
 We compress  a compact component $\mathcal C$ into a disc  $D=D(c,\rho)$ by computing: 
 \\
   (a) the number of roots in a component,\\
   (b) a point $c$
in or near  the minimal disc covering the root set of $\mathcal C$ (hereafter the {\em MCD} of $\mathcal C$) and  \\
 (c)   a reasonably  close upper bound $\rho$ on the 
  maximal distance    from $c$ to the roots  in  $\mathcal C$.
  
  In Secs. \ref{scmpralgd} and \ref{secmprstris} we readily perform steps (a)
  and (b) at a low computational cost because subdivision iterations fast strengthen the isolation of the component from external roots.   

We begin step (c)  with a range $[\rho_-,\rho_+]$ for
 the bound $\rho$,  $0<1/2^b=\rho_-<\rho_+$.
 We can 
 approximate $\rho$ by applying  bisection  to this range but 
 proceed faster by  applying bisection  of the logarithms, said to be {\em BoL},  to the range  $[\log_2(\rho_-),\log_2(\rho_+)]=[b,\log_2(\rho_+)]$, for $\log_2(\rho_+)\le 0$.
  
We apply  fast soft $m$-tests
 (see part (A) of Sec. \ref{ssbdrdei}) to
 define a pivot step of our accelerated bisection, with which we ultimately reduce compression to $O(m\log(b))$
 e/i tests.
Their computational cost dominates 
  the overall 
 cost of the solution of
Problems 0,0$^*$, 1, and 1$^*$, and  this enables us to complete  the proof of  our Main Theorems.
  

\begin{remark}\label{renndj}
Many ``natural" directions to proving our Main Theorems have led us into dead end. E.g., we tried to use 
a tree whose vertices represented components made up of {\em non-adjacent} suspect squares; then we readily bounded by $2m-1$ their overall number at all partition stages but could not  count them so readily
between these stages
 until we decided to count all suspect squares. See more examples of this kind in the next remark and Remark \ref{recchlft}.
\end{remark}


\begin{remark}\label{rebtcmpl}
By using some of the advanced techniques of Louis and Vempala in \cite{LV16}, we  readily bound the precision of computing in our root-finders. This enables us to deduce {\em a priori} estimates for the bit operation cost of our subdivision root-finders versus  inferior known a posteriori
estimates (see, e.g., \cite{PT13/16,BSSY18}). In particular, the bit operation cost estimates of the Main  Theorem of  \cite{BSSY18} contain
the term $$\log(1/{\rm Discr}(p(x))),$$ with ``Discr" standing for ``Discriminant", such 
that $$\log(1/{\rm Discr}(p(x/2^b)))=(d+1)db\log(1/{\rm Discr}(p(x))).$$
Subdivision iterations
 of the known root-finders, in particular,
 of those of  \cite{BSSY18}, 
scale the variable $x$ 
order of $m$ times  by factors of  
$2^b$. Hence  the bit-complexity specified in the  Main Theorem of \cite{BSSY18} has order of $mbd^2$, which {\bf exceeds} 
our bound and that of \cite{P95,P02} {\bf by a factor of} $d$ (see more comments in Sec. \ref{sblnscl}). 

\end{remark}


\section{Our novelties,  related works, implementation, and further progress}


\subsection{Our novelties (a short list)}
 Root-finding for polynomial given in the monomial basis, with their coefficients, has been studied
for four millennia and
very intensively in the last decades, but we had to introduce a number of nontrivial novel techniques (i) to handle efficiently black box polynomials and  (ii) to
accelerate the known subdivision algorithms to a near-optimal level.

Some of our techniques such as  randomization, numerical stabilization of root-squaring,   compression of a disc  without losing roots,  bisection of logarithms, approximation of root  radii and near-real roots, and even
 our simple error detection recipe can be  of independent interest.


\subsection{Related works}


From the pioneering paper \cite{LV16}   we adopt the model of computation,  extension of a black box polynomial root-finder to approximation of  matrix eigenvalues, and 
  the approach to estimation of the computational precision of root-finders, but otherwise we use distinct approach and techniques.                        
  
  Estimating the bit operation  complexity of our root-finders we use  \cite[Eqn. (9.4)]{S82} and 
\cite[Thm. 2]{M21}   
for multipoint polynomial evaluation.\footnote{The ingenious algorithm of 
  Moroz  supporting his \cite[Thm. 2]{M21} departs from the simple observation that  
 a polynomial $p(x-c)\mod(x-c)^{\ell-1}$ 
  approximates $p(x)$  
in the disc $D(c,\rho)$
within $O(\rho^{\ell+1})$
 for small $\rho$. To approximate $p(c)$ for any $c\in D(0,1)$  
  Moroz used $\tilde O((\ell+\tau)d)$  bit operations to compute
 a set of complex $c_i$
 and positive $\rho_i$
 for $i=1,\dots,K$  and $K=O(d/\ell)$ such that 
$D(0,1)\subseteq \cup_{i=1}^KD(c_i,\rho_i)$; then he ensured   approximation to $p(c)$  for every point $c\in D(0,1)$ within
$\epsilon:=3\sum_{_i=0}^d|p_i|/2^{\ell}$ by means of approximation  within $\epsilon/2$
of the lower degree polynomial
$p(x-c_i)\mod(x-c)^{\ell+1}$ for an appropriate $i$. At this point, application of  \cite[Alg. 5.3 and Thm. 3.9]{K98} 
yields new record   fast algorithm for multipoint  approximation of a polynomial given with its coefficients. Ignoring Kirinnis's paper \cite{K98}, Moroz
cited a  rediscovery of his result by Kobel and Sagraloff, made two decades later.}

Instead of applying high order Newton's iterations of \cite{LV16}, we accelerate classical subdivision root-finders.
The closest earlier variations
in \cite{R87,P00,BSSXY16,BSSY18}  use mach slower e/i tests, involving coefficients of $p$.
Two our fast novel 
e/i tests -- based on
 root-lifting and  approximation of Cauchy integrals, respectively,  work for a black box polynomial. 

 The algorithms of \cite{R87,P00}
dramatically   accelerate
 convergence of subdivision process by means of applying Newton's or Schr{\"o}der's iterations  (see Fig. \ref{fig4}). They converge with superlinear rate  to the convex  hull  of all roots lying in $C$.
The diameter of the  minimal covering disc, however, can be too large, and then
the papers \cite{R87,P00}
apply subdivision iterations again. 
The resulting root-finders 
recursively   intertwine subdivision iterations with  Newton's or 
   Schr{\"o}der's, and this 
    complicates the complexity estimates for a root-finder.
    
 \cite{R87,P00} have only estimated the number of ops but not of the bit operations involved.
\cite{BSSXY16,BSSY18}  estimated the bit operation complexity 
in terms of  the unknown   minimal distance between the unknown roots. Specified in terms of $b$ and $d$, the estimates  turned out to be inferior to ours
by a                                                                                                                                                                                                                                                                                                                                                                                                         factor of $d$ (see Remark \ref{rebtcmpl} and Sec. \ref{sblnscl}).

Instead of Newton's or 
   Schr{\"o}der's iterations, one can apply the algorithm of \cite{S82}. It first splits $p(x)$ into the product of two polynomials $f(x)$ and $g(x)$
whose zero sets lie in and outside the minimal disc covering the  component  $ C$, respectively. Then it recursively splits the polynomial $p(x)$ into factors  $f=f(x)$ and $g=g(x)$. The resulting  algorithm of  \cite{S82} is slower than ours 
by a                                                                                                                                                                                                                                                                                                                                                                                                         factor of $d$.
It has been accelerated to a near-optimal level in \cite{P95,P02}, although that root-finder is much involved and extensively uses the coefficients of $p$,
  like the root-finders of \cite{S82,R87,P00,BSSXY16,BSSY18}.

 \cite{M21,IM23} 
 presented efficient
 root-finders
 for  polynomials  having only well-conditioned zeros, but not for worst case polynomials $p(x)$. Furthermore, the algorithms
 of \cite{M21,IM23} extensively use the coefficients of $p(x)$.

  
     
\subsection{Implementation: State of the Art (briefly)}\label{sprpr}
  
 The near-optimal but  quite involved algorithm of \cite{P95,P02} has never been implemented, while  the slower algorithm of \cite{S82} has been implemented by Xavier Gourdon in \cite{G96} for the Magma and PARI/GP computer algebra systems.
By now it is superseded by implementations of various
 non-optimal complex  polynomial root-finders  such as MATLAB ``roots'', Maple ``solve'',  CCluster, and MPSolve; the latter one relies on Ehrlich's iterations and is user's current choice for approximation of all $d$ complex roots
 (see some competitive implementations of Newton's iterations in \cite{SS17,RSS24,MV24}).

   Current user's choice for real root-finding is the algorithm of \cite{KRS16}, although it has already been  superseded in \cite{IP21}
   according to extensive tests for standard test polynomials.  
  
We present  algorithms that  promise to compete for user's choice.
   Already  incorporation in  2020 of some of our techniques into the previous best implementation of subdivision
root-finder   \cite{IPY18} has enabled its noticeable  acceleration (see \cite{IP20}).  
The  more advanced publicly available implementation in \cite{IP22} ``relies on novel ideas and techniques" from \cite{P22} (see \cite[Sec. 1.2]{IP22}). According to the  extensive tests with standard test
polynomials, the algorithm has performed {\em dramatically faster}
than \cite{IPY18},
  has  at least competed  with  user's choice package MPSolve even for
approximation of all $d$ zeros of a polynomial  $p$ given with its 
coefficients,\footnote{The computational cost of subdivision root-finders
in a domain   $\mathbb D$  as well as the root-finder of 
\cite{P95,P02}
decreases at least proportionally  to the number of roots in $\mathbb D$,  while  Ehrlich's, Weierstrass's,  and a number of other functional iterations approximate the zeros of $p$ in  a domain $\mathbb D$ almost as slow as all its $d$ complex roots. Thus  already implementation in  \cite{IPY18} of subdivision iterations, slightly outperformed  MPSolve for root-finding in a domain $\mathbb D$ containing a small number
  of roots although  is  by far inferior for the approximation of all $d$  roots.}
 and  superseded MPSolve more and more significantly as the degree $d$ of $p$ grew large. 
 
 There is still a large room for further improvement because the root-finder of \cite{IP22}  inherits some earlier sub-algorithms from \cite{BSSXY16,IPY18},
 which are slower by order of magnitude  than our current ones (see the next subsection).
 
Application of  our near-optimal subdivision root-finder in Remark \ref{rerlrtf} to approximation of {\em real and near-real roots}
promises to  be at least competitive with user's current choice 
\cite{KRS16} and its  improvement in \cite{IP21}. It can be applied to a black  box polynomial, 
 counters the impact of input and rounding errors, and can be extended to approximation of real and near-real matrix eigenvalues.
                                                                                                                                  
\subsection{Further progress and promises}\label{sdrctacc}


  {\bf (A) Acceleration of e/i tests.} E/i tests are basic ingredient of  subdivision root-finders.
  \cite{IP20,IP22} used 
  the Cauchy-based e/i tests of our Part III and their heuristic acceleration.  We can devise and  test similar  heuristic acceleration of
the 
formally supported e/i tests of our Part II, based on root-lifting.
Namely, we can merely apply Alg.  \ref{algexclIIm} with an integer $q$ decreased dynamically from $m+1$, as long as the exclusion  claim remains empirically
correct.

\noindent      
{\bf (B) Acceleration of the  compression of a disc.} 
Assuming that
the unit disc $D(0,1)$
contains $m$ roots and 
is isolated from the $d-m$  external roots,
 we compress
that disc towards an $\epsilon$-disc or a rigid disc, in both cases losing no roots. Algorithmic extras 1 (for compression of a disc with a small number of roots) and 2 of Sec. \ref{sextr} can  help strengthen this approach.
 
Like us, \cite{IP22} reduces compression 
 to (a) non-costly root-counting in an isolated disc, (b) computing a point  $c$ lying  in or near the {\em MCD}, that is,
the minimal covering disc  \footnote{One may use the convex hull instead of the minimal covering disc.} of the root set in  $D(0,1)$, and (c)   approximation of the $m$th smallest root radius from $c$.
 
At stage (b)  \cite{IP22} computes a close approximation of the sum $s_1$ of the $m$ roots lying 
in the disc $D(0,1)$
and then outputs $c=\frac{s_1}{m}$. In Sec. \ref{scmpr}
we greatly accelerate 
this algorithms in the case where the radius of the output disc is small.

At stage (c)
\cite{IP22} applies 
the algorithm of our Sec. \ref{smvia1}, which  performs a constant  but fairly large number of e/i tests in a disc. 
  One may test this approach  versus our alternative of Sec. \ref{seiann}
 using a single e/i test for an annulus and versus the combination of
  the Las Vegas randomized algorithm of
Sec. \ref{mdls} with BoL of Sec.  \ref{sextrrk}.  
 
Finally, implementation of compression in \cite{IP20,IP22} is slowed down at the stage of verification based on Pellet's theorem, but we can  drop this verification because we readily detect highly unlikely output errors of our  root-finders at a dominated cost (see Sec. \ref{smrerrdtc}). 


\noindent {\bf (C) Fast approximation of the eigenvalues of a matrix.} 
 Our root-finders are essentially reduced to the evaluation of 
 NIR$(x)=\frac{p'(x)}{p(x)}$ at sufficiently many  points $x$, and
we can extend them to approximation of the  eigenvalues of a matrix or a matrix polynomial  by following
\cite{LV16} or based
 on the following alternative expressions:
 
$$\frac{p(x)}{p'(x)}=\frac{1}{{\rm trace}((xI-M)^{-1})}$$  
where $M$ is a matrix, while $p(x)=\det(xI-M)$ (cf. (\ref{eqratio}))
and  
$$\frac{p(x)}{p'(x)}=\frac{1}{{\rm  trace}(M^{-1}(x)M'(x))}$$  
where $M(x)$ is a matrix polynomial and $p(x)=\det(M(x))$ \cite[Eqn. (5)]{BN13}.\footnote{These expressions have been  obtained based on
 Jacobi’s formula
${\rm d}(\ln(\det(A)))={\rm trace}(A^{-1}{\rm d}A),$
 which links  differentials  of  a nonsingular  matrix $A$ and 
  of its determinant.}

These expressions  reduce 
approximation of the eigenvalues of a matrix or a matrix polynomial to its inversion. One can use the matrix inversion  algorithm and its  bit operation cost estimates of \cite{S15} towards alternative derivation of our Main Theorem 2 for any input matrix or matrix polynomial, but the  inversion can be greatly  accelerated in the highly important case where an input  matrix or matrix polynomial can be inverted fast, e.g., is data sparse.  
 

Bini and Noferini in \cite{BN13} proposed and explored such a reduction of approximation of all $d$ eigenvalues to 
root-finding by means of Ehrlich's iterations
and showed various benefits of this approach, but we can readily combine this reduction with  subdivision iterations 
instead of Ehrlich's and then obtain additional benefits:  our algorithms  (a)
are {\em  much faster} for approximation of a small number of eigenvalues that lie in {\em any fixed domain} and (b) {\em isolate the eigenvalues} instead of just approximating them.

\noindent {\bf (D) Acceleration of   real and near-real root-finding.} In Remark \ref{rerlrtf}
we adjust our near-optimal subdivision iterations to real and near-real root-finding,
which  counters the impact of input and rounding errors.

\noindent  {\bf (E) Acceleration by using fast multipoint polynomial evaluation.}
Mere incorporation of
fast algorithms of \cite{M21,IM23} should
greatly accelerate the performance  of \cite{IP20,IP22}
in the case of root-finding for  polynomials given with their
coefficients.

 
\subsection{Organization of the rest of the paper}\label{sorg}

In Parts II and III we devise and analyze our two e/i tests.
 We use Las Vegas randomization in both parts to bound computational precision and bit-operation cost.
In Part II the same randomization helps us devise fast e/i tests. In Part III we support that step with additional randomization. 

In  Part II we  further accelerate by a factor of $d$  e/i tests and overall root-finding  under the Random Root Model.

We make our presentation    in Part III largely independent of Part II by reusing  some definitions, results,  and arguments.

In Part IV we devise 
 and analyze our compression algorithm and incorporate it into the algorithms of Parts II and III --
 to  accelerate them by a factor of $b/\log(b)$ and to ensure pairwise isolation of the output roots and their clusters. 

In  Part V we cover 
 (i) low bounds on the bit-operation complexity  of polynomial   factorization, root-finding, and root isolation, (ii) acceleration with FMM 
(cf. Remark \ref{refmm}) and   initialization of root-finding by means of functional iterations,
 and (iii) some additional algorithms for root radii approximation.


\medskip
\medskip

{\Large \bf PART II: Classical subdivision root-finder with novel e/i tests  based on root-lifting} 
\medskip
\medskip

{\bf Organization of Part II:}  We devote the next section to background. In Sec. \ref{sei}   we cover our e/i tests. In Secs. \ref{sprcs11} and \ref{sblnrtfdg} we estimate computational precision of our root-finders and their bit operation cost, respectively.
In Sec. \ref{smxeig11} we extend 
our study of root-finding to approximation of matrix eigenvalues.


 \section{Background}\label{sbcgr}

\subsection{Basic definitions}\label{sdef}
 
 We extend the list of definitions used in the Introduction.
 
\begin{itemize}
  \item
 ${\rm NR}_{t}(x):=\frac{t(x)}{t'(x)}$ and  ${\rm NIR}_t(x):=\frac{t'(x)}{t(x)}$ are  {\em Newton's  ratio} and 
    {\em Newton's inverse ratio}, respectively, for a polynomial $t(x)$.
   We write NR$(x):={\rm NR}_{p}(x)$
and NIR$(x):={\rm NIR}_{p}(x)$ (cf. (\ref{eqratio})).
  
 \item
Define squares, discs, circles (circumferences),  and annuli on the complex plane: 
$S(c,\rho):=\{x:~|\Re(c-x)|\le 
 \rho$, $|\Im(c-x)|\le \rho\}$,~ 
$D(c,\rho):=\{x:~|x-c|\le \rho\},$
$C(c,\rho):=\{x:~|x-c|= \rho\}$,
~$A(c,\rho,\rho'):=\{x:~\rho\le |x-c|\le \rho'\}.$
 
 \item
   $\Delta(\mathbb S)$, 
    $MCD(\mathbb S)$, $X(\mathbb S)$, and $\#(\mathbb S)$ denote the diameter,   
    root set, minimal covering disc, and index (root set's cardinality) of a set $\mathbb S$ on the complex plane, respectively.

\item
A disc $D(c, \rho)$, 
a circle $C(c,\rho)$, or a square $S(c, \rho)$ is $\theta$-{\em isolated} for $\theta\ge 1$  if $X(D(c, \rho))=X(D(c, \theta\rho))$, $X(C(c,  \rho))=X(A(c,\rho/\theta,\rho\theta)$ or~$X(S(c,\rho))=X(S(c, \theta\rho)),$ respectively.  \\
  The {\em  isolation} of 
a disc $D(c, \rho)$, a circle $C(c, \rho)$, or a
 square $S(c, \rho)$ is the largest  upper bound
 on such a value $\theta$ 
  (see Fig. \ref{fig6}) and is  denoted
 $i(D(c,\rho))$, $i(C(c,\rho))$, or
 $i(S(c,\rho))$, respectively. \\
 We say ``{\em  isolated}" for ``$\theta$-isolated" if $\theta-1$
exceeds a positive constant.


\begin{figure}
\centering
\resizebox{!}{0.2\textheight}{
\begin{tikzpicture}
\draw (0, 0) circle (5cm);
\draw[red] (0, 0) circle (4cm);
\draw (0, 0) circle (3cm);

\draw [->] (0, 0) edge (-4, 3);
\draw [->] (0, 0) edge (0, 4);
\draw [->] (0, 0) edge (2.4, 1.8);

\draw node at (2, 1) {$\rho$};
\draw node at (0.5, 3.5) {$\rho\theta$};
\draw node at (-3.5, 3) {$\rho\theta^2$};

\draw node at (0, -0.2) {c};
\draw node at (1, -0.8) {*};
\draw node at (2.94, 0.585) {*};
\draw node at (-0.6, 0.15) {*};
\draw node at (-2.94, -0.585) {*}; 
\draw node at (3.89, 3.95) {*};
\draw node at (4.7, -2.75) {*};
\draw node at (-4.69, 1.8) {*};
\draw node at (-4.3, 3.43) {*};
\end{tikzpicture}}
\caption{The roots are marked by asterisks. The red circle $C(c,\rho\theta)$ and the   disc $D(c,\rho\theta)$ have isolation $\theta$.  The disc $D(c,\rho)$ has isolation
$\theta^2$.}
\label{fig6}
\end{figure}

 
\item
A disc $D$
has 
{\em rigidity} $\eta(\mathbb S)=\Delta(X(D))/\Delta(D)\le 1$. A disc is $\eta$-{\em rigid}
   for  any 
 $\eta\le \eta(D)$. \\
We say ``{\em  rigid}" for ``$\eta$-rigid"  
 if $\eta$ exceeds a positive constant.

 
 \item
$|u|:=||{\bf u}||_1:=\sum_{i=0}^{q-1}|u_i|$,  
$||{\bf u}||:=||{\bf u}||_2:=(\sum_{i=0}^{q-1}|u_i|^2)^{\frac{1}{2}}$,
 and $|u|_{\rm max}:=||{\bf u}||_{\infty}:=\max_{i=0}^{q-1}|u_i|$
denote the 1-, 2-, and $\infty$-norm, respectively,
of  a polynomial $u(x)=\sum_{i=0}^{q-1}u_i x^i$ and its coefficient vector ${\bf u}=(u_i)_{i=0}^{d-1}$, and then
(see \cite[Eqs. 
(2.2.5)--(2.2.7)]{GL13})  
\begin{equation}\label{eqnrms12}
||{\bf u}||_{\infty}\leq ||{\bf u}||_{2}\leq ||{\bf u}||_{1}\leq {\sqrt {q}}~||{\bf u}||_{2}\leq q||{\bf u}||_{\infty}.
\end{equation}

\item
   $r_1(c,t)=|y_1-c|,\dots,r_d(c,t)=|y_d-c|$ in 
   non-increasing order are the $d$ {\em root radii}, that is, the distances from a  complex point $c$ to the roots $y_1,\dots,y_d$ of a $d$th degree polynomial $t(x)$.
    Write $r^h_j(c,t):=(r _j(c,t))^h$, $r_j(c):=r_j(c,p)$.
\end{itemize}     

\begin{observation}\label{obscchy11}
 {\em Taylor's shift}, or translation, of the variable
$x\mapsto y:=x-c$ and
its {\em scaling} 
 combined (see (\ref{eqshft})) map the zeros $z_j$ of $p(x)$
into the zeros $y_j=\frac{z_j-c}{\rho}$
of $t(y)$, for $j=1,\dots,d$, and preserve the index 
and the isolation of  the disc $D(c,\rho)$.
\end{observation}

\subsection{Reverse polynomial}\label{srvrs}

Immediately observe the following properties
of the reverse polynomial  (\ref{eqpolyrev}).

\begin{proposition}\label{thratio}

{\rm (i)}  The roots of $p_{\rm rev}(x)$ are the reciprocals
of the roots of $p(x)$, and hence
\begin{equation}\label{eqrrpolyrev0}
r_j(0,p)r_{d+1-j}(0,p_{\rm rev})=1~{\rm for}~j=1,\dots, d.
\end{equation}

{\rm (ii)} The unit circle $C(0,1)$ has the same isolation for $p(x)$ and  $p_{\rm rev}(x).~~~~~~~~~~~~$

$${\rm (iii)}~p'_{\rm rev}(x)=dx^{d-1}p\Big(\frac{1}{x}\Big)-x^{d-2}p'\Big(\frac{1}{x}\Big)~{\rm for}~x\neq 0,~ p_d= p_{\rm rev}(0),~ p_{d-1}=p'_{\rm rev}(0).$$

${\rm (iv)}~~~~\frac{p_{\rm rev}(x)}{p'_{\rm rev}(x)}=\frac{1}{x}-\frac{{\rm NIR}(x)}{x^2},~x\neq 0,~{\rm NIR}(0)=\frac{p_{d-1}}{p_d}=-\sum_{j=1}^dz_j.~~~~~~~~~~~~~~~~~~~~$
\end{proposition}

Together with the first equation of (\ref{eqpolyrev}) this defines expressions for
 the reverse  polynomial $p_{\rm rev}(x)$, its derivative, and  NIR$p_{\rm rev}(x)=\frac{p_{\rm rev}(x)}{p'_{\rm rev}(x)}$ through $p(\frac{1}{x})$, $p'(\frac{1}{x})$,
    and NIR$(\frac{1}{x})$.  


\subsection{Root-lifting with divided differences}\label{srtsqr}

 Fix $\delta>0$, compute $p_h(y)$ for $y=x$ and $y=x-\delta$ (cf. (\ref{eqNRNIR})), 
   and at the cost of the evaluation of  $p(x)$ at  $2h$  points $x$ approximate  NIR$_{p_h}(x)=\frac{p'(x)}{p(x)}$ with a divided difference:
   \begin{equation}\label{eqNRNIRh} 
{\rm NIR}_{p_h}(x)\approx\frac{p_h(x)-p_h(x-\delta)}{p_h(x) \delta}=\frac{p_h'(y)}{p_h(x)}.  
 \end{equation} 
     

\begin{remark}\label{rertsq} 
Let $h=2^k$. The set of $2^k$-th roots of unity contains the sets of $2^i$-th roots of unity for $i=0,1,\dots,k-1$, and  one can recursively approximate the ratios NR$_{p_h}(x)$   and/or NIR$_{p_h}(x)$ for $h$ replaced by $2^i$ in (\ref{eqdndgnr}) and (\ref{eqNRNIRh}) and for $i=0,1,\dots$ 
\end{remark}

\begin{observation}\label{obrtsqind} 
(i) Root-lifting preserves the index  $\#(D(0,1))$ of the unit disc. (ii)  $h$-lifting  (achieved with 
 root-squaring for $h=2^k$ and a positive integer $k$) maps the values  $i(D(0,1))$, $\sigma(D(0,1))$,  $\eta(D(0,1))$, and $r_j(0,p)$ for all $j$, that is, the  isolation of the unit disc $D(0,1)$, its softness, its rigidity   and root radii,  to their $h$-th powers.
\end{observation} 
        
\subsection{The distance to a closest root and  certified inclusion}\label{sdstnc}
 
\cite[Sec. 1.10]{M07} lists various
 bounds on 
$r_d$ and $r_1$ in terms of the coefficients of $p$,  in particular \cite{F916}:
   
$r_1\le 2\max_{i=1}^d|p_{d-j}/p_d|^{1/j}$ unless $p_d=0$ 
 and $r_d\ge  2\min_{i=1}^d|p_{j}/p_0|^{1/j}$
 unless $p_0=0$.
 
\noindent For a complex $c$ and positive integers $j$ and $h$, Eqn. (\ref{eqratio}) and Observation \ref{obrtsqind} combined
imply that
\begin{equation}\label{eqrtrdbnds}
r_d(c)\le d~ |{\rm NR}(c)|,
~{\rm and}~
r^h_j(0,p)=r_j(0,p_h).
\end{equation}
 The roots lying far from a point $c$
little affect NIR$(c)$, and we extend bounds 
(\ref{eqrtrdbnds})
  as follows. 
 
\begin{proposition}\label{thrtr}
  If  $\#(D(c,1))=\#(D(c,\theta))=m$ for $\theta>1$, then 
   $|{\rm NIR}(c)|-\frac{d-m}{\theta-1}<\frac{m}{r_d(c)}.$ 
\end{proposition}
 \begin{proof}
 Assume that $|z_j|>\theta$ if and only if $j>m$
and then deduce from
 Eqn. (\ref{eqratio})
  that 
 $|{\rm NIR}(c)|\le 
 \sum_{j=1}^m\frac{1}{|c-z_j|}+
 \sum_{j=m+1}^d\frac{1}{|c-z_j|}< \frac{m}{r_d(c)}+\frac{d-m}{\theta-1}$. 
  \end{proof} 
 
{\em  Certify inclusion} into the disc $D(0,\sigma)$ if $\rho_d(0,p_h)\le \sigma^h$,
 $\sigma>1$ and  $h>0$, although
this bound is extremely poor for worst case polynomials
$p$:
\begin{proposition}\label{thrtr1}
The ratio $|{\rm NR}(0)|$ 
 is infinite for  $p(x)=x^d-u^d$ and $u\neq 0$, while  $$r_d(c,p)=r_1(c,p)=|u|.$$
 \end{proposition}
 \begin{proof} 
Notice that 
 $p'(0)=0\neq p(0)=-u^d$, $p(x)=\prod_{j=0}^{d-1}(x-u\zeta_d^j)$ for $\zeta_d$ of (\ref{eqdndgnr}).
 \end{proof}

Notice that
$\frac{1}{r_d(c)}\le |{\rm NIR}(c)|=\frac{1}{d}\Big|\sum_{j=1}^d \frac{1}{c-z_j}\Big|$
by  virtue of 
(\ref{eqratio}), and so 
 approximation  
to the root radius $r_d(c,p)$ is poor if and only if severe cancellation occurs in the summation of the $d$ roots. It seems that such a cancellation only occurs for a very narrow class of polynomials $p$, and if so, then
 the above non-costly  test certifies inclusion for a very large class of polynomials $p$.\footnote{Shifts of the variable $p(x)\leftarrow t(x)=p(x-c)$
for random $c$ should further enlarge this class.} 


\subsection{Subdivision root-finding involves  exclusion}\label{ssbde} 

It is not enough for us to have non-costly certification of inclusion because
exclusion  occurs in more than 25\% of all e/i tests in a subdivision process: 
  \begin{proposition}\label{thdscrd}
  Let $\sigma_i$ suspect squares enter the
   $i$-th subdivision step
   for $i=1,\dots,t+1$. Let it stop at 
   $f_i$ of them, discard $d_i$   and further  subdivide $\sigma_i-d_i-f_i$
 of them.  Write  $\sigma_1:=4$, $\sigma_{t+1}:=0$,
  $\Sigma:=\sum_{i=1}^t\sigma_i$,
   $D:=\sum_{i=1}^td_i$, and
   $F:=\sum_{i=1}^tf_i$.
   Then $3\Sigma=4(D+F-1)$ and $D>\Sigma/4$.
   \end{proposition}
    \begin{proof}
   Notice that 
   $\sigma_{i+1}=4(\sigma_i-d_i-f_i)~{\rm for}~i=1,2,\dots,t+1$,
   sum these equations for all $i$, substitute $\sigma_1=4$ and $\sigma_{t+1}=0$,
  and obtain  $3\Sigma+4\ge 4D+4F$. 
Represent subdivision
   process by a tree with  at most $F\le m$ leaves  whose other nodes are components made up of suspect squares, to obtain  pessimistic lower bounds $\Sigma+1\ge 2m\ge 2F$ and hence $D\ge \Sigma/4 +1/2$. 
    \end{proof}

 

\section{Our e/i tests}\label{sei} 


\subsection{Overview}\label{seiovrv} 


 Our Alg. \ref{algexclII}   of Sec. \ref{sbscei} and Alg. \ref{algexclIIm} of Sec. \ref{smdfd} are
deterministic  specifications of Alg. \ref{algexclIIr}. 
By virtue of Lemma \ref{leincl} all of them are correct  when they claim inclusion, but we will prove that they are  
 also  correct when output   exclusion.
  Namely, Algs. \ref{algexclII} and 
  \ref{algexclIIm}
   evaluate ${\rm NIR}(x)$ at $hq$ 
equally spaced points of the unit circle $C(0,1)$
for $h=O(\log(d))$ and certify exclusion if   at all these points
 $|{\rm NIR}_{p_h}(x)|\le \frac{1}{2\sqrt q}$ for $q>d$ or if  $|{\rm NIR}_{p_h}(x)|\le \frac{1}{3\sqrt q}$ for $q>m$, respectively, and otherwise certify inclusion.
 Our Example \ref{ex1} shows  that we cannot  devise such deterministic algorithms for $q<m$, but our Las Vegas 
e/i test evaluates polynomial $p(x)$ at $O(\log(d))$ points $x$ overall
under the Random Root Model.  
        

\subsection{Basic deterministic  e/i test}\label{sbscei} 


To simplify exposition 
assume 
in this subsection 
that we compute rather than approximate the values $p'(x)$, NR$_{p_h}(x)$, and NIR$_{p_h}(x)$  and that $r_d(0,p)\le \sigma$. In the next  specification of Alg. \ref{algexclIIr}  write $x_i:=\zeta^{i-1}$ for 
$i=1,\dots,q$, $\zeta$ of  (\ref{eqdndgnr}), and  $q>d$.
 
  \begin{algorithm}\label{algexclII} {\em Basic deterministic  e/i test}.   
     
INPUT: a black box polynomial 
$p(x)$ of a degree $d$
given by an oracle for its evaluation
 and $\sigma$,  $1<\sigma<\sqrt 2$.
 
 INITIALIZATION:  Fix two integers $q>d$ and $h:=\lceil\log_{\sigma}(1+2d\sqrt q)\rceil$, such that $\sigma^h\le 1+2d\sqrt q$. 
 
  COMPUTATIONS: For  $g=0,1,\dots,q-1$ recursively compute  the values
\begin{equation}\label{eqvg}  
  v_g:=\Big|\frac{p'_h(\zeta^g)}{p_h(\zeta^g)}\Big|.
\end{equation}    
 Stop and certify $\sigma$-soft inclusion  
 \#$(D(0,\sigma))>0$ if   $1+d/v_g\le \sigma^h$ for an integer $g<q$.
Otherwise stop and certify exclusion, that is, \#$(D(0,1))=0$.
\end{algorithm} 
    

          
\noindent   
 Alg. \ref{algexclII} outputs  $\sigma$-soft inclusion  correctly by virtue of Lemma \ref{leincl}.
 
Next we will prove that it also {\em correctly claims exclusion}.
 
      
\begin{lemma}\label{letl} 
A polynomial $p$ has no roots in the unit disc 
$D(0,1)$ if $|p(\gamma)|>2 \max_{w\in D(0,1)}|p'(w)|$ for a complex $\gamma\in C(0,1)$.         
\end{lemma} 
\begin{proof}
  By virtue of Taylor-Lagrange's theorem  
 $p(x)=p(\gamma)+(x-\gamma)p'(w)$, for any pair of $x\in D(0,1)$ and $\gamma\in C(0,1)$
 and for some $w\in D(0,1)$. Hence $|x-\gamma|\le 2$, and so $|p(x)\ge |p(\gamma)|-2 ~|p'(w)|.$
Therefore, $|p(x)|>0$ for $\gamma$ satisfying the assumptions of the lemma and for any $x\in D(0,1)$.
\end{proof}      

\begin{proposition}\label{thvrs} 
  $\#(D(0,1))=0$ for a black box polynomial $p(x)$ of a degree $d$ if 
  \begin{equation}\label{eqNIR}
\frac{1}{v}>2\sqrt {q}~{\rm for}~q>d,~v:=\max_{g=0}^{q-1} v_g,~{\rm and}~v_g~{\rm of~(\ref{eqvg})}.
\end{equation}
\end{proposition}
 \begin{proof} 
Bounds (\ref{eqvg}) and  (\ref{eqNIR}) together imply that  $|p'(\zeta^{g})|/v\le |p(\zeta^{g})|$
 for $g=0,\dots,q-1$ and hence
$ 2\sqrt {q} ||(p'(\zeta^{g}))_{g=0}^{q-1}||_{\infty}<  \widehat P~{\rm for}~\widehat P:=
||(p(\zeta^{g}))_{g=0}^{q-1}||_{\infty}.
$
Combine this bound with (\ref{eqnrms12}) to obtain
\begin{equation}\label{eqp'Vp}
2 ||(p'(\zeta^{g}))_{g=0}^{q-1}||_2<\widehat P.
\end{equation}
 Pad the   
 coefficient vector of the polynomial $p'(x)$  with $q-d-1$ initial coordinates 0 to arrive at 
   $q$-dimensional vector
   ${\bf p'}$. Then  $(p'(\zeta^{g}))_{g=0}^{q-1}=F{\bf p'}$ for the $q\times q$ matrix 
$F:=(\zeta^{ig})_{i,g=0}^{q-1}$   of discrete Fourier transform.
Substitute $F{\bf p'}$ for $(p'(\zeta^{g}))_{g=0}^{q-1}$ in   (\ref{eqp'Vp}) and obtain
 $2||F{\bf p'}||_2< \widehat P$.  
Hence  
$2\sqrt q||{\bf p'}||_2< \widehat P$ because
the
 matrix $\frac{1}{\sqrt q}F$  is unitary 
 (cf. \cite{GL13,P01}).
 It follows (cf. (\ref{eqnrms12}))
 that
$2||{\bf p'}||_1<\widehat P.$
 Since $\max_{x\in D(0,1)}|p'(x)|\le ||{\bf p'}||_1$, deduce that 
$2\max_{x\in D(0,1)}|p'(x)|<\widehat P=||(p(\zeta^{g}))_{g=0}^{q-1}||_{\infty}$.  Complete the proof 
by combining this bound with  Lemma \ref{letl}.   
 \end{proof} 
   

If Alg. \ref{algexclII} 
outputs exclusion, then $1+\frac{d}{v_g}> \sigma^h$ for $g=0,1,\dots, q-1$. Hence
$\frac{1}{v_g}>2\sqrt q$
 because
$\sigma^h> 1+2d\sqrt q$ according to initialization 
of Alg. \ref{algexclII}. 
Therefore,
the assumptions of  Prop. \ref{thvrs} hold for $p_h(x)$ replacing $p(x)$, and hence
the disc $D(0,1)$ contains no zeros of $p_h(x)$, but then 
it contains no zeros of $p(x)$
as well (see Observation \ref{obrtsqind} (iii)), and so  exclusion is correct. 

Readily verify the following estimates.

\begin{proposition}\label{theitst}
 Alg. \ref{algexclII}  is a $\sigma$-soft e/i test,  certifying  inclusion at the cost of evaluation of polynomial $p(x)$ at $2hj$
   points $x$ or  exclusion at  the cost of evaluation  at $2hq$    points $x$ for  three integers $h$, $j$, and $q$  such that  $1\le j\le q$, $q>m$, and $h:=\lceil\log_{\sigma}(1+2d\sqrt q)\rceil$,  $h=O(\log(d))$ for $\sigma-1$
 exceeding a positive constant. 
\end{proposition} 
\subsection{A faster deterministic  e/i test}\label{smdfd} 

Consider Problem 1$^*_m$
for a $\theta$-isolated disc $D$. Deduce from  equation (\ref{eqshft}) 
that the impact of the $d-m$ external roots on the values of NIR$(x)$ for $x\in D$
is negligible if  $\theta$ is sufficiently large and then 
modify  Alg. \ref{algexclII} letting it  evaluate NIR  just at $q=m+1$ rather than  $d+1$ points and still obtain correct output. We prove this  
 already where
$\log(\frac{1}{\theta})=O(\log(d))$; if   the unit disc $D(0,1)$ was initially just isolated, then a $O(\log(d))$-lifting ensures  desired stronger isolation.  

Next we supply relevant estimates.

  \begin{algorithm}\label{algexclIIm} {\em A faster deterministic e/i test}.   

INPUT: a black box polynomial $p(x)$ of a degree $d$, two integers $m$ and $q>m$, and  real
$\sigma$ and $\theta>1$  such that
             $1<\sigma<\sqrt 2$, $$|z_j|\le 1~{\rm for}~ j\le m,~ |z_j|>\theta~{\rm for}~j> m,$$  
$$\frac{d-m}{\theta^h-1}\le \frac{1}{6\sqrt q}~{\rm for}~h:=\lceil\log_{\sigma}(1+6m\sqrt q)\rceil~{\rm such~that}~\sigma^h\ge 1+6m\sqrt q.$$
 
  COMPUTATIONS: Recursively compute the values $v_g$ of (\ref{eqvg})
    for $g=0,1,\dots,q-1$. 
Certify $\sigma$-soft inclusion, that is, 
   \#$(D(0,\sigma))>0$, if   $v_g\ge\frac{1}{3\sqrt q}$ for
 some integer $g<q$.
   Otherwise certify exclusion, \#$(D(0,1))=0$.
\end{algorithm} 
 
\noindent Prop. \ref{thrtr}, for $p_h$ replacing $p$, and the assumed bound $\frac{d-m}{\theta^h-1}\le \frac{1}{6\sqrt q}$   together imply
$$\frac{m}{r_d(\zeta^g,p_h)}\ge v_g- \frac{d-m}{\theta^h-1}\ge v_g-\frac{1}{6\sqrt q}~{\rm for~all}~g.$$
 Alg. \ref{algexclIIm} claims inclusion if $v_g\ge\frac{1}{3\sqrt q}$ for some $g$, but then 
 $$\frac{m}{r_d(\zeta^g,p_h)}\ge \frac{1}{6\sqrt q},~r_d(\zeta^g,p_h)\le 6m\sqrt q,$$
 and~hence~
$$r_d(0,p_h)\le 1+6m\sqrt q~\le \sigma^h~{\rm for}~|\zeta|=1.$$  Now
Observation \ref{obrtsqind} (iii) implies that  {\em Alg. \ref{algexclIIm}  claims inclusion 
   correctly}. 
   
Next we prove  that it also {\em claims exclusion correctly}.


 \begin{proposition}\label{thvrsm1} 
Let a polynomial $p(x)$ of a degree $d$ have exactly $m$ roots  in the disc $D(0,\theta)$   for 
$$\theta^h\ge 1+6(d-m)\sqrt q,~h:=\lceil\log_{\sigma}(1+6m\sqrt q)\rceil,~{\rm  and~an~integer}~q>m$$
  and let   $v< \frac{1}{3\sqrt q}$ for
  $v$ of~(\ref{eqNIR}).
  Then $\#(D(0,1))=0$.
\end{proposition}
\begin{proof}                 
 Let $p(x)$ and its factor $f(x)$
 of degree $m$  share the sets of  their zeros in  $D(0,\theta)$. Then 
(\ref{eqratio}) implies that
$\Big|\frac{f'(x)}{f(x)}\Big|<\Big|\frac{p'(x)}{p(x)}\Big|+
\frac{d-m}{\theta-1}~{\rm for}~|x|=1,$
and so $\max_{g=0}^{q-1}\Big|\frac{f'(\zeta_q^g)}{f(\zeta_q^g)}\Big|< v+
\frac{d-m}{\theta-1}$. 
Recall that $\frac{d-m}{\theta-1}\le \frac{1}{6\sqrt {q}}$  and  $v< \frac{1}{3\sqrt q}$ by assumptions. 

Combine these bounds to obtain
 $\max_{g=0}^{q-1}\Big|\frac{f'(\zeta_{q}^g)}{f(\zeta_{q}^g)}\Big|< \frac{1}{2\sqrt {q}}$.
To complete the proof, apply Prop. \ref{thvrs}
for $f(x)$ replacing $p(x)$ and for $m$ replacing $d$.
 \end{proof} 
 
    
Readily extend Prop. \ref{theitst} 
 to  Alg. \ref{algexclIIm}
by decreasing its  bounds on  $q$ and $h$  
provided that
 $\#(D(0,\theta))=m$ for 
$\theta\ge 1+6(d-m)\sqrt q$. 

\begin{proposition}\label{theitstm}
Alg. \ref{algexclIIm}  is a $\sigma$-soft e/i test: it certifies either  inclusion  or  exclusion at the cost of evaluation 
 of $p(x)$ at $2hj$ or $2hq$ points $x$, respectively, 
 for  three integers $h$, $j$, and $q$  such that  $1\le j\le q$, $q>m:=\#(D(0,\theta)),~\sigma^h>1+6m\sqrt q$, 
$\theta\ge 1+6(d-m)\sqrt q$, and $h:=\lceil\log_{\sigma}(1+6m\sqrt q)\rceil$, so that  $h=O(\log(m))$ if $\sigma-1$
 exceeds a positive constant, and in this case we only need $O(m\log(m)$ queries of the oracle to perform an e/i test. 
\end{proposition}


\subsection{A randomized e/i test}\label{seirndalg} 
 
{\bf 7.4.1. Introductory comments.} Would Algs.
 \ref{algexclII} and \ref{algexclIIm} stay correct if we apply them 
 at fewer evaluation points? They have consistently stayed correct for $q$ of order $\log^2(d)$ in extensive numerical tests of \cite{IP20,IP22} applied to $p(x)$  
 rather than to $p_i(x)$ for $i>0$. The tests
 reported exclusion 
 where the values $|s_{h,q}|$ were reasonably small  for 
$s_{h,q}: =\frac{1}{q}\sum_{g=0}^{q-1}\zeta^{(h+1)g}~\frac{p'(\zeta^g)}{p(\zeta^g)},~{\rm for}~h=0,1,2$ (cf. Eqn. \ref{equ7.12.2} in Part III), 
and reported inclusion
otherwise. Clearly,
$|s_{h,q}|\le \delta$ for all $h$ if $|p(\zeta^g)|\le \delta/q$ for $g=0,1,\dots,q-1$,
while inclusion is certified unless $|p(\zeta^g)|$ is small for all $g$. 
 
The test results of \cite{IP20,IP22}, however, cannot be extended to 
 worst case input 
polynomials  $p(x)$
already for $q=m-1$.


 \begin{example}\label{ex1}
Let  $q=m-1$, $p(x)=(1-ux^{d-m})(x^m -mx+u)$,  $u\approx 0$ but $u\neq 0$. (i) First let $m=d$. Then $p(x)=(1-u)(x^d-dx+u)$, $p'(x)=(1-u)d(x^{d-1}-1)$ and NIR$(x)=0$ for $x$ being any $(d-1)$-st roots of unity. Then our e/i test for $q=d-1$ would imply that $\#(D(0,1))=0$,
 while $p(y)=0$ for $u\approx 0$ and some $y\approx 0$. (ii) For $0<m<d$
notice that $p'(x)=uf(x)+m(1-ux^{d-m})(x^{m-1}-1)$ where
$f(x)=(d-m)
x^{d-m-1}(x^m -mx+u)$
and so $uf(x)\approx 0$
while $1-ux^{d-m}\approx  1$ for $|x|\le 1$ and $u\approx 0$. Therefore,
$p'(x)\approx 0$ where 
$x$ is any $(m-1)$-st root of unity, and then  again extension of our e/i test to $q=m-1$ would imply that $\#(D(0,1))=0$,
 while $p(y)=0$ for $u\approx 0$ and some $y\approx 0$.
 \end{example}
 
 Clearly, \cite{IP20,IP22} have never encountered worst case inputs in their extensive tests.  
  Next, in line with this observation,
   we  prove 
that our e/i tests, 
which evaluate $p(x)$  at 
$O(\log(d)$ points, are correct  whp 
  under the Random Root Model. 

{\bf 7.4.2. Our algorithm and correctness proof.} We apply
 Alg. \ref{algexclIIr} 
  for $q=1$
 and evaluate NIR$_{p_h}(x)$ at a single random point  $x_1\in C(0,1)$, by means of evaluation of polynomial $p(x)$ at $2h$ 
   points $x$ for $h=O(\log(d))$. We refer to this algorithm as {\bf Alg. \ref{algexclIIr}(rrm).}
 The algorithm is correct 
 by virtue of Lemma
  \ref{leincl}
  where it claims 
   soft inclusion. 
 Next  we prove that it is correct  whp 
  under the Random Root Model where it outputs exclusion. 

  \begin{lemma}\label{thrndrts}
  Fix any constant $\beta>1$ and a sufficiently large integer $h$ of order $\log(d)$. Then under the Random Root Model, 
Alg. \ref{algexclIIr}(rrm) 
is correct with a probability at least 
 $1-P$, where 
\begin{equation}\label{eqerrprb0}  
 P\le  \frac{\gamma d}{\pi R^2}~{\rm for}~\gamma:=\pi \max\Big\{\frac{\beta^2}{\sigma^2},O\Big(\frac{1}{\beta^{2h}}\Big)\Big\}.
 \end{equation}
  \end{lemma}
   
 \begin{proof} 
 By virtue of Lemma \ref{leincl} {\em Alg. \ref{algexclIIr}(rrm) outputs  $\sigma$-soft inclusion  correctly} and, therefore, can be incorrect only where it outputs exclusion for a disc $D(0,1)$ containing a root.  
 
  Let this root be $z_1$, say.  
Then fix the values of the random variables $z_2,\dots,z_d$ and  $x_1\in C(0,1)$ and
 let $\gamma$ denote the area (Lebesgue's measure) of  the set $\mathbb S$ of the values of the variable $z_1$ for which
\begin{equation}\label{eqNIRsgmh}
v_1=|{\rm NIR}_{t_h}(x_1)|< \frac{d}{\sigma^h-1}.
\end{equation} 
Then  
Alg. \ref{algexclIIr}(rrm)
is incorrect with  a probability at most
\begin{equation}\label{eqerrprb}
 P=\frac{\gamma d}{\pi R^2}
 \end{equation} 
 where the factor of $d$ enables us to cover the cases where
 the value of 
 any of the $d$ iid random variables $z_1,\dots,z_d$, and not necessarily $z_1$, lies in the disc $D(0,1)$.

Fix two values  $\bar z,z\in\mathbb S$ such that $$\bar z={\rm argmax}_{y\in \mathbb S}~|y|~{\rm  and}~z={\rm arg max}_{y\in \mathbb S}~|y-\bar z|$$
 and notice that
 \begin{equation}\label{eqgmmz}
  \gamma\le \pi \min 
 \{|\bar z|^2|,|\bar z-z|^2\}. 
\end{equation}   
   
To estimate $|\bar z-z|$,
first write $y:=z^h\in D(0,1),~\bar y:=\bar z^h\in D(0,1)$ and prove that
\begin{equation}\label{eqdltnh} 
 |\Delta|=|\bar y-y|<\nu_h~{\rm for}~ \nu_h:=\frac{8d}{\sigma^h-1}.
\end{equation}

Write NIR and NIR$'$ to denote the values NIR$_{t_h}(x_1)$ for $z_1=z$ and 
 $z_1=\bar z$, respectively, recall (see (\ref{eqNIRsgmh})) that
$\max\{|{\rm NIR}|,|{\rm NIR'}|\}< \frac{d}{\sigma^h-1}$, and 
obtain
\begin{equation}\label{eqdltth} 
|\Delta_{t_h}|< \frac{2d}{\sigma^h-1}~{\rm for}~\Delta_{t_h}:={\rm NIR}'-{\rm NIR}. 
\end{equation}
Now let $y_1:=z_1^h,  \dots,y_d=z_d^h$ denote the $d$ zeros of $t_h(x)$
and deduce from (\ref{eqratio})
that
$$|{\rm NIR}|=\Big|\frac{1}{x_1-y}+S\Big|~{\rm  and}~|{\rm NIR}'|=\Big|\frac{1}{x_1-\bar y}+S\Big|~{\rm for}~ S:=\Big|\sum_{j=2}^d\frac{1}{x_1-y_j}\Big|.$$ 
   Hence  
\begin{equation}\label{eqNIRxy'}  
|\Delta_{t_h}|=|{\rm NIR}-{\rm NIR}'|
  =\Big|\frac{1}{x_1-y}-\frac{1}{x_1- \bar y}\Big|=\frac{|\Delta|}{|(x_1-y)(x_1-\bar y)|}.
    \end{equation}  
  
  Recall that $x_1\in C(0,1)$, while $y,\bar  y\in D(0,1)$, and so
$\max\{|x_1-y|,|x_1-\bar y|\}\le 2.$

 Combine this inequality with  (\ref{eqNIRxy'}) 
 and deduce that
 $|\Delta_{t_h}|\ge 
  |\Delta|/4.$
  
  Combine the latter
  bound with (\ref{eqNIRxy'})
  and obtain  (\ref{eqdltnh}).
  
  Now fix a large integer $g$, substitute $\bar y=\bar z^h$ and $y=z^h$
  for $h=g$ and $h=g+1$, 
  and obtain
  $$\lambda_{g+1}:=z^{g+1}-\bar z^{g+1},~
 \lambda_{g}:= z^g-\bar z^g,~{\rm and~so}~z^{g+1}=z\bar z^g+z\lambda_g,~
 |\lambda_h|\le \nu_h~{\rm for}~h=g,g+1.$$
  Therefore, $\lambda_{g+1}+\bar z^{g+1}=\lambda_{g}z+\bar z^{g}z$, and so
$$z-\bar z=\frac{\lambda_{g+1}-z\lambda_g}{\bar z^g},~
|\bar z-z|\le \frac{|\lambda_{g+1}|+|z|\cdot|\lambda_g|}{|\bar z|^g}<\frac{\nu_{g+1}+|z|\nu_g}{|\bar z|^g}=O\Big(\frac{d}{(\sigma|\bar z|)^g}\Big).$$
Hence
$|z-\bar z|=O(1/\beta^g)$ for 
a fixed $\beta>1$. Combine this bound with
(\ref{eqerrprb}) and (\ref{eqgmmz}), choosing sufficiently large $g$ of order $\log (d)$ unless $\sigma|\bar z|< \beta$.
 
 Combine this bound with (\ref{eqerrprb}) and (\ref{eqgmmz}), recall that $h=g$ or $h=g+1$, and obtain the lemma.
\end{proof}

  For a fixed  $\sigma>1$, e.g., for
  $\sigma=1.2$,
choose a 
 constant $\beta>1$ 
and a sufficiently large $g=O(\log(d))$
such that $\gamma<\pi$ 
in (\ref{eqerrprb0}).
Then under the Random Root Model the  probability $P$ of the output error for Alg. \ref{algexclIIr}(rrm)
is less than a constant
$\nabla<1$. 
We can decrease the error probability below 
$\nabla^v$ by applying
the algorithm $v$ times,
in particular below $d^a$
for any fixed positive constant $a$  and for $v=O(\log(d))$.
We arrive at the following 

\begin{corollary}\label{coerrprb}
Under Random Root Model
one only needs to query an evaluation oracle to perform  an e/i test
correct;y whp.
\end{corollary}




\begin{remark}\label{recchlft}
It is tempting to combine 
root-lifting or root-squaring with the algorithms of Part III to obtain an
alternative randomized e/i test supporting Cor. \ref{coerrprb}, but for that we would need  to raise disc isolation from order of $1+1/m$ to a constant; this would require too costly lifting of order $m$.
\end{remark} 
     

\subsection{Counting  the evaluation queries of our root-finders}\label{sovrlhla}     
     

Combine Cors.  \ref{cosgmz}
and \ref{coerrprb} and Prop. \ref{theitstm} and obtain

\begin{corollary}\label{coeitstm}
We can  reduce root-finding to deterministic  evaluation of $p(x)$ at $O(m^2b\log(d))$  points as well as to the evaluation at 
  $O(mb\log^2(d))$ points  under  the Random Root Model.
\end{corollary}

 
 \section{Precision of computing}\label{sprcs11} 
 

\subsection{Doubling computational precision towards its a posteriori estimation}\label{sdblpst11} 

Given an algorithm $\mathbb A$ for black box  polynomial computations  and target tolerance TOL to the output errors,
how should one choose computational precision? 
Here is a simple {\em practical recipe}
from \cite{B96,BF00,BR14} and   MPSolve for obtaining  {\em a posteriori estimate} within a factor of two from 
optimal:   first apply algorithm  $\mathbb A$ with  a lower  
precision, then recursively double it and stop where an output error  decreases below TOL.
 

\subsection{A priori  estimates for computational precision}\label{sapri11}

In the rest of this section we deduce 
{\em a priori estimate}
of $O(b\log(d))$ bits
 for computational precision. Namely, we prove that this precision 
supports correctness of our  subdivision root-finders; our analysis includes 
some advanced  techniques of \cite{LV16}. 
  
Estimating computational
precision of our soft e/i tests we assume that they are applied to the unit disc $D(0,1)$ because we 
 can extend the results to any disc  $D(c,\rho)$
 by applying map (\ref{eqshft}) and increasing the upper bounds on the precision by a factor of $\log_2(1/\rho)$. This factor  is at most $b$ because we 
 only handle discs  of radii
 $\rho>\epsilon=1/2^b$
 until we  reach
  $\epsilon$-discs and  stop. Therefore, our bound  $O(\log(d))$ on the computational precision of our e/i tests for the disc $D(0,1)$ will imply the bound $O(b\log(d))$ for  any disc.
 
 We also make {\bf Assumption  
 0} that  we can request and obtain from evaluation oracle (black box subroutine) the values of $p$
with a relative error bound within 
$1/d^{\nu}$ for any fixed constant $\nu>0$
of our choice; we can represent these values with the bit-precision  $\log_2(d)+\log_2(\nu)$.

  In  our e/i tests we only need to compare the values  $|{\rm NIR}_{t_h}(x)|$
 at some complex points $x=v_d\ge 1/d^{\beta}$
 for a constant $\beta>0$, and next we  estimate that under  Assumption 0 we can indeed do this by
performing our  computations with a precision of $O(\log(d))$  bits.
 
Multiplication with a precision $s$ contributes a relative error at most $2^s$
(cf. \cite[Eqn. A.3 of Ch. 3]{BP94}) and
otherwise just
sums the relative error bounds of the factors for we ignore the dominated impact of  higher order errors.   Hence we deduce that $h=O(\log(d))$
multiplications for root-lifting to their $h$th powers increase
these bounds by at most  a factor of $h$, which is immaterial  since 
we allow relative errors of order $1/d^{\nu}$ for a constant $\nu>0$ of our choice.  

We apply this analysis to
 approximation of   NIR$(x)$
 for $p_h(x)$ based on 
(\ref{eqNRNIR}) and under the following \\
{\bf  Assumption 1:} Whp all $q$ evaluation points  $x\in C(0,1)$ of an e/i test for the unit disc $D(0,1)$
lie at the
distance at least  $\phi:=\frac{2\pi}{mqw}$ from all roots for any fixed $w>1$ of our choice.
 
This assumption holds  whp because we apply our  e/i tests under random rotation $x\mapsto vx$  for  random variable $v$ specified in our next proposition. 
 
\begin{proposition}\label{thisass}  
Let $\phi_g:=\min_{j=1}^{d}|v\zeta^g-z_j|$ and let $\phi:=\min_{g=0}^{q-1} \phi_g$ for the roots $z_j$, $\zeta$ of  (\ref{eqdndgnr}) and random variable $v$ sampled from a fixed arc of  the circle  $C(0,1)$ of length $2\pi/q$  under the uniform probability distribution on that arc.
Let $i(D(0,1))=m$.  Then
$\phi\le \frac{2\pi}{mqw}$  for $w>1$ with a probability at most $\frac{\pi}{w}$. 
\end{proposition}
\begin{proof}
Let $\psi_j$ denote the length of the arc $C(0,1)\cap D(z_j,\frac{2\pi}{mqw})$, $j=1,2,\dots,d$. Then $\psi_j=0$ for at least $d-m$ integers $j$,
and $\max_{j=1}^d\psi_j \le \frac{\pi^2}{mqw}$ for sure. Hence 
$\sum_{j=1}^d\psi_j < \frac{\pi^2}{qw}$, and so 
$\phi_g\le \frac{2\pi}{mqw}$ for a fixed $g$ 
with a probability
at most $\frac{\pi}{qw}$, and the lemma follows.
\end{proof}
 
 Next,  by applying Prop. \ref{thisass} for a sufficiently large $w=d^{O(1)}$, we prove that a precision
of $O(\log(d))$ bits supports our e/i tests
under 
Assumptions 0 and  1.
We will
 (i) prove that
$|{\rm NIR}(z)|=d^{O(1)}$, (ii) will approximate this value within $d^{\nu}\delta$ for any fixed constant $\nu$ and for $\delta$
 of (\ref{eqNRNIR}) chosen
such that    
 $\log(\frac{1}{\delta})=O(\log(d))$, and (iii) will estimate that 
such a bound withstands the impact of rounding errors of computations for (\ref{eqNRNIR}). 
  
 Eqn. (\ref{eqratio})
 immediately implies
\begin{proposition}\label{thnirprc11} 
Under Assumption 1,
$\max_{g=1}^q|{\rm NIR}(\zeta^g)|\le
{d^{\eta}}$ for $\eta=O(1)$.  
\end{proposition}
\begin{corollary}\label{coprc011}
 A precision of 
$\lceil(\eta+\beta)\log_2(d)\rceil$
    bits is sufficient  to represent 
 ${\rm NIR}(x)$ within $\frac{1}{d^{\beta}}$ for $g=1,2,\dots,q$.
 \end{corollary}
          
Instead of ${\rm NIR}(x)=\frac{ t'(x)}{t(x)}$  for $x\in C(0,1)$ we actually
approximate $\frac{ t'(y)}{t(x)}$ where $|y-x|\le \delta$ and $t'(y)$ is equal to the divided difference of
  (\ref{eqNRNIR}).  Thus we
shall increase the above error bound  by adding upper bounds   
 $\alpha$ on $|\frac{ t'(x)}{t(x)}-\frac{ t'(y)}{t(x)}|$ and 
$\beta$ on the rounding error of computing $
\frac{t(x)-t(x-\delta)}{\delta t(x)}$, 
 for $x=\zeta^g$ and
$g=1,2,\dots,q$. 

In the proof of the next theorem  we use the following lemma, which is \cite[Fact 3.5]{LV16}. 
  
\begin{lemma}\label{lefct3511}  
  For  $t(x)=\prod_{j=1}^d(x-y_j)$ and a non-negative integer $j\le d$ it holds that
$$t^{(j)}(x)=j!~t(x)\sum_{S_{j,d}}~\prod_{j\in S_{j,d}} \frac{1}{x-y_j}$$
where the summation
$\sum_{S_{j,d}}$ is over all subsets $S_{j,d}$ of  the set $\{1,\dots,d\}$ having cardinality $j$. 
\end{lemma}

\begin{proposition}\label{thdltprc0111}
Under the assumptions of Prop. \ref{thnirprc11} it holds that 
$$|t(x)\alpha|=\Big|\frac{t(x)-t(x-\delta)}{\delta}-t'(x)\Big|\le \Big|t(x)\frac{ (d\xi')^2\delta}{1-d\xi'\delta}\Big|~{\rm for}~\xi'=\frac{1}{\rho}.$$
\end{proposition}  
 
\begin{proof}
The
claimed bound on $|\xi'|$ follows
from  (\ref{eqratio}). It remains to 
apply the  first five lines of 
the proof of \cite[Lemma 3.6]{LV16} with $f$, $\alpha$, $n$, and $\xi$ replaced by
 $t$, $-\delta$, $d$, and $w$, respectively. 
Namely, first obtain from Taylor's expansion that 
 $t(x)-t(x-\delta)=\sum_{j=0}^{\infty}\frac{\delta^j}{j!}t^{(j)}(x)-t(x)$.
 
 Substitute the expressions of Lemma \ref{lefct3511} and obtain
 $$t(x)-t(x-\delta)=\delta t'(x)+\sum_{j=2}^{\infty} \delta^j t(x)\sum_{S_{j,d}}~\prod_{j\in S_{j,d}}\frac{1}{x-y_j}.$$
 Combine this equation with the assumed bounds on 
 $\frac{1}{|x-y_j|}$ 
 and deduce that 
 $$\Big|\frac{t(x)-t(x-\delta)}{\delta}-t'(x)\Big|\le \Big|\frac{f(x)}{\delta}\sum_{j=2}^{\infty} \delta^jw^j d^j\Big|\le 
 \Big|f(x)\frac{(d~w)^2~\delta}{1-\delta~d~w}\Big|.$$ 

\end{proof}
 
\begin{corollary}\label{coprc11} We can ensure that
  $\alpha<1/d^{\bar\eta}$ for
 any  fixed constant $\bar\eta$ by choosing  $\delta$ of order $d^{O(1)}$, represented with a precision of $O(\log(d))$ bits.
 \end{corollary}
  
 
 \begin{proposition}\label{therrnbl11} 
Suppose that the values $t(y)$ have been  computed by an oracle within a relative error bounds $\nabla(y)$ 
for $y=x$, $y=x-\delta$, and
$\delta$ of Cor. \ref{coprc11}.
  Then one can ensure  that $\beta<\frac{1}{d^{\bar\eta}}$ for any fixed constant $\bar\eta$ by choosing 
 a proper $\nabla(y)$ of order $1/d^{O(1)}$,
 represented with a precision of $O(\log(d))$ bits, which we can do under Assumption 0.
\end{proposition}

\begin{proof}
Write  
$$\beta\delta=\frac{t(x-\delta)(1+\nabla(x-\delta))}{t(x)(1+\nabla(x))}-\frac{t(x-\delta)}{t(x)}=\frac{t(x-\delta)}{t(x)}~\frac{\nabla(x)-\nabla(x-\delta)}{1+\nabla(x)}.$$

 Taylor-Lagrange's formula  implies that
  $$\frac
{t(x-\delta)}{t(x)}=1+\frac{t(x-\delta)-t(x)}{t(x)}=1+\delta\frac{t'(u)}{t(x)},~{\rm 
for}~u\in[x-\delta,x].$$ Hence $|\frac
{t(x-\delta)}{t(x)}|\le 1+|\delta\alpha'|+|\frac{t'(x)}{t(x)}|$ for $\alpha'=\frac{t'(u)-t(x)}{t(x)}$. 

By extending Cor. \ref{coprc11}  obtain 
$|\alpha'|\le \frac{1}{d^{O(1)}}$
and  deduce from (\ref{eqratio}) and Assumptions 0 and 1 that 
$|\frac{t'(x)}{t(x)}|\le
 \frac{1}{d^{O(1)}}$
 (cf. the proof of Prop. \ref{thnirprc}). 
Hence 
$$w:=\Big|\frac{t(x-\delta)}{t(x)}\Big|\le d^{O(1)}~{\rm
and}~|\beta\delta|\le |w|\frac{|\nabla(x)|+|\nabla(x-\delta)|}{1-|\nabla(x)|}.$$

Now choose $\nabla(y)$ such that $|\nabla(y)|\le\frac{|\delta|}{20|w|}$ for $y=x$
and $y=x-\delta$.
 Then verify that
$$\log\Big(\frac{1}{|\nabla(y)|}\Big)=O(\log(d)),~ 1-|\nabla(x)|\ge 1-\frac{|\delta|}{20|w|}>\frac{19}{20},~{\rm and}~ |\delta\beta|< \frac{|\delta|}{8}.$$                                                             
\end{proof} 

Combine Cors. \ref{coprc011} and \ref{coprc11} and Prop. \ref{therrnbl11} to obtain
\begin{corollary}\label{coprcall11}
Given an oracle for the evaluation of a  polynomial $p(x)$
of a degree $d$ 
with a relative error  of any order in $1/d^{O(1)}$ (this holds under Assumption 0), we can perform our black box subdivision root-finders with 
a precision of $O(\log(d))$ bits.
\end{corollary} 

\begin{remark}\label{reprc}
We derived our precision bound $O(\log(d))$ (in the case of root-finding in the unit disc $D(0,1)$) based on operating entirely with NIR$(x)$.
For comparison, if we relied on computing $p(x)$ and on the  bound   
 $\min_{x:~|x|=1}(|p(x)|/\sum_{i=0}^d|p_i|)\ge (\frac{\theta-1}{2\theta})^d$ of \cite[Eqn. (9.4)]{S82},
 we would have only needed  a bound of  order $m^2d$ on the computational precision of our  root-finders.
\end{remark}

\section{Approximation of matrix eigenvalues}\label{smxeig11}
  
 Given a $d\times d$ matrix $M$ that has characteristic polynomial $p(x)=\det (xI-M)$,  our root-finders can approximate all  its $m$ zeros (that is, the eigenvalues of $M$) lying in a fixed disc $D$ on the complex plane isolated from its external zeros (or  eigenvalues of $M$).
 
 Recall  Gershgorin's  bound for any  eigenvalue $\lambda$ of a matrix $M=(a_{i,j})_{i,j=1}^d$ 
 \cite[Thm. 1.3.2]{S98}:
$$|\lambda-a_{i,i}|\le \sum_{j\neq i}|a_{i,j}|~{\rm for~some}~i,~1\le i\le d.$$ 

Hence, clearly, $|\lambda|\le ||M||_F$.

 Combine this bound with the a priori precision estimates of the previous section for our 
root-finders\footnote{For a posteriori estimates for the bit operation complexity of Problems 0, $0^*$, 2, and 2$^*$, one can combine our root-finders with the recipe of doubling the precision
of computing  of Sec. \ref{sdblpst11}.} and 
 with the following result of Storjohann \cite{S05}:
\begin{theorem}\label{thstrh11} One can evaluate the determinant of a $d\times d$ integer matrix $A$ by using an expected number of $O(d^{\omega}\log^2(d)\log(|A||_F))$ bit operations. 
\end{theorem}
  For any $d\times  d$  matrix $M$ we arrive at  a record  bit operation 
 complexity bound
$\tilde O(m^2 \tilde bbd^{\omega})$ for approximation within 
$\epsilon$ of all  $m$ eigenvalues of any matrix $M$ lying in the disc $D$ provided  that $\tilde b=\log_2(||M||_F/\epsilon)$.   This  bound 
exceeds (\ref{eqblneig}) by a factor of $b$ and can be decreased by a factor of $m$ under the Random Root Model for the roots of the  characteristic polynomial of $M$.
 

\section{Bit operation complexity of polynomial root-finding}\label{sblnrtfdg}
 
 \subsection{Straightforward extension of the precision bound and a direction to  saving evaluation queries}\label{sstrfwdr}

 
Recall the bound
\begin{equation}\label{eqmu}
 \mu(s):=O(s\log(s)=\tilde O(s)
\end{equation}
on the number of bit operations  for integer multiplication modulo $2^s$
  \cite{HH21}   
(cf. \cite{SS71,K81/97,S82,S82a}), combine it with 
the  estimate $O(bm)$ of 
Cor. \ref{coeitstm} for the number of evaluation  queries in our root-finders,  the upper bound $2d$ ops for the evaluation of a polynomial $p(x)$ given 
 with its coefficients, and the precision bound $O(b\log(d))$  of Sec. \ref{sprcs11},    and obtain the bit-operation complexity bound $\tilde O(b^2dm)$
for our 
root-finders.\footnote{By extending precision bound of \cite{LV16} we greatly simplified derivation of our bit-complexity estimates versus
the papers \cite{S82,K98}, each of 46 pages, which bypass precision bound and had to use long tedious step-by-step analysis of their
much involved root-finders.}

Since $p(x)$ is given 
 with its coefficients, we apply fast multipoint evaluation
 of a polynomial in \cite{F72,MB72} and use fewer ops,  
    roughly by a factor of $d/\log^2(d)$. Peter
Kirrinnis in  \cite{K98}  and
 Moroz in \cite{M21} have non-trivially extended this progress   to saving bit-operations:
 
\begin{theorem}\label{thkrrn11} 
 Given a positive $b$, the coefficients of  a polynomial $t(x):=\sum_{i=0}^d t_ix^i$ such that   $||t(x)||_1=\sum_{i=0}^d|t_i|\le 2^{\tau}$, and $q$ complex points $x_{1},\dots,x_q$  in the unit disc $D(0,1)$, one can approximate the values $p(x_{1}),\dots,p(x_q)$
within  $1/2^b$ by using $\tilde O((d+q)(b+\tau))$ 
 bit operations. 
\end{theorem}  

\begin{proof}  
For $q=d$ this is 
\cite[Thm. 2]{M21}, based on using novel data structure for hyperbolic approximation of $p$, proposed and analyzed by Moroz  in \cite{M21}. 
Trivially extend this theorem to the cases where $q>d$ (by  considering $t(x)$ a 
polynomial of degree $q$ with $q-d$ leading coefficients 0)
and  where $q<d$ (by removing $d-q$ points $x_i$, with no increase of the complexity bound).
\end{proof}  
 
 We only apply Thm. \ref{thkrrn11}  for $\tau=0$, in which case its 
 bit operation cost bound turns into  
\begin{equation}
\label{eqmrz11}
\mathbb B_{\rm Moroz}=\tilde O((d+q)b).
\end{equation}
 
 For comparison, the  State of the Art before \cite{M21} was the bound
 \begin{equation}
 \label{eqkrrnn11}
\mathbb B_{\rm Kirr}=\tilde O((d+q)(b+q)) 
\end{equation}
of \cite[Thm. 3.9, Alg. 5.3, and Appendix A.3]{K98} by Kirrinnis, who  combined his 
and Sch{\"o}nhage's study  of the bit operation complexity
 of basic polynomial computations (cf. \cite{S82}) with  the algorithm  of \cite{F72,MB72}.\footnote{Precision of computing
 required in the latter 
 algorithm grows fast as $d$ increases; e.g.,   the IEEE standard double precision of the customary numerical computations is not enough to support that algorithm already for $d>50$; for $b=O(\log(d))$ one can fix this deficiency by reducing the task to multiplication of Vandermonde matrix by a vector applying  FMM  (cf. \cite{P15,P17}). One can, however, support  bounds of Eqns. (\ref{eqmrz11}) and (\ref{eqkrrnn11}) on the overall bit operation complexity  by computing with multiple precision for larger values $b$.}  Moroz obtained 
 \cite[Thm. 2]{M21} by applying bound (\ref{eqkrrnn11}) to his lower degree polynomials approximating $p(x)$.
 
 
 Bound (\ref{eqmrz11}) combined with the  precision bound $O(\log(d))$ of  Sec. \ref{sprcs11}
implies the bit operation  
 cost bound $\tilde O((d+q)b)$, dominated by 
an estimated bit operation cost  $\mathbb B$ for the evaluation of  $p(x)$  with relative error $1/d^{O(1)}$ at $\tilde O(db)$ points.
 It remains to estimate  $\mathbb B$, which we do next.

\subsection{Reduction of our root-finders  to multipoint polynomial approximation}\label{srdctnmltpnt}
  
 Given a complex $c$, a pair of positive $b$ and $\rho$,  and
 $d+1$ coefficients of a polynomial $p=p(x)$ of (\ref{eqpoly}) such that the disc $D(c,\rho)$ is isolated and contains precisely $m$ roots, we approximate all these  roots within $R/2^b$
 for $R=|c|+\rho$
by applying subdivision root-finding iterations 
 with  e/i tests of Sec.  \ref{sei}. Then  every subdivision step   
  is reduced  to application of our  e/i tests
 to $\bar m=O(m)$ discs $D(c_{\lambda},\rho_{\lambda})$, where  $\rho_{\lambda}\ge R/2^b$ for all $\lambda$. Such a test amounts  to approximation of  $p(x)$  at $q$ equally spaced points $x$ on each   circle
 $C(c_{\lambda},\rho_{\lambda})$.
 
 In Thm. \ref{thkrrn11} 
a polynomial $t(x)$ 
with $||t(x)||_1\le 1$
is approximated at some points $x\in D(0,1)$. 
To ensure the assumption about $x$, 
 we scale the variable $x\mapsto
Rx$ to map the discs $D(c_{\lambda},\rho_{\lambda}\sigma)\mapsto D(\bar c_{\lambda},\bar\rho_{\lambda})\subseteq D(0,1)$, for
 $\bar c_{\lambda}= c_{\lambda}/R$,
 $\bar\rho_{\lambda}=\rho_{\lambda}/R\ge 1/2^b$,
 and all $\lambda$.
 
 Then we write $t(x):= p(Rx)/\psi$ for $\psi=O(R^d)$
 such that
$||t(x)||_1=1$ and     
approximate $t(x)$ within $1/2^{\bar b}$ or equivalently within the relative error bound  $1/(|t(x)|2^{\bar b})$ at $\bar m$ points $x$ in  $\bar m$ e/i 
tests.\footnote{Under the Random Root Model of Sec. \ref{scmplests} we can whp decrease this bound by a factor of $\bar m/\log(\bar m)$  and then extend this decrease  to the related subsequent estimates.} 
Due to Cor. \ref{coprcall11} it is sufficient to ensure the bound 
$1/(|t(x)|2^{\bar b})=1/d^{O(1)}$; we  deduce it by choosing $\bar b\ge \log(\frac{1}{|t(x)|})+O(\log(d))$. 

In   our e/i tests we only need to approximate  $t(x)$ at the points $x$ lying on $\theta_{\lambda}$-isolated circles 
 where
 $\frac{1}{\theta_{\lambda}-1}=O(m)$  (cf. Assumption 1 of Sec. \ref{sapri11}). Then  
$\log(\frac{1}{|t(x)|})=O(d\log(m))$
by virtue of  \cite[Eqn.   
(9.4)]{S82}).\footnote{That bound follows from
the estimate $|f|\ge 1/2^d$
of \cite[Thm. 4.2]{S82}, which holds provided that
$t(x)=f\prod_{j=1}^m(x-x_j)\prod_{j=m+1}^d(\frac{x}{x_{j}}-1)$, $||t(x)||_1=1$,
$|x_j|<1$ for $j\le m$, and
$|x_j|>1$ for $j> m$.} 

Substitute $\bar \rho_{\lambda}\ge 1/2^b$ and obtain 
$\bar b=\log(\frac{1}{|t(x)|})+O(\log(d))=O(d\log(m)+b)$.
\begin{remark}\label{relbt11}
For $m=d$, the above bound on $\log(\frac{1}{|t(x)|})$ is sharp up to a constant factor and is reached for the polynomial $t(x)=(x+1-\frac{1}{m})^d/(2-\frac{1}{m})^d$, $\rho_{\lambda}=1$, and $x=-1$. For $m<d$, \cite[Thm.   
4.5]{S82} implies a little stronger   bounds, namely,   $\log(\frac{1}{|t(x)|})=O(d+ m\log(m))$ and hence
$\bar b=O(m\log(m)+d+b)$
 because all discs $D(\bar c_{\lambda},\bar\rho_{\lambda})$ lie in  $\theta$-isolated unit disc $D(0,1)$ for $\theta-1$
exceeding a positive constant. 
\end{remark}

\subsection{Bit operation cost of our e/i tests and root-finders}\label{sblnrts1} 
 
 Apply  Thm. \ref{thkrrn11} 
with $q$ and $b$ replaced  
by $\bar q=O(m^2)$ and $\bar b=O(b+d+m\log(m))$, respectively, and obtain that the bit operation cost of our $O(m)$ e/i tests  of Sec. \ref{sei} at any fixed subdivision step is in 
 $\tilde O((b+d+m)m^2)$
  for  
$1\le m\le d$. This implies the upper estimate
\begin{equation}\label{eqbrts11}    
 \mathbb B_{\rm roots}= \tilde O((b+d)(m^2+d)b)
\end{equation}
for the overall bit operation cost of our $O(b)$ subdivision 
steps.
Based on \cite[Alg. 5.3]{K98} 
and Eqn. (\ref{eqkrrnn11}) we obtain a little larger but still reasonably small upper bound
\begin{equation}\label{eqbrts'11} 
\mathbb B_{\rm roots}'= \tilde O(((m^2 + d)(m^2 + d + b)b).
\end{equation}

\subsection{Higher bit operation  complexity of   general polynomial root-finding based on Pellet's theorem, map (\ref{eqshft}), and root-lifting}\label{sblnscl} 

Recall  the followings theorem 
of Pellet 
 \cite{P881}, which extends the previous results by Cauchy and Rouch{\'e}: 
 
\begin{theorem}
Given a monic polynomial $t(x)=\sum_{i=0}^dt_ix^i$ with complex coefficients and $t_d=1$, let the
polynomial
$s(x)=x^d+p_{d-1}
x^{d-1}+\cdots +|p_{i+1}|x^{i+1}-|p_i|x^i|+p_{i-1}|x^{i-1}+\cdots+|p_0|$
have two distinct positive roots $r$ and $R$, $r<R$. Then the
polynomial $t(x)$ has exactly 
$i$ zeros in or on the
circle $C(0,r)=\{z:~|z|=r\}$ and no zeros in the annulus 
$A(0,r,R)=\{z:~r<|z|<R\}$.
\end{theorem}

One must know the coefficients 
of a polynomial $t(x)$ to apply this theorem and  then
 can  
perform  very fast 
an e/i test for a disc
centered at the origin.
  
Such
soft e/i tests were the basis of the subdivision root-finders of \cite{BSSXY16,BSSY18}, which apply them 
to polynomials
$t(x)$ obtained from
 $p(x)$ of 
(\ref{eqpoly}) by means of shifts of the variable $x$, scaling it  
by  factors ranging from $R/2^b$ to
$R=\max_{j-1}^d|z_j|$, and order of $\log(\log(d))$ root-squaring steps.

To apply Pellet's  theorem to such a polynomial $t(x)$
 one must compute  the values $|\frac{t_i}{t_d}|$ for $i=0,1,\dots,d-1$.   Because of scaling alone, 
 the coefficient $t_i$ must be computed within $1/2^{b-i}$ to approximate the zeros of $p$ within $1/2^b$
 for every $i$, $i=0,1,\dots,d$.
 
 As in the proof of Observation \ref{obbd},
 it follows that one must process at least  
  $0.5(d+1)db$ bits of the coefficients of $t(x)$ and must perform at least  
  $0.25(d+1)db$ bit operations
 to support even a single 
 e/i tests towards
 approximation of the   zeros of $p(x)$ within $1/2^b$. 
  
 Already a single subdivision iteration
for Problem 1${_m}$ can involve up to 
 $4m$ soft e/i tests
and hence up to
at least $0.25(d+1)dbm$
bit-operations.
This is {\em  by a factor of $d$ exceeds}  the  bound of  \cite{P95,P02}  for approximation ofall  $m$ roots in an isolated disc.

Such a  bit-count and  the bit operation cost bounds for the root-finders of \cite{P95,P02,BSSY18}  decrease   by a factor of $d$ where the minimal pairwise distance between the zeros of $p$ exceeds a positive constant; {\em the ratio of the bounds of \cite{BSSY18} and  \cite{P95,P02} remains at least} $d$.

  Similar argument shows that already  for a singe e/i test {\em representation of the leading coefficients}   $t_d$ of $t(x)$ alone under scaling involves order of
  $bd$ bits. Moreover, this is amplified to $d^2b$ for $t_d^d$ in every  e/i test  and up to $d^2bm$ per a subdivision iteration with $m$ e/i tests.

It may be interesting that a shift by $c$ such that $|c|\le 2$  is much less costly than 
scaling:\footnote{We can relax the  restriction $|c|\le 2$ above  by means of scaling the variable 
 $x$, but as we showed above,  scaling can greatly increase  the   bound on the  overall number of bit operations involved.}

\begin{theorem}\label{thkrrndft}  {\rm Bit operation cost of the shift of the variable, \cite[Lemma 2.3]{S85}
(cf. \cite[Lemma 3.6]{K98}).}  Given a positive $b$, the coefficients of  a polynomial $p(x):=\sum_{i=0}^d p_ix^i$ such that   $||p(x)||_1\le 1$, and 
 a complex $c$ such that $|c|\le 2$, one can approximate    within $1/2^b$ the coefficients $t_0,t_1,\dots,t_d$
of the polynomial $t(x):=\sum_{i=0}^d t_ix^i:=p(x-c)$ by performing   $O(\mu((b+d)d))$ bit operations for $\mu(s)$ of (\ref{eqmu}).
 \end{theorem}

  

 
   

{\Large \bf PART III: Classical subdivision root-finder with Cauchy-based novel e/i tests}
 
\medskip

Next we elaborate upon  Alg. \ref{algeic} and analyze it.
\medskip

{\bf Organization of the presentation  in Part III:}
 We devote the next section to some basic definitions.  
 In Sec. \ref{scntavr} we count roots in a disc  on the complex plane by means of approximation of a Cauchy integral over its boundary circle.
In Sec. \ref{mdls} we 
 analyze  the resulting
 root-counters and e/i tests in a  disc.   
   In Secs.  \ref{sprcs}
  we estimate the computational precision supporting  our root-finders. 
   In Sec.   \ref{sblnrtfnd} we estimate their bit operation complexity   assuming that a polynomial $p(x)$ is represented with its coefficients.

\section{Background: basic definitions}\label{sbdef}
\begin{itemize}
 \item
``Wlog" stands for  ``without loss of  generality".
 \item
$S(c,\rho):=\{x:~|\Re(c-x)|\le 
 \rho$, 
 $|\Im(c-x)|\le \rho\},~ 
D(c,\rho):=\{x:~|x-c|\le \rho\}$,

$C(c,\rho):=\{x:~|x-c|= \rho\}$, and 
$A(c,\rho,\rho'):=\{x:~\rho\le |x-c|\le \rho'\}$ \\
denote
a square, a disc, a circle (circumference),  and an annulus (ring) on the complex plane, respectively.

 \item
 The largest  upper bound
 on such a value $\theta$ is  
  said to be the
 {\em  isolation} of 
the disc $D(c, \rho)$, the circle $C(c, \rho)$, or the
 square $S(c, \rho)$, respectively (see Fig. \ref{fig6}), and is denoted
 $i(D(c,\rho))$, $i(C(c,\rho))$, and
 $i(S(c,\rho))$,  respectively.


\item
   $r_1(c,t)=|y_1-c|,\dots,r_d(c,t)=|y_d-c|$ in 
   non-increasing order are the $d$ {\em root radii}, that is, the distances from a  complex point $c$ to the zeros $y_1\dots,y_d$ of a $d$th degree polynomial $t(x)$.
 $r_j(c):=r_j(c,p)$, $r_j:=r_j(0)$ for $j=1,\dots,d$.
\end{itemize}

 \begin{observation}\label{obscchy}
Eqn. (\ref{eqshft}) maps the zeros $z_j$ of $p(x)$
into the zeros $y_j=\frac{z_j-c}{\rho}$
of $t(y)$, for $j=1,\dots,d$, and  preserves
 the index $\#(D(c,\rho))$ and the isolation  $i(D(c,\rho))$.
\end{observation}


\section{Cauchy root-counting in a disc}\label{scntavr}

\subsection{The power sums of the roots in a disc}\label{scntuntd}

 Generalize (\ref{eqchint}) as follows  (cf.  \cite{A00}): 
 \begin{equation}\label{eqnps}
 s_h=s_h(\mathcal D):=\sum_{j:~z_j\in \mathcal D}z_j^h=\frac{1}{2\pi \sqrt{-1}}\int_{\mathcal C}x^h~\frac{p'(x)}{p(x)} ~dx,~{\rm for}~h=0,1,\dots.
\end{equation} 
For  $\mathcal D=D(0,1)$
and $\zeta$ of (\ref{eqdndgnr})  approximate $s_h$
with finite sums (cf. \cite{S82,TW14})
 \begin{equation}\label{equ7.12.2}
s_{h,q}: =\frac{1}{q}\sum_{g=0}^{q-1}\zeta^{(h+1)g}~\frac{p'(\zeta^g)}{p(\zeta^g)}~{\rm for}~h=0,1,\dots,q-1
\end{equation}    
and recall that $s_0(D(0,1))=\#(D(0,1))$
(cf.
(\ref{eqchint})), while  an e/i test is just the decision whether 
$\#(D(0,1))>0$.


 \subsection{Cauchy sums in  the unit disc as  weighted  power sums of the roots}\label{scchrts}  
  

Cor. \ref{copwrsm0} of the following lemma expresses the $h$th Cauchy sum  $s_{h,q}$
as the sum of the $h$th powers of the roots $z_j$ with the weights 
$\frac{1}{1-x^q}$.

\begin{lemma}\label{lepwrsm}
For a complex $z$, two integers  $h\ge 0$ and $q>1$,  and $\zeta$ 
of (\ref{eqdndgnr}),
it holds that
\begin{equation}\label{eqratiosm}
\frac{1}{q}\sum_{g=0}^{q-1}\frac{\zeta^{(h+1)g}}{\zeta^g-z}=\frac{z^h}{1-z^q}.
\end{equation}
\end{lemma}
 
\begin{proof}  
First let $|z|<1$ and obtain
$$\frac{\zeta^{(h+1)g}}{\zeta^g-z}=
\frac{\zeta^{hg}}{1-\frac{z}{\zeta^g}}=
\zeta^{hg}\sum_{u=0}^{\infty}\Big(\frac{z}{\zeta^g}\Big)^u=
\sum_{u=0}^{\infty}\frac{z^u}{\zeta^{(u-h)g}}.$$ 
The equation in the middle follows from Newman's expansion 
$$\frac{1}{1-y}=
\sum_{u=0}^{\infty}y^u~{\rm for}~ 
y=\frac{z}{\zeta^g}.$$
We can apply it because $|y|=|z|$, while
 $|z|<1$ by assumption. 

Sum the above expressions in $g$, recall that
\begin{equation}\label{eqprmrt} 
\zeta^u\neq 1~{\rm and}~\sum_{s=0}^{q-1}\zeta^u=0~{\rm for}~0<u<q, 
~\zeta^q=1,
\end{equation} 
 and deduce 
that $$\frac{1}{q}\sum_{g=0}^{q-1}\frac{z^u}{\zeta^{(u-h)g}}=z^u$$ for $u=h+ql$ and an  integer $l$, while  
$\sum_{g=0}^{q-1}\frac{z^u}{\zeta^{(u-h)g}}=0$ for all other values $u$.  
Therefore
$$\frac{1}{q}\sum_{g=0}^{q-1}\frac{\zeta^{(h+1)g}}{\zeta^g-z}=z^h\sum_{l=0}^{\infty}z^{ql}.$$
 Apply Newman's expansion for $y=z^q$ and deduce (\ref{eqratiosm})
for $|z|<1$. 
By continuity extend  (\ref{eqratiosm})
to the case of $|z|=1.$

Now let $|z|>1$. Then  
$$\frac{\zeta^{(h+1)g}}{\zeta^g-z}=-\frac{\zeta^{(h+1)g}}{z}~~\frac{1}{1-\frac{\zeta^g}{z}}=-\frac{\zeta^{(h+1)g}}{z}\sum_{u=0}^{\infty}\Big(\frac{\zeta^g}{z}\Big)^u=
-\sum_{u=0}^{\infty}\frac{\zeta^{(u+h+1)g}}{z^{u+1}}.$$ 
Sum these expressions in $g$,
write $u:=ql-h-1$, apply (\ref{eqprmrt}),
and obtain 
$$\frac{1}{q}\sum_{g=0}^{q-1}\frac{\zeta^{(h+1)g}}{\zeta^g-z}= 
-\sum_{l=1}^{\infty}\frac{1}{z^{ql-h}}
=-z^{h-q}\sum_{l=0}^{\infty}\frac{1}{z^{ql}}.$$
Apply Newman's expansion for $y=z^q$ again
and obtain 
$$\frac{1}{q}\sum_{g=0}^{q-1}\frac{\zeta^{(h+1)g}}{\zeta^g-z}=
-\frac{z^{h-q}}{1-\frac{1}{z^{q
}}}=\frac{z^h}{1-z^q}.$$
Hence (\ref{eqratiosm}) holds
in the case where $|z|>1$ as well.
\end{proof}

Sum equations (\ref{eqratiosm}) for 
$z=z_j$ and 
$j=1,\dots,d$, combine Eqns. \ref{eqratio} and
(\ref{equ7.12.2}), and 
 obtain
\begin{corollary}\label{copwrsm0}
For the roots $z_j$   and all $h$, the Cauchy sums $s_{h,q}$ of (\ref{equ7.12.2}) satisfy
$s_{h,q}=\sum_{j=1}^d\frac{z_j^h}{1-z_j^q}$
unless $z_j^q=1~
{\rm for~some}~j.$
\end{corollary}


 \subsection{Approximation of the power sums in the unit disc}\label{spwrsmerr}  
  

 Cor. \ref{copwrsm0} implies that                                                                                                                                                                                                                                                                                                                                                                                                                                                                                                                                                                                                                                                                                                                                                                                                                                                                                                       
 the  values $|s_{h,q-}s_{h}|$  decrease exponentially in $q-h$ if the unit disc $D(0,1)$ is isolated.  

\begin{theorem}\label{thpwrsm0} 
\cite[Eqn. (12.10)]{S82}.\footnote{Unlike the proof in \cite{S82}, we rely on Cor. \ref{copwrsm0} of independent interest.  Thm. \ref{thpwrsm0} does not follow from
  \cite[Thms. 2.1 and 2.2]{TW14} but can possibly be proved with some techniques of \cite{TW14}.}
 Let the unit circle $C(0,1)$ be   $\theta$-isolated and let $d_{\rm in}$ and $d_{\rm out}=d-d_{\rm in}$ denote the numbers
 of the roots lying in and outside the disc
 $D(0,1)$, respectively.    
 Then 
\begin{equation}\label{equ7.12.6} 
|s_{h,q}-s_h|\le
\frac{d_{\rm in}\theta^{-h}+d_{\rm out}\theta^{h}}{\theta^q-1}\le
\frac{d\theta^{h}}{\theta^{q}-1}~{\rm for}~h=0,1,\dots,q-1.
\end{equation}  
\end{theorem}

\begin{proof}
Enumerate the roots $z_j$ so that 
$|z_j|<1$ if and only if $j\le d_{\rm in}$;  
then  Cor.
 \ref{copwrsm0} implies that
\begin{equation}\label{equshqsh} s_{h,q}-s_h=s_{h,q,{\rm in}}+
s_{h,q,{\rm out}}
\end{equation}
where
\begin{equation}\label{equshqinout}
s_{h,q,{\rm in}}=\sum_{j=1}^{d_{\rm in}}z_j^h\Big(\frac{1}{1-z_j^q}-1\Big)~
{\rm and}~s_{h,q,{\rm out}}=\sum_{j=d_{\rm in}+1}^{d}\frac{z_j^{h}}{1-z_j^q}=\sum_{j=d-d_{\rm out}+1}^{d}\frac{z_j^{h}}{1-z_j^q}.
\end{equation}

Notice that $\frac{1}{1-z_j^q}-1=
\frac{z^q}{1-z_j^q}$, and so $|z_j^h(\frac{1}{1-z_j^q}-1)|\le
\frac{\theta^{-h}}{\theta^q-1}$ for $z_j\in D(0,1)$, 
while
$\Big|\frac{z_j^{h}}{1-z_j^q}\Big|\le \frac{\theta^{h}}{\theta^q-1}$ for $z_j\notin D(0,1)$.
Combine these bounds with Eqns. (\ref{equshqsh}) and (\ref{equshqinout}) and obtain 
(\ref{equ7.12.6}).
\end{proof}

 By choosing $q$ that exceeds $\log_{\theta}(2d+1)$  we ensure 
that $|s_0-s_{0,q}|<1/2$.
 

\begin{example}\label{exrtcnt}
Let  the unit circle $C(0,1)$ be 2-isolated.
Then Thm. \ref{thpwrsm0}  implies  that $|s_{0,q}-s_0|<1/2$ for  $d\le 1,000$ if
$q=10$ and for  $d\le 1,000,000$ if $q= 20$.
 \end{example}  
 
    
  
\begin{algorithm}\label{algexcl}
 \begin{description}  
INPUT: a black box polynomial $p$ of a degree $d$ and $\theta>1$. 
  
  INITIALIZATION: Compute the integer $q= \lfloor\log_{\theta}(4d+2)\rfloor>\log_{\theta}(2d+1)$.

  COMPUTATIONS: Compute Cauchy sum $s_{0,q}$  and output an    
  integer $\bar s_0$ closest to it.
  \end{description}
\end{algorithm}




\begin{observation}\label{thtsts}
(i) Alg. \ref{algexcl} 
evaluates $p(x)$ at
 $q=\lfloor\log_{\theta}(4d+2)\rfloor$ points. (ii) It outputs $\bar s_0=\#(D(0,1))$ if the circle $C(0,1)$ is $\theta$-isolated  for $\theta>1$.
(iii) If the algorithm
outputs  $\bar s_0>0$, then $\#(D(0,\theta))>0$.
\end{observation}
\begin{proof}
 Thm. \ref{thpwrsm0} immediately implies claim (ii) but  also implies that $\#(D(0,1))>0$ if $\bar s_0>0$ unless  
$\#(A(0,1/\theta,\theta))>0$. In both cases 
$\#(D(0,\theta))>0$.
\end{proof}    
 

\subsection{Extension to  any disc}
\label{scchanyd}  

Map (\ref{eqshft}), preserving  $\#(D(0,1))$ (cf. Observation \ref{obscchy}),
 enables us to extend  definition (\ref{equ7.12.2}) of  
Cauchy sum to any disc 
$D(c,\rho)$ as follows:
\begin{equation}\label{equ7.12.00}
s_{0,q}(p,c,\rho):=s_{0,q}(t,0,1): =\frac{\rho}{q}\sum_{g=0}^{q-1}\zeta^g~\frac{p'(c+\rho\zeta^g)}{p(c+\rho\zeta^g)},
\end{equation}
for $\zeta=\zeta_q$ of (\ref{eqdndgnr}),
that is, $s_{h,q}(p,c,\rho)$ is the Cauchy sums $s_{h,q}(t,0,1)$ of the zeros of the polynomial $t(y)=p(\frac{y-c}{\rho})$  in the unit disc $D(0,1)$.
 
\begin{observation}\label{obcstch}  
Given a polynomial $p(x)$, a complex $c$, a positive $\rho$,
 a positive integer $q$, and the $q$-th roots of unity,  we can compute $s_{0,q}(p,c,\rho)$ by means of the evaluation of NIR$(\zeta^g)$ for $g=0,1,\dots,q-1$,  reduced to
evaluation of  $p$ at $2q$ 
points, and in addition
performing $q+1$ divisions, $2q-1$ multiplications, and $2q-1$ additions.
\end{observation}  
Refer to Alg. \ref{algexcl} 
applied to the polynomial $t(y)$ of (\ref{eqshft}) for an isolation parameter $\theta>1$ 
and $q=\lfloor\log_{\theta}(4d+2)\rfloor$   
 as {\em  Alg. \ref{algexcl}$_{c,\rho,q}$} and also as {\em  Alg. \ref{algexcl} applied to  $\theta$-isolated circle} $C(c,\rho)$ for $2q$  evaluation points. Combine Observations \ref{obscchy} and \ref{obcstch} to extend Observation \ref{thtsts} as follows.


\begin{proposition}\label{thtsts1}
(i) Alg. \ref{algexcl}$_{c,\rho,q}$ evaluates polynomial $p(x)$ at $q=\lfloor\log_{\theta}(4d+2)\rfloor$ points $x$ and in addition performs $5q-2$ ops. (ii) It outputs $\bar s_0=\#(D(c,\rho))$ if the circle $C(c,\rho)$ is $\theta$-isolated.
(iii) If the algorithm
outputs  $\bar s_0>0$, then $\#(D(c,\theta\rho))>0$.
\end{proposition}  

By applying  Thm. \ref{thpwrsm0} 
we can estimate the errors 
$|s_{h,q}(t,0,1)-s_h(t,0,1)|$ of the approximation of the $h$th power sums 
$s_h(t,0,1)$ of the roots of $t(x)$ in the unit disc $D(0,1)$.
Next we specify this
 for $h=1$.
\begin{proposition}\label{ths1ccjh} (Cf. \cite{IP22}.)
Let for $\theta>1$  a $\theta$-isolated disc 
$D(c,\rho)$ contain exactly $m$ roots.
Fix a positive integer $q$ and
define the Cauchy sum $s_{1,q}(p,c,\rho)$
of (\ref{equ7.12.00}). Then
$$|mc+\rho s_{1,q}(p,c,\rho) - s_1(p,c,\rho)|\le  \frac{\rho d \theta}{\theta^q-1}.$$
\end{proposition}
  

 
  

\section{Root-counting and $\ell$-tests}\label{mdls}
 
 
\subsection{Overview}\label{sovrv} 
 
 We seek fast  root-counters and 
 $\ell$-tests based on computing
 Cauchy sum $s_{0,q}$
 for a smaller integer $q$, but the output can deceive us  if we decrease $q$ below $m$
  (see Example \ref{ex1}). 
  

Extensive tests for $q$
of order $\log^2(d))$  
in   \cite{IP20,IP22} have never encountered  
such unlucky inputs, and  we only  provide some indirect formal support for this empirical behavior in Sec. \ref{seirndalg}.
 
 In the next two subsections we  support Las Vegas randomized root-counters and $\ell$-tests  that evaluate polynomial $p$
  at $O(m\log(d))$ points. 
    
 Combine this estimate with Cor. \ref{cosgmz} for $k$ of order $b$ and obtain  Las Vegas randomized solution of Problems 0, 1, 0$^*$, and 1$^*$ at the cost of evaluation of the  polynomial $p$
  at  $O(m^2b)$  points. 
  

\subsection{A basic randomized root-counter}\label{srndrtcnt} 


Next we fix a parameter $\gamma>1$ and  devise a root-counter for a disc  $D(c,\rho)$ with
 a fixed center $c$ 
(wlog we let $c=0$)
and a  radius $\rho$ sampled at random in a fixed small range. Then we prove that under such a  Las Vegas randomization
the test  fails with a probability at most $1/\gamma$.


 \begin{algorithm}\label{algrndrtcm} 
{\em Basic randomized root-counter.} 
  
\begin{description} 

 
INPUT: $\gamma\ge 1$ and  a $d$-th  degree black box polynomial $p$ having at most $m$ zeros in the  disc $D(0,\sqrt 2)$. 
 
 
INITIALIZATION: Sample a random value $w$
in the range $[0.2,0.4]$ under the uniform probability distribution in that range  and write $\rho:= 2^w$ and $q=\lfloor10m\gamma\log_{2}(4d+2)\rfloor$.

     
 COMPUTATIONS: Apply Alg. \ref{algexcl}$_{0,\rho,q}$.   
\end{description} 
\end{algorithm}

\begin{proposition}\label{thrndrtcm}
 Alg. \ref{algrndrtcm} queries  evaluation oracle at  $q=\lfloor 10m\gamma\log_{2}(4d+2)\rfloor$ locations	and outputs the value $\bar s_0=\#(D(0,\rho))$ with a probability at least $1-1/\gamma$.
 \end{proposition}

\begin{proof}
Readily verify the bound on the number of  query locations  claimed in the proposition.  \\
By assumption, the annulus 
$A(0,1,\sqrt 2)$ contains at most $m$
roots. \\ 
Hence at most $m$
root radii $r_j=2^{e_j}$, for $j=1,\dots,m'\le m$, lie in the range $[1,\sqrt 2]$, that is, $0\le e_j\le 0.5$ for at most $m'\le m$ integers $j$.\\
Fix $m'$  intervals, centered at $e_j$, for $j=1,\dots,m'$,  each of a length  at most $1/(5m'\gamma)$. Their  overall length is at most $1/(5\gamma)$. 
 Let  $\mathbb U$  denote their union.  \\
Sample a random $u$ in the range  $[0.2,0.4]$ under the uniform probability distribution in that range and notice that 
Probability$(u\in \mathbb U) \le 1/\gamma$.  \\
Hence  with a probability at least $1-1/\gamma$  the circle $C(0,\rho)$
is $\theta$-isolated for $\theta=2^{\frac{1}{10m\gamma }}$, in which case Alg. \ref{algrndrtcm} outputs 
$\bar s_0=s_0=\#(D(0,\rho))$
 by virtue of claim (ii) of Prop. \ref{thtsts1}.
\end{proof}

\subsection{Refining the bound on the error probability}\label{srfn} 

 When we apply Alg. \ref{algrndrtcm} $v$ times, the number of query locations of 
the evaluation oracle
grows linearly -- by a factor of $v$, while the bound on Las Vegas error probability of outputting errors decreases exponentially -- from $1/\gamma$ to $1/\gamma^{v}$. 

First let $\ell=1$ or $\ell=m$ and then again reduce the root radius approximation problem to the decision problem 
 of $\ell$-test. Namely, 
 {\em narrow  Alg. \ref{algrndrtcm} to an $\ell$-test} for a fixed $\ell$ in the range  $1\le\ell\le m$, so that with a probability at least $1-1/\gamma$ this $\ell$-test  outputs 0 if $\#(D(0,1))<\ell$   and outputs 1 if  $\#(D(0,2^{0.5}))\ge\ell$.  Moreover,  output 1 is certified by virtue of claim (iii) of  Prop. \ref{thtsts1} 
 if $\ell =1$;  similarly the output 0 is certified  if $\ell =m$.

 Next extend  the $\ell$-test by means of applying Alg. \ref{algrndrtcm}  
 for $v$ iid  random variables $w$ in the range $0.2\le w\le 0.4$. 
Call this 
 $\ell$-test 
{\bf Alg. \ref{algrndrtcm}v,$\ell$}.

Specify its output  $g$
for $\ell=1$ as follows: let
$g=1$  if  Alg. \ref{algrndrtcm}$v,1$ outputs 1  at least  once in its $v$ applications; otherwise let $g= 0$. Likewise, 
let $g=0$  if  Alg. \ref{algrndrtcm}$v,m$   outputs 0  at least once in its $v$ applications; otherwise let $g=1$.

 Recall Prop. \ref{thrndrtcm} and then readily verify the following proposition.
  
\begin{proposition}\label{thderndrtc}
(i) For two integers $1\le\ell\le m$ and  $v\ge 1$  and a real $\gamma \ge 1$, Alg. \ref{algrndrtcm}v,$\ell$ evaluates $p(x)$
at $\lfloor10m\gamma\log_{2}(4d+2)\rfloor v$ points; this evaluation involves more ops than the other stages of the algorithm,
even in the cases of the Mandelbrot and sparse  polynomials $p$ (cf. Observation \ref{obcstch}).
 (ii) Output 1 of Alg. \ref{algrndrtcm}v,1  and  the output 0 of Alg. \ref{algrndrtcm}v,$m$ are certified.
 (iii) Output 0 of Alg. \ref{algrndrtcm}v,1  and  output 1 of Alg. \ref{algrndrtcm}v,$m$ are correct with a probability at least $1-1/\gamma^{v}$.
 \end{proposition}

 
\noindent The proposition bounds the cost of $\ell$-tests for $\ell=1$
 (e/i tests) and  $\ell=m$.
 This is sufficient for our root-finding, but our next natural extension covers   $\ell$-tests for any  integer $\ell$ in the range $1<\ell <  m$.

Assign the values 0 or 1 to the output integers $g$ as follows: 
compute the average of the $v$ outputs in $v$ independent applications of the test and  output $g$ denoting the  integer closest to that average; if the average is 1/2,  write  
$g=0$ to break ties. 

This extension of Alg. \ref{algrndrtcm}v,$\ell$ is only interesting for
 $1<\ell<m$; otherwise its
cost/error estimates are inferior to those of Prop. \ref{thderndrtc}.
 
\begin{proposition}\label{thderndrtc1}
For an integer $\ell$  such that $1< \ell< m\le d$ and real $v\ge 1$ and $\gamma >1$,
  Algs. \ref{algrndrtcm}v,$\ell$ queries evaluation oracle at $\lfloor10m\gamma\log_{2}(4d+2)\rfloor v$ locations, and its
output is correct with a probability at least 
$1-(4/\gamma)^{v/2}$.
\end{proposition}
\begin{proof}
Since $D(0,1)\subset D(0,\sqrt 2)$,
we only consider
 the following  three cases:
 
(i) $\#(D(0,1))<\ell$ and $\#(D(0,2^{0.5}))\ge\ell$,

(ii) $\#(D(0,2^{0.5}))<\ell$, and then also $\#(D(0,1))<\ell$,

(iii) $\#(D(0,1))\ge\ell$, and then also  $\#(D(0,2^{0.5}))\ge\ell$.

In case (i) the output $g$  is always correct. 

In case (ii) the output 0 is correct, while the output 1 is wrong and  occurs if and only  if more than $v/2$ tests go wrong, by outputting 1.

Likewise, in case (iii)
the output 1 is correct, while the output 0 occurs if and only  if more than $v/2$ tests go wrong, by outputting 0.

In both cases (ii) and (iii) we deal with  
the binomial distribution having  failure probability $1/\gamma$, and so (cf. \cite{C46}) the probability $P$ of at least 
 $v/2$ failures among $v$ tests is at most $\sum_{j=\bar v}^v
\begin{pmatrix} v\\j \end{pmatrix}T_j$ for
$T_j= \gamma^{-j}(1-1/\gamma)^{v-j}\le\gamma^{-j}$. 

Therefore, $P\le
 \sum_{j=0}^v
 \begin{pmatrix} v\\j \end{pmatrix}=T_{\bar v}2^v\max_jT_j$. Substitute $
\sum_{j=0}^v
 \begin{pmatrix} v\\j \end{pmatrix}=2^v$ and $\bar v=\lceil v/2\rceil$ and obtain 
  $P<2^v T_{\bar v}\le 2^v/\gamma^{v/2}=(4/\gamma)^{v/2}$.
\end{proof}
 

\subsection{E/i test for an annulus}\label{seiann} 

The power sums of the
roots lying in an 
annulus $A(c,\rho,\rho')$ can be expressed as the difference  of the power sums of the roots lying in the two discs 
$D(c,\rho')$ and
$D(c,\rho)$, and our  algorithms and their analysis can be readily extended, based on Thm. \ref{thpwrsm0} supporting e/i tests for an annulus. 
We can immediately replace the $m$-test of 
Sec. \ref{smvia1} with such e/i tests.

 
\section{Estimation of the precision of computing}\label{sprcs}


\subsection{Doubling computational precision towards its a posteriori estimation}\label{sdblpst} 

As in Part II we can obtain near-optimal  posteriori estimates
for computational precision by recursively doubling it.
 
Next we obtain its {\em a priori  estimate} of $O(b\log(d))$ bits. 
The derivation is a little simpler than in Part II
due to observations in the next subsection.

\subsection{Reduction of a priori estimates to approximation of $x$NIR$(x)$ within $\frac{3}{8}$}\label{sfrstrd} 

  In our e/i tests
 we  approximate NIR$(x)=\frac{p'(x)}{p(x)}$
 with the  first order  divided difference of (\ref{eqNRNIR}); whp we only need to apply the tests to $\theta$-isolated  discs  
 $D(c,\rho)$ for $\theta-1>0$ of order $\frac{1}{m}$ and $\rho\ge \epsilon=1/2^b$. 

 Map (\ref{eqshft}) reduces such a test for $p(x)$ to a test in the  disc $D(0,1)$ for the polynomial $t(x)=p(\frac{x-c}{\rho})$ replacing $p(x)$ and decreases the error estimate  by a factor of $\rho$. Hence  map (\ref{eqshft}) increases  the precision bound by at most  a factor of $b$, and  we obtain the claimed precision bound in $O(b\log(d))$ in case of any disc $D(c,\rho)$ from the bound $O(\log(d))$
for the unit disc $D(0,1)$.
 
 The output of an e/i test is an integer, and  we only need to certify that the output error bound of Alg. \ref{algdblexpk} 
 is less than $\frac{1}{2}$. The error can be magnified by a factor of $v$ in $v$ applications of Alg. \ref{algdblexpk} but this factor is immaterial in our notation $\tilde O$. 

Our e/i tests compute the   $q$ values $\rho x\frac{p'(c+\rho x)}{p(c+\rho x)}=x\frac{t'(x)}{t(x)}$,
for $x=\zeta^g$,  $\zeta$ of (\ref{equ7.12.2}), and
$g=0,1,\dots,q-1$, sum 
these values,  and 
divide the sum by $q$.

By slightly abusing notation  write ${\rm NIR}(x):=\frac{t'(x)}{t(x)}$ rather than  ${\rm NIR}(x)=\frac{p'(x)}{p(x)}$. According to straightforward error analysis (cf. \cite[Part A of Ch.3]{BP94} and the references therein) rounding errors introduced by these additions and division are dominated, and we can ignore them;  then  it is sufficient to compute $\zeta^g{\rm NIR}(\zeta^g)$ within, say, $\frac{3}{8}$ for every $g$. Indeed, in that case
 the sum of $q$ error bounds is at most $\frac{3q}{8}$ and decreases  below $\frac{1}{2}$ in division by $q$.

\subsection{Precision of representation of NIR$(x)$ within $\frac{1}{8}$}\label{sprcrprs} 

   Recall  Eqn. (\ref{eqratio}) and represent $x{\rm NIR}(x)$ for $|x|=1$ within $\frac{1}{8}$ by using a precision of $O(\log(d))$ bits.  

\begin{proposition}\label{thnirprc} 
Write $C:=C(c,\rho)$ and assume that $|x|=1$ and the circle $C$
 is $\theta$-isolated for $\theta>1$. Then 
$|x{\rm NIR}(x)|\le
\frac{d}{1-1/\theta}=\frac{\theta d}{\theta-1}$. 
\end{proposition}
\begin{proof}
 Eqn. (\ref{eqratio}) implies that 
$|{\rm NIR}(x)|=|\sum_{j=1}^d\frac{1}{x-z_j}|\le\sum_{j=1}^d\frac{1}{|x-z_j|}$ where   $|x-z_j|\ge (1-1/\theta)\rho$ for 
$|z_j|<1$, while
$|x-z_j|\ge (\theta-1)\rho$ for $|z_j|>1$ since $i(C)\ge \theta$. Combine these bounds, write
$m:=\#(D(c,\rho))$, and obtain 
$|x{\rm NIR}(x)|\le \frac{m}{1-1/\theta}+ \frac{d-m}{\theta-1}\le \frac{d}{1-1/\theta}$ for $|x|=1$.
\end{proof}

\begin{corollary}\label{coprc0}
 One can represent 
 $x{\rm NIR}(x)$ within $\frac{1}{8}$ for $|x|=1$ by using a precision of 
$3+\lceil\log_2(\frac{d}{1-1/\theta})\rceil$
    bits, which is in $O(\log(d))$ provided that  $\frac{1}{\theta-1}=d^{O(1)}$.
 \end{corollary}
 
\subsection{What is left to estimate?}\label{srmnngestm} 
 
 Instead of $x{\rm NIR}(x)=\frac{x t'(x)}{t(x)}$ we actually
approximate $\frac{x t'(y)}{t(x)}$ where $|y-x|\le \delta$ and $t'(y)$ is equal to the divided difference of
  (\ref{eqNRNIR}).  Thus, we
shall increase the above error bound $\frac{1}{8}$ by adding to it the upper bounds   
  $\frac{1}{8}$ on
$\alpha:=|\frac{x t'(x)}{t(x)}-\frac{x t'(y)}{t(x)}|$ for $|x|=1$ and on
 the rounding error $\beta$ of computing $
x\frac{t(x)-t(x-\delta)}{\delta t(x)}$. 
Next we  estimate $\alpha$ in terms of $d$, $\theta$, and $\delta$ --
by extending the proof of 
\cite[Lemma 3.6]{LV16}, then   
 readily ensure that $\alpha<\frac{1}{8}$ by choosing $\delta$ with a precision $\log_2(1/|\delta|)$ of order $\log(d)$, and finally
  estimate $\beta$. 
    We simplify our estimates by dropping the factor $x$ where $|x|=1$.
  
\subsection{Precision of  approximation with divided difference}\label{sprcappr} 
  
  Recall the following lemma, which is \cite[Fact 3.5]{LV16}. 
  
\begin{lemma}\label{lefct35}  
  For  $t(x)=\prod_{j=1}^d(x-y_j)$ and a non-negative integer $j\le d$ it holds that
$$t^{(j)}(x)=j!~t(x)\sum_{S_{j,d}}~\prod_{j\in S_{j,d}} \frac{1}{x-y_j}$$
where the summation
$\sum_{S_{j,d}}$ is over all subsets $S_{j,d}$ of  the set $\{1,\dots,d\}$ having cardinality $j$. 
\end{lemma}  
  
\begin{proposition}\label{thdltprc01}
Under the assumptions of Prop. \ref{thnirprc} let 
$w=w(x):=\max_{j=1}\frac{1}{|x-y_j|}$. Then
\\
$|t(x)|\alpha=|\frac{t(x)-t(x-\delta)}{\delta}-t'(x)|\le |t(x)\frac{ (d ~w)^2\delta}{1-\delta~d~w}|$. 
\end{proposition} 
 
\begin{proof}
The
claimed bound on $|\xi'|$ follows
from  (\ref{eqratio}). It remains to 
apply the  first five lines of 
the proof of \cite[Lemma 3.6]{LV16} with $f$, $\alpha$, $n$, and $\xi$ replaced by
 $t$, $-\delta$, $d$, and $w$, s respectively. 
Namely, first obtain from Taylor's expansion that 
 $t(x)-t(x-\delta)=\sum_{j=0}^{\infty}\frac{\delta^j}{j!}t^{(j)}(x)-t(x)$.
 
 Substitute the expressions of Lemma \ref{lefct35} and obtain
 $$t(x)-t(x-\delta)=\delta t'(x)+\sum_{j=2}^{\infty} \delta^j t(x)\sum_{S_{j,d}}~\prod_{j\in S_{j,d}}\frac{1}{x-y_j}.$$
 Combine this equation with the assumed bounds on 
 $\frac{1}{|x-y_j|}$ 
 and deduce that 
 $$\Big|\frac{t(x)-t(x-\delta)}{\delta}-t'(x)\Big|\le \Big|\frac{f(x)}{\delta}\sum_{j=2}^{\infty} \delta^jw^j d^j\Big|\le 
 \Big|f(x)\frac{(d~w)^2~\delta}{1-\delta~d~w}\Big|.$$ 
 \end{proof}
In our estimates for $\alpha$ and $\delta$
 we  apply Prop. \ref{thnirprc} for
 $w\le \frac{\theta}
{\theta-1}$ and obtain
\begin{corollary}\label{coprc} 
 Under the assumptions of   Prop. \ref{thnirprc}, let  $w\le \frac{\theta}
{\theta-1}$. Then $$\alpha<1/8~{\rm for}~|\delta|\le \frac{(\theta-1)^2}{10d^2\theta^2}.$$
 \end{corollary}

\subsection{The impact of rounding errors}\label{simpctrnd}

To approximate  $|x{\rm NIR}_{c,\rho}(x)|$ within $\frac{3}{8}$ it remains to approximate $\beta$   within $\frac{1}{8}$.
  
 
 \begin{proposition}\label{therrnbl} 
Suppose that a black box oracle has computed  the values $t(y)$  within a relative error bounds $\nabla(y)$ 
for $y=x$, $y=x-\delta$, and
$\delta$ of Cor. \ref{coprc}.
  Then we can ensure  that $|\beta|<\frac{1}{8}$ by choosing 
 a proper $\nabla(y)$ of order $1/d^{O(1)}$,
 represented with a precision of $\log(\frac{1}{\nabla(y)})=O(\log(d))$ bits.
\end{proposition}

\begin{proof}
Write $\beta\delta=\frac{t(x-\delta)(1+\nabla(x-\delta))}{t(x)(1+\nabla(x))}-\frac{t(x-\delta)}{t(x)}=\frac{t(x-\delta)}{t(x)}~\frac{\nabla(x-\delta)-\nabla(x)}{1+\nabla(x)}$.

 Taylor-Lagrange's formula  implies that $\frac
{t(x-\delta)}{t(x)}=1+\frac{t(x-\delta)-t(x)}{t(x)}=1+\delta\frac{t'(u)}{t(x)}$, 
for $u\in[x-\delta,x]$. Hence $|\frac
{t(x-\delta)}{t(x)}|\le 1+|(1+\alpha')\delta|$ for $\alpha'=\frac{t'(u)-t(x)}{t(x)}$. 

By extending Cor. \ref{coprc}  obtain 
$|\alpha'|\le \frac{1}{8}$. 

Hence 
$|\frac{t(x-\delta)}{t(x)}|\le w:= 1+
\frac{9|\delta|}{8}$,
while $|\beta\delta|\le w\frac{|\nabla(x)|+|\nabla(x-\delta)|}{1-|\nabla(x)|}$.

Now choose $\nabla(y)$ such that $|\nabla(y)|\le\frac{|\delta|}{17w}$ for $y=x$
and $y=x-\delta$, and so $ w(|\nabla(x)|+|\nabla(x-\delta)|)\le \frac{2}{17\delta}$.
 Then verify that
$\log(\frac{1}{|\nabla(y)|})=O(\log(d))$, $1-|\nabla(x)|\ge 1-\frac{|\delta|}{17w}>\frac{16}{17}$, and $|\delta\beta|< \frac{|\delta|}{8}$.                                                             
\end{proof} 

By combining Cors. \ref{coprc0} and \ref{coprc} and Prop. \ref{therrnbl}  obtain
\begin{corollary}\label{coprcall}
Suppose that a black box oracle can output   approximations of a
$d$-th degree polynomial $t(x)$  with any relative error in $1/d^{O(1)}$ for  $|x|\le 1$ and let the unit disc be $\theta$-isolated for $t(x)$ and
for  $\frac{1}{\theta-1}=d^{O(1)}$. Then we can perform   e/i test for $t(x)$ on the unit disc by  using a precision of $O(b\log(d))$ bits.
\end{corollary} 
  


 
\section{Bit operation  cost of Cauchy  e/i tests and root-finding}\label{sblnrtfnd}


We omit extension of the precision bound to the bit complexity estimate for the eigenvalue problem because  
 it can be verbatim the same as in Part II.
 
\subsection{Our goal and a basic theorem}\label{sglbscthm}
  
Next  we  estimate  the bit operation complexity of our root-finders by combining them  with  fast multipoint evaluation  of a polynomial by means of the algorithm of  \cite{F72,MB72} 
  non-trivially extended in \cite{K98,M21} to support the following theorem:
 
\begin{theorem}\label{thkrrn} 
 Given a positive $b$, the coefficients of  a polynomial $t(x):=\sum_{i=0}^d t_ix^i$ such that   $||t(x)||_1=\sum_{i=0}^d|t_i|\le 2^{\tau}$, and $q$ complex points $x_{1},\dots,x_q$  in the unit disc $D(0,1)$, one can approximate the values $t(x_{1}),\dots,t(x_q)$
within  $1/2^b$ by using $\tilde O((d+q)(b+\tau))$ bit operations. 

\end{theorem}  
\begin{proof}  
For $q=d$ this is 
\cite[Thm. 2]{M21}, based on using novel data structure for hyperbolic approximation of $p$, proposed and analyzed in \cite{M21}. 
Trivially extend this theorem to the cases where $q>d$ (by  considering $t(x)$ a 
polynomial of degree $q$ with $q-d$ leading coefficients 0)
and  where $q<d$ (by removing $d-q$ points $x_i$, with no increase of the complexity bound).
\end{proof}  

 We only apply this theorem for $\tau=0$, in which case its 
bit operation bound turns into  
\begin{equation}
\label{eqmrz}
\mathbb B_{\rm Moroz}=\tilde O((d+q)b)
\end{equation}
 
 For comparison, the previous State of the Art can be represented by the bound
 \begin{equation}
 \label{eqkrrnn}
\mathbb B_{\rm Kirr}=\tilde O((d+q)(b+q)),
\end{equation}
due to  \cite[Thm. 3.9, Alg. 5.3, and Appendix A.3]{K98} based on combining Kirrinnis's 
and Sch{\"o}nhage's study  of  
the bit operation cost
 of basic polynomial computations (cf. \cite{S82}) with  the algorithm  of \cite{F72,MB72}.\footnote{Precision of computing
 required in the latter 
 algorithm grows fast as $d$ increases; e.g.,   the IEEE standard double precision of the customary numerical computations is not enough to support that algorithm already for $d>50$; for $b=O(\log(d))$ one can fix this deficiency by applying FMM (see Remark  
  \ref{refmm}) to the equivalent task of multiplication of Vandermonde matrix by a vector. One can, however, support  bounds of Eqns. (\ref{eqmrz}) and (\ref{eqkrrnn}) on the overall bit operation complexity  by computing with multiple precision for larger values $b$.}  Moroz obtained 
 \cite[Thm. 2]{M21} by applying bound (\ref{eqkrrnn}) to his lower degree polynomials approximating $p(x)$.

\subsection{Acceleration with fast multipoint polynomial approximation}\label{srtfndmltpnt} 
   
 Given a complex $c$, a pair of positive $b$ and $\rho$,  and
 $d+1$ coefficients of a polynomial $p=p(x)$ of (\ref{eqpoly}) such that the disc $D(c,\rho)$ is isolated and contains precisely $m$ roots, we approximate all these roots within $R/2^b$
 for $R=|c|+\rho$
by applying subdivision root-finding iterations 
with e/i tests of Sec. \ref{mdls}.  Then  every subdivision step   
  is reduced  to application of  e/i tests
 to $\bar m=O(m)$ discs $D(c_{\lambda}, \rho_{\lambda})$ for  $\rho_{\lambda}\ge R/2^b$ and all $\lambda$. 
 
 Such a test, based on our e/i test of Sec. \ref{mdls},  amounts  to approximation of  $p(x)$  at $O(m\log(d))$ equally spaced points $x$ on each of
 $v$ circles sharing the  center $c_{\lambda}$ and having
 radii in the range  $[\rho_{\lambda}/\sigma,\rho_{\lambda}\sigma]$
for $1<\sigma<\sqrt 2$.

 We only use 
 Thm. \ref{thkrrn} for $\tau=0$ to cover the case where
a polynomial $t(x)$ 
with $||t(x)||_1\le 1$
is evaluated at points $x_j\in D(0,1)$. To ensure these assumptions  
 we first scale the variable $x\mapsto
Rx$ to map the discs $D(c_{\lambda},\rho_{\lambda}\sigma)\mapsto D(\bar c_{\lambda},\bar\rho_{\lambda})\subseteq D(0,1)$, for
 $\bar c_{\lambda}= c_{\lambda}/R$,
 $\bar\rho_{\lambda}=\rho_{\lambda}/R\ge 1/2^b$,
 and all $\lambda$.
 
 Then we write $t(x):= p(Rx)/\psi$ for $\psi=O(R^d ||p(x)||_1)$
 such that
$||t(x)||_1=1$ and     
approximate $t(x)$ within $1/2^{\bar b}$ or equivalently within the relative error bound  $1/(|t(x)|2^{\bar b})$ at $\bar q:=qv$ points $x$ in  $\bar m$ e/i tests.\footnote{Scaling $p$
and $x$ above is done once and for all e/i tests; to compensate for this we only need to change our bound on the output  precision $b$ by $\lceil|\log_2R|\rceil$ bits.}

Due to Cor. \ref{coprcall} it is sufficient to ensure the bound 
$1/(|t(x)|2^{\bar b})=1/d^{O(1)}$; we  deduce it by choosing any $\bar b\ge \log(\frac{1}{|t(x)|})+O(\log(d))$. 

In   our e/i tests we only need to approximate  $t(x)$ at the points $x$ lying on $\theta_{\lambda}$-isolated circles 
 where
 $\frac{1}{\theta_{\lambda}-1}=O(m)$. Map the disc $D(\bar c_{\lambda},\bar\rho_{\lambda})$ into the unit disc $D(0,1)$. This does not change isolation of the boundary circle and may increase the norm $||t(x)||_1$
 by  at most  factor of $2^{b+d}$ because $|\bar c_{\lambda}|\le 1$ and 
 $\bar\rho_{\lambda}\ge 1/2^b$. Now substitute 
$||t(x)||_1\le 2^{b+d}$
and 
$\log(\frac{||t(x)||_1}{|t(x)|})=O(d\log(m))$ 
(cf. \cite[Eqn.(9.4)]{S82})\footnote{This bound is readily deduced from the following non-trivial result
\cite[Thm. 4.2]{S82}: Let
$t(x)=f\prod_{j=1}^m(x-z_j)\prod_{j=m+1}^d(\frac{x}{z_{j}}-1)$, $||t(x)||_1=1$,
$|z_j|<1$ for $j\le m$, and
$|z_j|>1$ for $j> m$. Then  
$|f|\ge 1/2^d$.}
and obtain that
$\bar b=\log(\frac{1}{|t(x)|})+O(\log(d))=O(d\log(m)+b)$.
\begin{remark}\label{relbt}
For $m=d$ the above bound on $\log(\frac{||t(x)||_1}{|t(x)|})$ of \cite[Eqn.(9.4)]{S82}
 is sharp up to a constant factor and is reached for the polynomial $t(x)=(x+1-\frac{1}{m})^d/(2-\frac{1}{m})^d$, $\rho_{\lambda}=1$, and $x=-1$. For $m<d$  \cite[Thm.    
4.5]{S82} implies a little stronger   bounds, namely,  $\log(\frac{||t(x)||_1}{|t(x)|})=O(d+ m\log(m))$ and hence
$\bar b=O(m\log(m)+d+b)$
 because all discs $D(\bar c_{\lambda},\bar\rho_{\lambda})$ lie in  $\theta$-isolated unit disc $D(0,1)$ for $\theta-1$
exceeding a positive constant. 
\end{remark} 

\subsection{The bit operation complexity estimates}\label{sfblncsteirt}

 Apply  Thm. \ref{thkrrn} 
with $q$ and $b$ replaced  
by $\bar q=O(m^2)$ and $\bar b=O(b+d+m\log(m))$, respectively,\footnote{Under the Random Root Model of Sec. \ref{scmplests} we can whp decrease this bound by a factor of $m$  and then extend this decrease  to the related subsequent estimates.}  and obtain that the bit operation cost of our $O(m)$ e/i tests  of Sec. \ref{mdls} at any fixed subdivision step is in 
 $\tilde O((b+d+m)m^2)$
  for $1\le m\le d$. This implies the upper estimate
\begin{equation}\label{eqbrts}    
 \mathbb B_{\rm roots}= \tilde O((b+d)m^2b)
\end{equation}
for the overall bit operation cost  of our $O(b)$ subdivision steps, which includes the time-complexity of the evaluation 
oracle. Based on \cite[Alg. 5.3]{K98} 
and Eqn. (\ref{eqkrrnn})  obtain  a little larger  upper bound on the bit operation cost of our subdivision root-finders:
\begin{equation}\label{eqbrts'} 
\mathbb B_{\rm roots}'= \tilde O((m^2 + d + b)(m^2 + d)b).
\end{equation}

\medskip

{\Large \bf PART IV: Acceleration of Root-Finding With Compression}
\medskip
\medskip

In this part of the paper  we decrease the number 
  of subdivision steps of our root-finders  of Parts II and III by a factor of  $b/\log(b)$, thus  arriving at our Main Theorems 1 -- 4.

  In Sec. \ref{scmpr} we extend the outline of this acceleration from Sec. \ref{saccIV}. It relies on an algorithm for compression of  an isolated  disc, which we devise and analyze also  in Sec. \ref{scmpr}.
  The algorithm involves
estimation of the $\ell$th smallest root radius, which is of independent interest and is covered in Sec. \ref{sextrrk}. It is reduced to recursive application of  an $\ell$-test; we present it  in  Sec. \ref{mdls} for any positive $\ell\le d$ and any input disc, but in Sec. \ref{smvia1} we specify  simpler tests that are sufficient in application to compression.
In Sec.
\ref{sovrllcmpl}
we estimate the overall
complexity of the resulting root-finder, which intertwines at most $2m-2$ compression steps with e/i tests, whose number we estimate
in Sec. \ref{scntei}. In Sec.
\ref{sextr} we further accelerate polynomial root-finding for a large class of inputs.
  
\section{Estimation of the $\ell$-th smallest root radius}\label{sextrrk}  
   In this section we
  narrow a fixed range   $[0,R]$ for the
  $\ell$-th smallest root radius $r_{d-\ell+1}(c)$ from a fixed complex point $c$. 
  To simplify our exposition we let $R=1$ and $c=0$, wlog due to
  (\ref{eqshft}). 
\medskip

{\bf Problem RR$_{\ell}$.} Compute a range $[\rho_-,\rho_+]$
such that $0\le \rho_-\le r_{d-\ell+1}\le \rho_+$ and $\rho_+\le \epsilon=1/2^b$ or 
$\rho_+\le\beta\rho_-$
for fixed $b>0$
and $\beta=2^{\phi}>1$,
  provided that we are given a positive integer $\ell\le d$,  initial range $[0,1]$ for $r_{d-\ell+1}$, and a $\sigma$-soft $\ell$-test for a fixed $\sigma>1$.
  
  We first use this test to check whether
 $r_{d-\ell+1}\le 1/2^b$. If so, we are  done. Otherwise we seek
  $r_{d-\ell+1}$ in the range $[\frac{1}{2^b\sigma},1]$. We apply binary search, said to be {\em bisection of logarithms} or {\em BoL},  to the range $[-\bar b,0]=[-\log_2(\frac{\sigma}{\epsilon}),\log_2(1)]$, for $\bar b=\log_2(\frac{\sigma}{\epsilon})$,
  to approximate $\log_2(r_{d-\ell+1})$.  
 
Since the test is soft, so is BoL
as well:
for   a range $[b_-,b_+]$ with  midpoint $\widehat b=0.5(b_++b_-)$ regular rigid bisection continues the  search in the range  $[b_-,\widehat b]$ or $[\widehat b,b_+]$, but our  soft BoL  continues it in the range  $[b_-,(1+\alpha)\widehat b]$ or $[(1-\alpha)\widehat b,b_+]$ for a fixed  small positive $\alpha$.
Accordingly we slightly
 increase the range of search for 
$r_{d-\ell+1}$ to
$[-\bar b,0]$
 for $\bar b:=\log_2(\frac{\sigma}{\epsilon})$.
 BoL still converges linearly and just  a little slower than the regular rigid bisection: the approximation error decreases by a factor of  $\frac{2}{1+\alpha}$ rather than  by twice.\footnote{If we applied BoL to the range $[0,1]$,
   we would have to bisect the infinite range $[-\infty,0]$, but we  deal with  the narrower  range  $[-b,0]$ since we  have  narrowed the range $[0,1]$ to $[\epsilon,1]$.} 
      
\begin{algorithm}\label{algdblexpk}  
{\em Bracketing the $\ell$-th smallest root radius  based on $\ell$-test.}  
 \begin{description}
 INPUT:  An integer $\ell$ such that   $1\le \ell\le d$ and  $\#(D(0,1))\le \ell$, $\sigma=1+\alpha$ for  $0<\alpha< 1/2$, 
  a $\sigma$-soft $\ell$-test$_{0,\rho}$
  for a positive $\rho$ (see  Sec. \ref{seiann} for such tests),  
 and two  tolerance values  $\beta:=2^{\phi}>1$ 
and 
$\epsilon=1/2^{b}$,   
  $b \ge 1$. 
  
   OUTPUT: A  solution
 $[\rho_-,\rho_+]$
 of Problem RR.
 
INITIALIZATION: Write $\bar b:=
\log_2(\frac{\sigma}{\epsilon})$,
$b_-:=-\bar b$, and $b_+:=0$.

COMPUTATIONS: Apply   
$\ell$-test$_{0,\frac{\epsilon}{\sigma}}$. Stop and output 
  $[\rho_-,\rho_+]= [0,\epsilon]$ if it outputs 1.
Otherwise  
apply  $\ell$-test$_{0,\rho}$ for  
\begin{equation}\label{eqtrhob}
\rho_-=2^{b_-},~\rho_+=2^{b_+},~ \rho:=2^{\widehat b},~{\rm and}~\widehat b:=\lceil(b_-+b_+)/2\rceil.
\end{equation}
 Write $b_-:=(1-\alpha) \widehat b$
if the $\ell$-test outputs 0; write $b_+:=(1+\alpha) \widehat b$
if the test outputs 1.
Then update $\widehat b$ and  $\rho$ 
according to (\ref{eqtrhob}).

Continue recursively: 
stop and output the current range 
 $[\rho_-,\rho_+]=[2^{b_-},2^{b_+}]$ if $b_+-b_-\le \phi$, that is, if $\rho_+\le \beta\rho_-$; otherwise
 update  
  $\widehat b$  and  $\rho$ and reapply the computations. 
\end{description} 
\end{algorithm}

The disc $D(0,\epsilon)$
contains at least $\ell$ roots, that is, the range $[0,\epsilon]$ brackets the $\ell$th smallest root radius $r_{d-\ell+1}$ if the $\sigma$-soft $\ell$-test applied to the disc $D(0,\frac{\epsilon}{\sigma})$ outputs 1. 
Otherwise, the disc $D(0,\frac{\epsilon}{\sigma})$ contains less than $\ell$ roots, and so
the root radius $r_{d-\ell+1}$ lies in  the range  $[\frac{\epsilon}{\sigma},1]$, in which case we continue our search. 

Readily verify that our policy of recursively updating the range  $[b_-,b_+]$ maintains this property throughout and that updating stops where $b_+-b_-\le \phi$, that is, where $\rho_+\le \beta\rho_-$. This implies {\em correctness} of Alg. \ref{algdblexpk}. 

Let us estimate the number 
  $\mathbb I$ of the invocations of 
  an $\ell$-test. 
 
\begin{proposition}\label{thintlrng}
Alg. \ref{algdblexpk} invokes 
an $\ell$-test $\mathbb I$ times where
$ \mathbb I\le \lceil\log_{\nu}(\frac{\bar b}{\phi})\rceil+1$
for $\nu=\frac{2}{1+\alpha}$, and so 
$ \mathbb I=O(\log(b))$ provided that 
$\nu-1$, $\sigma-1$, and $\phi$ are positive constants.
\end{proposition}
\begin{proof} 
  Alg. \ref{algdblexpk} stops if the    $\ell$-test   outputs 1. 
 Otherwise the test decreases the length $\bar b=b_+-b_-$ of its input range $[b_-,b_+]$ for the exponent to $0.5(1+\alpha)\bar b$. 
   $i$ 
applications change it into
$\frac{(1+\alpha)^i}
{2^i}\bar b$
 for $i=1,2,\dots$, and
 the proposition
  follows.
\end{proof}

\begin{remark}\label{retrhob}  
 In the  case where we know that $\#(D(0,\rho))\ge \ell$ for a fixed  $\rho$ in the range $[\epsilon,1]$, 
 we only need to refine the initial range $[\rho,1]$  for the root radius $r_{d-\ell+1}$.
\end{remark}   
 
Clearly, $\mathbb I$ applications of $\ell$-test increase the number of evaluation points for the ratio
$\frac{p'(x)}{p(x)}$ by at most 
a factor of $\mathbb I$, and if
the $\ell$-test is randomized,
then the error probability also grows by at most a factor of $\mathbb I$.
Let us specify these estimates for Alg. \ref{algdblexpk}.   

\begin{corollary}\label{corrcst}
Let precisely $\ell$ zeros of a black box polynomial $p(x)$ of a degree $d$  lie in the unit disc, that is, let $\#(D(0,1))=\ell$,
and let $i(D(0,1))\ge\sigma$, for $\sigma-1$ exceeding a positive constant. 
Given  two positive 
tolerance 
values $\epsilon=1/2^b$ and $\phi$ and a black box $\sigma$-soft 
$\ell$-test, 
which evaluates the polynomial $p(x)$ at 
 $\mathcal {NR}$  points, Alg. \ref{algdblexpk} approximates the $\ell$-th smallest
root radius $|z_j|$ within the relative error $2^{\phi}$
or determines that  $|z_j|\le \epsilon$,  at the cost of evaluation of polynomial $p(x)$ at
$\mathbb I \cdot\mathcal {NR}$ points for $\mathbb I$  estimated in Prop. \ref{thintlrng},
$\mathbb I=O(\log(b))$
provided that 
$\nu-1$ and $\phi$ are positive constants. 
\end{corollary}

\section{Compression of an isolated  disc}\label{scmpr}


\subsection{A motivation: compact components in the union  of suspect squares}\label{scmptrcmpcmpn}  
  
As in Sec. \ref{saccIV},
 partition the union of suspect squares at every subdivision iteration into components and associate breaking
(partition of)  components into sub-components with  component tree,
which has at most $2m-2$ edges.
 Next we will bound the numbers of subdivision steps (and hence of e/i tests)  at and between partitions of  components. 
 
 We call a pair of suspect squares {\em non-adjacent}  if they share no vertices.
Call a component {\em compact} if it contains   
just a single non-adjacent suspect square, that is, consists of at most four suspect squares sharing a  vertex (see Fig. \ref{figcmpc}).
\begin{figure}
\centering
\resizebox{!}{0.15\textheight}{
\begin{tikzpicture}
\draw (0, 0) circle (0.74);
\draw (0, 0) circle (1.42);
\draw (-2, 2) -- (2,2) -- (2, -2) -- (-2, -2) -- (-2, 2);
\draw (-0.5, 0.5) -- (0.5,0.5) -- (0.5, -0.5) -- (-0.5, -0.5) -- (-0.5, 0.5);
\draw (-1, 2) -- (-1, -2);
\draw (1, 2) -- (1, -2);
\draw (2, -1) -- (-2, -1);
\draw (2, 1) -- (-2, 1);
\draw (0, 2) -- (0, -2);
\draw (-2, 0) -- (2, 0);
\draw[red] (-0.1, 0.15) node [left]{$*$}; 
\draw[blue] (-0.1, -0.15) node [left]{$*$};
\draw[red] (-0.1, 0.15) node [right]{$*$}; 
\draw[blue] (-0.1, -0.15) node [right]{$*$}; 
\end{tikzpicture}}
\caption{Roots (asterisks)  define compact components. A subdivision step halves their diameters  and at least doubles isolation of their minimal covering squares and discs.}\label{figcmpc}
\end{figure}

 \begin{figure}
\centering
\begin{tikzpicture}
         \draw (0, 2) circle (1.8cm) ;
         \draw (0, -2) circle (1.8cm);
         \draw (-3, 3) -- (3,3) -- (3, -3) -- (-3, -3) -- (-3, 3);
         \draw (-1, 3) -- (-1, -3);
         \draw (1, 3) -- (1, -3);
         \draw (-3, 1) -- (3, 1);
         \draw (-3, -1) -- (3, -1);
         
         \node at (0,2) [circle,fill,inner sep=1.5pt]{};
         \node at (0,-2) [circle,fill,inner sep=1.5pt]{};
\end{tikzpicture}
\caption{Two discs superscribing two non-adjacent suspect squares 
are not minimal but
do not overlap.}
\label{fignnadjdscs}
\end{figure}

\begin{observation}\label{obcmp} (See Fig. \ref{fignnadjdscs}.) The $\sigma$-covers of two non-adjacent suspect squares
 are separated by at least the distance $a\Delta$ for the side length $\Delta$ and {\em separation  coefficient} 
 $$a:=2-\sigma \sqrt 2>0~{\rm for}~1<\sigma< \sqrt 2.$$     
\end{observation}
 
\begin{lemma}\label{eqbrkcmp}
 Let $|z_g-z_h|\ge \sigma\Delta$, for 
a pair of roots $z_g$ 
and $z_h$, the softness  $\sigma$ of e/i tests, and the side length $\Delta$ of a suspect square, and let $i>0.5+ \log_2(m)$. Then in at most  
 $i$ subdivision steps the roots $x$ and $y$  lie 
in distinct components.  
\end{lemma}
\begin{proof}
 Two roots  $z_g$ and $z_h$ lying in the same component  at the $i$th subdivision step are connected by a chain of $\sigma$-covers of  at most $m$ suspect squares, each having  diameter
$\sigma\Delta/2^{i-0.5}$. 
Hence
$$|z_g-z_h|\le m \sigma\Delta/2^{i-0.5},~  
i\le \log_2\Big(\frac{m\sigma\Delta}{|z_g-z_h|}\Big)+0.5.$$

 Substitute the bound $|z_g-z_h|\ge \sigma\Delta$ and deduce that $m\le 2^{i-0.5}$ or equivalently $i\le 0.5+ \log_2(m)$ if   
the roots $z_g$ and $z_h$ are connected by a chain of suspect squares, that is, lie in the same component. This implies the lemma.
\end{proof}

\begin{corollary}\label{cobrkcmp}
 A component of suspect squares is broken in at most $1+\lfloor0.5+ \log_2(m)\rfloor$ successive subdivision steps unless it is
compact. 
\end{corollary}
  
\begin{observation}\label{obcmpisl}
(i) A compact  component is surrounded by a frame of discarded squares (see Fig. \ref{figcmpc})  and hence is covered by a 2-isolated square.
(ii) If it stays
  compact in $h$ successive 
  subdivision steps, then its  minimal covering square and disc become $2^{h+1}$-isolated and $2^h\sqrt 2$-isolated, respectively. 
\end{observation}
\begin{proof}
Readily verify claim (i). Claim (ii) follows because a 
 subdivision step can only thicken the frame of 
discarded squares while decreasing at least by twice the diameter of any component that 
stays compact.
\end{proof} 

Cor. \ref{cobrkcmp} bounds the number of subdivision steps
between partitions of non-compact components, but compact
components can be processed very slowly (see Fig. \ref{fig4}). 
For a remedy,  we intertwine subdivision process
 with compression of  MCD of a component that stays compact in three subdivision steps, in which case the MCD is $4\sqrt 2$-isolated by virtue of Observation 
  \ref{obcmpisl}.

Next we elaborate upon such compression of MCDs.

\subsection{The compression problem and flowchart of its solution}\label{scmpralg} 
 
{\bf Problem COMPR.}  
INPUT: a  positive   $\epsilon\le 1$, a constant $\eta>0$,  a black box polynomial $p(x)$ of a degree $d$,  and an isolated  disc $D$ with a non-empty root set.

   OUTPUT: an integer $m$, a complex $c\in D$, and positive $\rho$ and $\eta$ such that 
 $\#(D(c,\rho))=\#(D)$ and the disc $D(c,\rho)$ is $\eta$-rigid
 unless $\rho\le\epsilon$.
  \medskip
 
{\bf Flowchart of a compression algorithm.} To solve Problem COMPR
we successively compute
(a) $m=\#(D)$, (b)  
 complex point $c$ lying in or near the  minimal covering disc  of the root set $X(D)$, and (c)  positive $\rho$ and and $\eta$ such that 
 $\#(D(c,\rho))=m$ 
 and the disc $D(c,\rho)$  is $\eta$-rigid
 unless $\rho\le\epsilon$.
 
 Wlog assume that $D$ is a unit disc $D(0,1)$ and 
{\em (a) compute $m=\#(D)$ for  isolated  unit disc} $D:=D(0,1)$   by using $O(\log(d))$  
 queries of an evaluation oracle (see Observations \ref{thtsts} and \ref{obcstch}). 
 
  Next we elaborate upon the algorithms for
 stages (b) and (c). 
 
\subsection{A point in or near the  minimal covering disc for the $d$ roots}\label{svrntcmpr}

 \begin{definition}\label{desstris}
  A disc is said to be {\em strongly isolated} 
  if it is $d^{\nu}$-isolated for
 a sufficiently large constant $\nu$ of our choice. In particular
 every disc containing all $d$ roots is strongly isolated.
\end{definition}
   
 Next we successively compress the unit disc  $D=D(0,1)$ in the s where it contains  $d$  roots,  is  strongly isolated,  and  is just isolated.
  
Schr{\"o}der's  iterations
  \cite{S870},  
\begin{equation}\label{eqSCHR} 
z=y-m~{\rm NR}(y),
\end{equation}
 fast
 converge to a root of multiplicity $m$  \cite[Sec. 5.4]{M07} and for $m=d$ 
 outputs a point in or near the MSD  of the 
 the set of the $d$ roots 
  if point $y$  lies far enough from that set:

\begin{theorem}\label{thqinch}
    Write $D:=D(0,1)$
  and
    $\Delta:=\Delta(X(D))$, let $\#(D)=d$,
      fix $y$ such that $|y|>2$, let $z:=y-d~{\rm NR(y)}$        and  $\bar x\in MCD(D)$.
  Then $|z-\bar x|\le \alpha~\Delta$ for
 $\alpha:= \frac{(|y|-1)^2}{(y|-2)^2}=(1+\frac{1}{(y|-2})^2$,
 and so $\alpha=2.25$ for $|y|=4$, $\alpha=25/16$ for $|y|=6$, and $\alpha\mapsto 1$ as $|y|\mapsto \infty$.
  \end{theorem}  
 \begin{proof} 
Recall that 
$\bar x,z_1,\dots,z_d \in X(D)$, $|z_j|\le 1$  and
$|\bar x-z_j|\le \Delta$ for 
$j=1,\dots,d$.
Hence (cf. (\ref{eqratio}))  
$$\Big|{\rm NIR}(y)-\frac{d}{y-\bar x}\Big|\le \sum_{j=1}^d\Big|\frac{1}{y-z_j}-\frac{1}{y-\bar x}\Big|\le
\sum_{j=1}^d\frac{|\bar x-z_j|}{|y-\bar x|\cdot|y-z_j|}\le \frac{\Delta d}{(|y|-1)^2}.$$ 
 Therefore, 
\begin{equation}\label{eqetadlt} 
  |u-w|\le \beta~{\rm for}~u:=\frac{{\rm NIR}(y)}{d},~w:=\frac{1}{y-\bar x},~{\rm and}~\beta:= \frac{\Delta}{(|y|-1)^2}, 
  \end{equation}
   and so
\begin{equation}\label{eqx-z}  
 |z-\bar x|=
  \Big|\bar x-y+\frac{d}{{\rm NIR}(y)}\Big|=\Big|\frac{1}{u}-\frac{1}{w}\Big|=\frac{|u-w|}{|uw|}\le \frac{\beta d}{|{\rm NIR}(y)|}|y-\bar x|.
  \end{equation}
       
 Deduce from (\ref{eqratio}) and the bounds $\max_{j=1}^d|z_j|\le 1<|y|$   that
 $$|{\rm NIR}(y)|=\Big|\sum_{j=1}^d\frac{1}{y- z_j}\Big|\ge d\max_{j=1}^{d}\frac{1}{|y- z_j|}\ge \frac{d}{|y|-1}.$$

Substitute this bound into   Eqn. (\ref{eqx-z}) and obtain that

  $$|z-\bar x|\le \beta~||y|-1|\cdot |y-\bar x|$$ 
  Substitute $\beta:= \frac{\Delta}{(|y|-\Delta)^2}$ of (\ref{eqetadlt}) and obtain 
   $$|\bar x-z|\le \alpha' \Delta~{\rm for}~
\alpha':=\frac{(|y|-1)|y-\bar x|}{(|y|-\Delta)^2}.$$
  Substitute the bounds  $\Delta\le  
  2$ and $|\bar x|\le 1$ and obtain the theorem.
   \end{proof} 
 
\begin{corollary}\label{cortrd1m0}
Under the assumptions of Thm.  \ref{eqetadlt} the disc   
$D(z,r_1(z))$ is $\frac{1}{2+2\alpha}$-rigid. 
 \end{corollary}\begin{proof}
 Verify the claim in the case where
 $x_d=0$, $x_1=\Delta=1$, e.g., $0\le x_j\le 1$ for all $j$, and $z=1+\alpha$ 
 and notice that this is 
  a worst case input.
  \end{proof}
  
\subsection{Compression algorithm for a disc containing $d$ roots}\label{scmpralgd}
 
Next we compress a disc containing $d$ roots  towards their  minimal covering disc. 

\begin{algorithm}\label{algcmprsch} {\em  $\epsilon$-compression 1.} 
\begin{description}   

INPUT: a  positive   $\epsilon=1/2^b\le 1$ and a black box polynomial $p(x)$ of a degree $d$ such that
 \begin{equation}\label{eqdsc1}  
  \#(D(0,1))=d.
\end{equation}
   
   OUTPUT:  a complex $z$, and a positive $r$ such that  $\#(D(z,r))=d$ and  the disc $D(z,r)$
  is $\frac{2}{13}$-rigid unless $r\le\epsilon$.
   
   INITIALIZATION: Fix 
    $y\ge 4$.

COMPUTATIONS: \\
(i) Compute and output $z=y-d~{\rm NR}~(y)$. \\ (ii) Compute an approximation $\tilde r\ge 0$ to $r_1(z)$ by applying  Alg. \ref{algdblexpk} for  $\epsilon$ replaced by $\epsilon/4$. \\
Output  $z$ and choose $r$ to define $\epsilon$-compression
 into the disc 
 $D(z,r)$.
 
Namely,  output $r=\epsilon$ if 
 $\tilde r\le 3\epsilon/4$. \\
 Otherwise output $r=\tilde r+0.25 \epsilon$, then the disc 
 $D(z,r)$ is rigid.
 \end{description}
\end{algorithm}

{\em Correctness of the algorithm} is readily verified if Alg. \ref{algdblexpk} outputs a range
$[\rho_-,\rho_+]$ of length at most $\epsilon/4$, for in this case $r_1\le r=\tilde r+\epsilon/4\le \epsilon$,
and $\#(D(z,\epsilon))=d$.

Otherwise $1/\beta\le \tilde r/r_1\le \beta$ for a constant $\beta\ge 1$
of Alg. \ref{algdblexpk}, while 
the disc $D(x,r_1)$
is $\frac{1}{2\alpha+2}$-rigid 
 by virtue of Cor. \ref{cortrd1m0}
 where  $\alpha\le 2.25$
 by virtue of  Thm. \ref{thqinch} because $y\ge 4$ in our case. 
 It follows that
the disc 
$\bar D:=D(z,r_{1}(z))$
is $\frac{2}{13}$-rigid.  
 
\begin{remark}\label{realph} 
If we initialize Alg. \ref{algcmprsch} at a sufficiently large $y$, then the same argument 
 ensures that the disc $D(z,r_{1}(z))$ is $\eta$-isolated for any $\eta>4$  unless $\tilde r\le 3\epsilon/4$.
\end{remark}
 
This compression queries  
 evaluation oracle  twice  while performing  Schr{\"o}der's single iteration for  computing $z$ and otherwise runs at
the cost of performing Alg. \ref{algdblexpk}, estimated in Cor. \ref{corrcst}.

We decrease the latter cost for a large class 
of polynomials $p$ by choosing a larger tolerance $\epsilon$. Then we  yield desired 
compression if  the algorithm outputs $r$
exceeding $\epsilon/4$. Otherwise we  recursively repeat the computations for the tolerance recursively 
decreasing to $\epsilon/4$.
 
 \subsection{Compression of an isolated  disc}\label{secmprstris}  
 
If $m<d$ and  $i(D)\ge d^{\nu}$
for a sufficiently large constant $\nu$, then Eqn. (\ref{eqratio}) implies that the
$d-m$ external roots, lying outside the disc $D$, little affect  the proof of 
Thm. \ref{thqinch} 
and our whole study in the previous two subsections as long as we apply Schr{\"o}der's  iteration at a complex point $y$ lying far enough from both  disc $D$ and the circle $C(0,d^{\nu})$.

The MCD of a compact component can be  isolated but not strongly.  Then, however,  we 
 apply subdivision steps to that
 component and in order of  $\log(d)$  steps have it  broken or  strongly isolated  by virtue of Observation 
  \ref{obcmpisl}. Hence we can extend our study in the previous subsections.  $O(\log(d))$  subdivision steps for compact components  involve $O(m\log(d))$ 
 e/i tests. The factor $\log(d)$ is immaterial for the $\tilde O$ estimates of 
 our Main Theorems, but we  remove it next.
 
\subsection{From isolation to strong isolation}\label{sistostris} 
Let $s_1$ denote  the sum of all $m$ roots in $D$.
Clearly, the point
 $s_1/m$ lies in   
MCD and even in the convex hull of the root set $X(D)$.
 We compute approximation $c=s_1(b')/m$ of $s_1/m$ within $1/2^{b'}$ at the cost of $O(b')$ queries of 
an evaluation oracle
(see Thm. \ref{thpwrsm0}). For $b'$ of order $\log(d)$ we yield error bound
$1/d^{\nu}$ for a fixed positive constant 
$\nu$  of our choice at the cost of $O(\log(d))$ queries of 
an evaluation oracle \ref{thpwrsm0}).
 
By applying Alg. \ref{algdblexpk} we compute  approximation $\rho$ of $r_{d-m+1}(c)$ also within
$1/d^{\nu}$, and so
\begin{equation}\label{eqc_rh}
|c-s_1/m|\le 1/d^{\nu}~{\rm
and}~|\rho-r_{d-m+1}(c)|\le  1/d^{\nu}.
\end{equation}
The  rigidity $\eta=\eta(D(c,\rho))$ of the disc $D(c,\rho)$ is given by the value $\frac{1}{\rho}\max^*|x_i-x_j|$ where maximum 
 is over all pairs of roots lying in the disc $D(c,\rho)$. Now notice that 
\begin{equation}\label{eqrg}
 \max^*|x_i-x_j|>
 0.5~r_{d-m+1}(s_1/m).
\end{equation}
Indeed, because of invariance of the claim under map (\ref{eqshft}),
we can assume that $s_1=0$ and
$x_i<0$. Then $\Re (x_j)> 0$ for some $j$,
and (\ref{eqrg}) follows.

Furthermore,
$|r_{d-m+1}(s_1/m)-r_{d-m+1}(c)|\le |s_1/m-c|$; combine this bound with Eqn. (\ref{eqc_rh})
and obtain that
 $|r_{d-m+1}(s_1/m)-\rho|\le 2/d^{\nu}$.
 Together with (\ref{eqrg}) this estimate implies that
 $$\max^*|x_i-x_j|>0.5(\rho-2/d^{\nu}).$$

 If $\rho \ge 4/d^{\nu}$, then $\rho-2/d^{\nu}\ge \rho/2$, the disc $D(c,\rho)$ is 0.25-rigid, and compression is completed.
 
 If $\rho \le 4/d^{\nu}$, then the disc
 $D(c,\rho)$, 
 sharing with the unit disc $D(0,1))$ all its $m$ roots, is strongly 
 isolated, namely, is $d^{\nu}/4$-isolated;
  then we can apply  
  to it the compression  algorithms of the previous sections. 
 

\section{Reduction of $m$-test to  e/i tests}\label{smvia1} 

To complete our compression algorithms we should supply a $\sigma$-soft $\ell$-test for the input of Alg. \ref{algdblexpk}, for which we can apply
the low cost Las Vegas $\ell$-test in Sec.  \ref{mdls}. 

In application to compression, however, we
can apply two alternative tests because we only need either (i) a $\sigma$-soft $m$-test for a sub-disc
$D(0,\rho)$  of the unit disc $D(0,1)$ containing $m$ roots where $\rho<1$ or alternatively (ii) e/i test for the annulus $A(0,\rho,1)$ of Sec. \ref{seiann}. 

{\em (i) A $\sigma$-soft $m$-test for a sub-disc.} Clearly, we can immediately reduce  $d$-test in the unit disc $D(0,1)$ for $p(x)$ to performing 1-test (that is, e/i test) for that disc and for  the reverse polynomial
 $p_{\rm rev}(x)$ of
(\ref{eqpolyrev}), but
next we  reduce $m$-test for any $m$, $1\le m\le d$, and for any disc $D$ to bounded number of e/i tests for sub-discs.  If at least one of the e/i tests fails, then 
 definitely there is a root in a narrow annulus about
       the boundary circle of an input disc, and then the internal disc of the annulus  would not contain $m$ roots.  Otherwise 
we obtain a root-free concentric annulus about the boundary circle of an input disc
  and then  
 reduce $m$-test to  root-counting in an isolated disc, which we can perform at a dominated cost (see claim (ii) of  Prop. \ref{thtsts1}).
  
      
\begin{algorithm}\label{algcrtc1} {\em Reduction of $m$-test  to e/i tests.} 
 

INPUT: a polynomial $p(x)$ of a degree $d$, two values $\bar\theta>1$ and $\bar\sigma>1$,
an integer  $m$ such that $1\le m\le d$ and $\#(D(0,\bar\sigma))\le m$,
and two algorithms $\mathbb {ALG}~0$
and $\mathbb {ALG}~1$, which for a black box polynomial perform   
 $\bar\sigma$-soft e/i test  for any disc and root-counting in an isolated disc, respectively.
 
INITIALIZATION (see Fig. \ref{figdetei}): Fix a positive $\rho$ such that  $\bar \sigma=O(\rho)$, say, $\rho=\bar \sigma/5$, and compute the integer 
\begin{equation}\label{eqv}
v= \lceil 2\pi/\phi\rceil
\end{equation}
 where 
$\phi$ is the angle defined by an arc of $C(0,1)$ with end points
at the distance $\rho$, and so $v=O(1)$ for $\bar\sigma=O(\rho)$. Then compute the values
  \begin{equation}\label{eqmurh0}
c_j=\exp(\phi_j\sqrt {-1}),~j=0,1,\dots,v,
\end{equation}
such that 
$|c_j-c_{j-1}|=\rho$ for $j=1,2,\dots,v$
 and 
$\phi_{j}-\phi_{j-1}=\phi$.
 \begin{figure}
    \centering
   \begin{tikzpicture}
   [scale=1.0,every node/.style={draw=black, circle}]
  \draw node[fill,inner sep=1pt] (0, 0) {};
  \draw [red] (0, 0) circle (2cm) ;
  \draw node [fill, label=right:{$c_1$},inner sep=1.5pt] at (2, 0) {};
  \draw (2, 0) circle (1.3cm);
  \draw node [fill, label=right:{$c_2$},inner sep=1.5pt] at (1.5775, 1.2294) {};
  \draw (1.5775, 1.2294) circle (1.3cm);
  \draw node [fill, label=right:{$c_3$},inner sep=1.5pt] at (0.4885, 1.9394) {};
  \draw (0.4885, 1.9394) circle (1.3cm);
  \draw node [fill, label=right:{$c_4$},inner sep=1.5pt] at (-0.8069, 1.8300) {};
  \draw (-0.8069, 1.8300) circle (1.3cm);
  \draw node [fill, label=right:{$c_5$},inner sep=1.5pt] at (-1.7614, 0.9474) {};
  \draw (-1.7614, 0.9474) circle (1.3cm);
  \draw node [fill, label=right:{$c_6$},inner sep=1.5pt] at (-1.9717, -0.3346) {};
  \draw (-1.9717, -0.3346) circle (1.3cm);
  \draw node [fill, label=right:{$c_7$},inner sep=1.5pt] at (-1.3489, -1.4766) {};
  \draw (-1.3489, -1.4766) circle (1.3cm);
  \draw node [fill, label=right:{$c_8$},inner sep=1.5pt] at (-0.1563, -1.9939) {};
  \draw (-0.1563, -1.9939) circle (1.3cm);
  \draw node [fill, label=right:{$c_9$},inner sep=1.5pt] at (1.1024, -1.6687) {};
  \draw (1.1024, -1.6687) circle (1.3cm);
  \draw node [fill, label=right:{$c_0$},inner sep=1.5pt] at (1.5775, -1.2294) {};
  \draw (1.5775, -1.2294) circle (1.3cm);
  
  \end{tikzpicture}

\caption{$v=9$, $|c_j-c_{j-1}|=\rho>|c_9-c_0|$
for $j=1,2,\dots,9$.}
\label{figdetei}
\end{figure}     

OUTPUT: Either an upper bound $r_{d-m+1}< 1-\rho\frac{\sqrt 3}{2}$
 or one of the two lower bounds 
  $r_{d-m+1}\ge 1-\bar \sigma\rho$ or
 $r_{d-m+1}> 1+(\sqrt 3-1)\rho$.
 
 COMPUTATIONS:   
 Apply 
  algorithm  $\mathbb {ALG}~0$ (soft e/i test) to the 
   discs 
 $D(c_j,\rho)$, $j=0,1,\dots,v$.
(a) Stop and certify that
 $r_{d-m+1}\ge 1-\bar \sigma\rho$,  that is, $\#(D(0,1-\bar \sigma\rho))<m$ for the open disc $(D(0,1-\bar \sigma \rho)$,
    unless 
  algorithm $\mathbb {ALG}~0$ outputs exclusion in all its $v+1$ applications.
  
 (b)  If it outputs exclusion $v+1$ times, then apply algorithm $\mathbb {ALG}~1$ (root-counter) 
 to the $\theta$-isolated disc $D(0,1)$
 for 
 \begin{equation}\label{eqqtht0}
  ~\theta:=\min\Big\{1+(\sqrt 3-1)\rho,\frac{1}{1-0.5\rho\sqrt 3}\Big\}.
\end{equation}
   Let 
$\bar s_0=\#(D(0,1))$ denotes its output  integer. Then conclude that \\
 (i) $r_{d-m+1}\le 1-\rho\frac{\sqrt 3}{2}$, and so
  $\#(D(0,1-\rho\frac{\sqrt 3}{2}))=m$ if $\bar s_0=m$,
   while \\
   (ii) $r_{d-m+1}\ge 1+(\sqrt 3-1)\rho$,  and so $\#(D(0,1+(\sqrt 3-1)\rho))<m$  for the open disc
   $D(0,1+(\sqrt 3-1)\rho)$
   if $\bar s_0<m$.

\end{algorithm}
  
  To prove
correctness of Alg. \ref{algcrtc1} we need the following result.
\begin{theorem}\label{thvrfrtcnt}
For the above values  $c_0,\dots, c_{v}$, 
  and $\rho$,
 the domain 
$\cup_{j=0}^{v-1}D(c_j,\rho)$ covers the annulus  
$A(0,1-0.5\rho\sqrt 3,1+(\sqrt 3-1)\rho)$.
\end{theorem}
\begin{proof}
 Let $c_{j-1}$ and $c_j$ be the  
centers  of two neighboring circles
$C(c_{j-1},\rho)$ and $C(c_j,\rho)$,
let $z_j$ and $z'_j$ be their intersection  points  such  that $|z_j|=|z_j'|-|z_j-z'_j|$
for $j=1,2,\dots, v$,
 observe that the values $|z_j-z_j'|$,
 $|z_j|$,  $|z'_j|$, and $|c_{j-1}-c_j|$  are invariant in $j$, and hence so is the annulus $A(0, |z_j|, |z_j'|)$  of width $|z_j-z_j'|$
 as well; notice that it lies in the domain
$\cup_{j=0}^{v-1}D(c_j,\rho)$. 

  
\begin{figure}\label{fig. 3}

\centering
\begin{tikzpicture}[scale=1.0,every node/.style={draw=black, circle}]
\draw (4, 0) arc (0:180:4);
\draw node (c1) [fill, label=right:{$c_{0}$},inner sep=1pt] at (0.3993, 3.9800) {};
\draw (0.3993, 3.9800) circle (2cm);
\draw node (c2) [fill, label=right:{$c_{1}$},inner sep=1pt] at (-1.5774, 3.6784) {};
\draw (-1.5774, 3.6784) circle (2cm);

\draw node (z) [fill, label=right:{$z$},inner sep=1pt] at (-0.3256, 2.1160) {};
\draw node (z') [fill, label=right:{$z'$},inner sep=1pt] at (-0.8545, 5.5398) {};

\draw (c1) -- (c2)
 (c1) -- (z)
 (c1) -- (z')
 (c2) -- (z)
 (c2) -- (z')
 (z) -- (z'); 

\end{tikzpicture}

\caption{ $|c_{1}-c_0|=|c_{0}-z|=|c_1-z|=|c_{0}-z'|=|c_1-z'|=\rho$.}\label{figz}
\end{figure}

Let us prove that $|z_j-z_{j}'|=\rho\sqrt 3$
for $j=1$ and hence also for $j=2,\dots,v$.

Write $z:=z_1$ and $z':=z_1'$ and observe (see Fig. \ref{figz}) that $|c_{1}-c_0|=|c_{0}-z|=|c_1-z|=|c_{0}-z'|=|c_1-z'|=\rho$, and so
 the points $c_0$, $c_{1}$, and $z$
are the vertices of an 
equilateral triangle with  a side length 
$\rho$. Hence $|z-z'|=\rho\sqrt 3$.
The circle $C(c_j,\rho)$ intersects the line 
segment $[z_j,z_j']$ at a point $y_j$
such that $|y_j-z_j|>\rho\sqrt 3/2$ and 
$|y_j-z_j'|>(\sqrt 3-1)\rho$ for $\rho<1$, and
 this implies the theorem if $|c_v-c_0|=\rho$.
 If, however, $|c_v-c_0|<\rho$ and if
 $C(c_{0},\rho)\cap C(c_v,\rho)= \{\bar z,\bar z\}$, then
 $|\bar z-\bar z'|>|z-z'|$, and the theorem follows in this case as well.
\end{proof}
\begin{corollary}\label{ocvrfrtcnt}
The disc $D(0,1)$ is 
$\theta$-isolated for 
 $\theta$ of (\ref{eqqtht0}) provided that 
the discs $D(c_j,\rho)$  contain no roots for  $j=0,1,\dots,v$.
\end{corollary} 

{\em Correctness of Alg. \ref{algcrtc1}.}
 (i)  If 
  algorithm $\mathbb {ALG}~1$ outputs inclusion in at least one of its applications, then one of the discs
   $D(c_j,\bar\sigma\rho)$ contains a root, and then 
   Alg. \ref{algcrtc1} correctly reports that   
    $r_{d-m+1}\ge 1-\sigma\bar \rho$. \\
 (ii) Otherwise Cor. \ref{ocvrfrtcnt} holds, and then
 correctness of Alg. \ref{algcrtc1} follows from correctness of
 $\mathbb {ALG}~1$.  

{\em Softness} is at most $\frac{1}{1-\bar\sigma \rho}$,   
$\frac{1}{1-0.5 \rho\sqrt 3}$, 
and $1+(\sqrt3-1)\rho$ in cases (a), (b,i), and (b,ii), respectively.
 
{\em Computational complexity.}
Alg. \ref{algcrtc1} invokes  algorithm $\mathbb {ALG}~0$ at most $v+1$ times for $v=O(1)$ of (\ref{eqv}). If $\mathbb {ALG}~0$ outputs exclusion $v+1$ times,   then  algorithm $\mathbb {ALG}~1$ is invoked and runs at a dominated cost.
  
  \begin{remark}\label{rerdctn2}
One can modify  Alg. \ref{algcrtc1}
as follows (see \cite{IP22}):
instead of fixing  a positive $\rho$ and then computing $v$
  of (\ref{eqv})
  fix  an integer $v>0$, apply   algorithm $\mathbb {ALG}~0$ at 
$v$ equally spaced points $c_j=\exp(\frac{2j\pi\sqrt {-1}}{v}),~j=0,1,\dots,v-1$, let $\phi=\frac{2\pi}{v}$,
and then define $\rho$.  We also arrive at this algorithm if we compute or guess this value $\rho$ and then apply   Alg. \ref{algcrtc1}
but 
 notice that $c_v=c_0$ for $c_j$ of (\ref{eqmurh0}) and only invoke algorithm $\mathbb {ALG}~0$ for 
$c_0,\dots,c_{v-1}$.
\end{remark}

\section{Counting e/i tests}\label{scntei} 

\subsection{An upper estimate}\label{suppest} 

The {\em upper estimate} $2m-2$ of  part (ii) of Sec. \ref{saccIV}  on the number of compression  steps is 
{\em sharp} -- reached in the case of the complete binary tree  $\mathbb T$.

Besides  compression steps our subdivision root-finders involve  e/i tests. Let us count them.
Each compression step is preceded by an e/i test, and so 
we may need to perform at least $2m-2$ e/i tests.  According to the main theorem of this section below, this lower bound  is  sharp up to a constant factor:   
 
\begin{theorem}\label{thcntssp}
Let subdivision iterations  
with compression be applied to a suspect square containing at most $m$ roots. Then they involve  $O(m)$ e/i tests overall.
\end{theorem} 

\subsection{The first lemma}\label{sfstlmm}  

We begin our proof 
of Thm. \ref{thcntssp}  with  
\begin{lemma}\label{eqnjc}
In every component $C$ made up of $\sigma(C)$
suspect squares we can choose $n(C)\ge \sigma(C)/5$ non-adjacent suspect squares.
\end{lemma}
\begin{proof}
Among  all leftmost squares in $C$ the uppermost one has at most four neighbors --
none on the left or directly above. Include this square into the  list of non-adjacency
squares and remove it and all its neighbors from $C$. Then
apply this recipe recursively until you remove all  $\sigma(C)$ suspect squares from $C$. Notice that at least 20$\%$ of them  were non-adjacent. 
\end{proof}

\subsection{The second lemma}\label{sscnlmm}
 
\begin{lemma}\label{lecmpn1} 
Fix two  integers $i\ge 0$ and $g\ge 3$, fix a value $\sigma$ of softness,   $1<\sigma<\sqrt 2$, and define  the separation  fraction $a:=2-\sigma\sqrt 2$ of Observation \ref{obcmp}.
 For all $j$ let $j$-th subdivision 
 step  process $v_j$ components made up of $w_j$ 
suspect squares, $n_j$ of which are non-adjacent, let $\Delta_j$ denote their side length, and let the $j$-th steps  perform  no compression 
for $j=i,i+1,\dots,i+g$. Then   
\begin{equation}\label{eqdltdcr}
w_{i+g}\ge (n_i-v_{i})a2^{g-0.5}~{\rm if}~\Delta_i\ge 2^g\Delta_{i+g}.
\end{equation}

\end{lemma}

\begin{proof} At the $i
$th step a pair of non-adjacent suspect squares of the same component is separated by at least the  gap $a\Delta_i$, and  there are at least $n_i-v_i$ such gaps
with overall length 
$(n_i-v_i)a\Delta_i$.
Neither of these gaps disappears at the $(i+g)
$th step, and so they
must be covered by $w_{i+g}$ 
suspect squares, each of diameter $\sqrt 2~\Delta_{i+g}=\Delta_i/2^{g-0.5}$
under (\ref{eqdltdcr}). Hence 
$(n_i-v_i)a\Delta_i\le w_{i+g}
\Delta_i/2^{g-0.5}$.
\end{proof} 

Next we extend Lemma \ref{lecmpn1} by allowing  compression steps. We only need to extend the equation
$2\Delta_{j+1}=\Delta_{j}$  to the case where the
$j$th subdivision step for $i\le j\le i+g$  
 performs non-final compression, each followed   
 with subdivision steps applied to a single  
 suspect square with side length  at most  $4\Delta_{j}$. 
 
The next two subdivision steps perform only five e/i tests, making up    at most $5m -10$  e/i tests  performed at all pairs of subdivision steps  following
   all (at most $m-2$) non-final compression steps.
   
  Now we only need to bound by $O(m)$ the number of  e/i tests in the  remaining subdivision steps,  not interrupted with compression steps. 
  
After a compression step we resume subdivision process  with a suspect square whose side length 
   bound decreases in the   next two subdivision steps  from at most
     $4\Delta_{j}$ to   at most  $\Delta_{j}$. 
 
   Hence we can apply 
 to it Lemma \ref{lecmpn1}, because now it 
   accesses  the same
number $n_j$ of non-adjacent suspect squares and an increased number $\sigma_j$ of all suspect squares (cf. (\ref{eqdltdcr})).
 
\subsection{Counting e/i tests at partition steps}\label{scntgeiprt}
 
 Let $h$ denote  the height of our component   tree  and count suspect squares 
processed in subdivision steps that partition components, temporarily ignoring the other steps.  

Clearly,  bounds (\ref{eqdltdcr}) 
still hold  and can only be strengthened. 
Notice that $v_i\ge 1$ for all $i$, 
sum $v_{i+g}$ as well as bounds (\ref{eqdltdcr})
 for a fixed $g>0$  and $i=0,1,\dots,h-g$, and obtain

\begin{equation}\label{eqsgmnv}\sum_{i=0}^{h-g} n_i-\lambda \sum_{i=0}^{h-g}\sigma_{i+g}\le \sum_{i=0}^{h-g}v_{i+g}~{\rm for}~\lambda:=\frac{\sqrt 2}{a2^g}.
\end{equation} 
Notice that  
$$\sum_{i=0}^{h-g}v_{i+g}\le \sum_{i=0}^hv_i-\sum_{i=0}^{g-1}v_i<2m-g,$$
Combine the above estimate and the bounds  $n_i\le m$ and
(cf. Lemma \ref{eqnjc})
$w_i\le 5n_i,$  
for all $i$, write $n':=\sum_{i=0}^{h} n_i$,  
obtain 

$$\sum_{i=0}^{h-g} n_i\le n'-\sum_{i=h-g+1}^hn_i\le 
n'-mg,~
\sum_{i=0}^{h-g}w_{i+g}\le \sum_{i=0}^{h}w_i\le 5n',$$ 

and so 
$$(1-5\lambda)n'\le 2m-1-g+ mg~{\rm for}~\lambda~{\rm of~(\ref{eqsgmnv})}.$$

   Obtain that $$0.5n'\le 2m-1-g+ mg~{\rm for}~g=1+\Big\lceil\log_2\Big(\frac{5\sqrt 2}{a}\Big)\Big\rceil.$$
   
    Hence
    
\begin{equation}\label{eqcmpnt1}
 n'\le 4m-2-2g+2mg,~{\rm and~so}~ n'=O(m)~{\rm for}~
g=O(1).
\end{equation}


\subsection{Counting all e/i steps}\label{scntall}
 
Next  
consider  a non-compact component, count all suspect squares processed in all subdivision steps applied  until it is partitioned, and then prove in Lemma \ref{lecmpnt1}
that up to a constant factor the overall number of these suspect squares is bounded by 
their overall number at all partition steps. Such a result will extend 
to all subdivision steps the bound $n'=O(m)$ of (\ref{eqcmpnt1}), so far  restricted to partition steps. Together with Lemma \ref{eqnjc} this will complete our proof of Thm. \ref{thcntssp}.

     
\begin{lemma}\label{lecmpnt1}
Let  $k+1$ successive subdivision steps
with $\sigma$-soft e/i tests  be applied  to a non-compact component for $\sigma$ of our choice in the range $1<\sigma<\sqrt 2$.
Let the $j$-th step output   $n_j$ non-adjacent suspect squares for $j=0,1,\dots,k$.
   Then
$\sum_{j=0}^{k-1}n_{j}\le n_k\frac{10\sqrt 2}{a}$ for the separation fraction $a:=2-\sigma\sqrt 2$ (cf. Observation 
\ref{obcmp}.)
\end{lemma}  


\begin{proof}
Substitute $g=k-i$
into (\ref{eqcmpnt1})
 and obtain
 $w_k\ge (n_i-v_i)a2^{k-i-0.5}$, for $i=0,1,\dots,k-1$.
 
  Recall that $w_j\le 5 n_j $ (cf. Lemma \ref{eqnjc})
and that in our case $v_i=1$ and $n_i\ge 2$ and obtain that  $n_i-v_i\ge n_i/2$ for all $i$.

 Hence $5n_k\ge n_ia2^{k-i-1.5}$ for all $i$,
 and so $$\sum_{i=0}^{k-1}n_i\le n_k\frac{10\sqrt 2}{a}\sum_{j=1}^{k}2^{-j}<
 n_k\frac{10\sqrt 2}{a}.$$ 
  \end{proof} 
 
    
\begin{remark}\label{recmpnt1}
The overhead constants hidden in 
 the "O" 
notation of our estimates are fairly large but overly pessimistic.
E.g., (i)  the roots $x$ and $y$   
involved in the proof of Lemma \ref{lecmpn1} are linked by a path of suspect squares, but we deduce the estimates of the lemma by using only a part of this path, namely, the part that covers the gaps between non-adjacent suspect squares. 
 (ii) Our proof of that lemma 
 covers the worst case where a long chains 
 between the roots $x$ and $y$ can consist of up to $4m$ suspect squares, but for typical inputs such chains are much shorter  and can consist of just a few suspect squares.  (iii) In the proof of Lemma \ref{lecmpnt1} 
 we assume that $n_i-1\ge n_i/2$;  this holds for any $n_i\ge 2$, but the  bounds are strengthened
  as $n_i$  increases.  (iv) In a component made up of five suspect squares at least two are non-adjacent,
 and so the bound of Lemma \ref{eqnjc} can be strengthened at least in this case.
\end{remark}


\section{Overall complexity of our root-finders}\label{sovrllcmpl}

We reduced our root-finding in a  square with $m$ roots in its  close neighborhood 
to $O(m)$ e/i tests in addition to at most $2m-2$
compression steps. We  
 reduced compression first of an isolated and then of strongly isolated  disc to (i) counting roots in an isolated disc ($O(\log(d))$  evaluation queries)
 (ii) application of Schr{\"o}der's single iteration (two evaluation queries),  and (iii) approximation of the root radius $r_{d-m+1}(z,p)$, that is, 
performing $O(\log(b))$  $m$-tests.
 
 In Sec. \ref{smvia1}   we reduced an $m$-test  to $O(1)$
 e/i tests, and so the overall complexity of subdivision iterations   is dominated by the complexity of performing $O(m\log(b))$ e/i tests and  $O(m\log(d))$
 queries of an evaluation oracle. By combining the latter estimates   with those  of Sec. \ref{sei} for an e/i test, we complete the proof of  Main Theorems 1 -- 4. 
  
\section{Algorithmic extras}\label{sextr}  
  
 Next we
 accelerate our root-finders for a large class of inputs.

1. To compress the isolated unit disc $D:=D(0,1)$ containing a small number $m$ of roots $z_1,\dots,z_m$,  compute
$s_1,\dots,s_m$ the first $m$ powers sums of the roots lying in $D$, then recover 
the coefficients of the factor $f(x)=\prod_{j=1}^m(x-z_j)$ from {\em Newton's identities}
(cf. \cite[Eqn. (2.5.4)]{P01}),  approximate its zeros, and apply a  black box root-finder to the polynomial
$p(x)/f(x)$ represented by the pair of polynomials $p(x)$ and $f(x)$. 

2. In the case of a  large integer $m$ the latter algorithm suffers from swelling the coefficients of the factor $f(x)$. For $m=d$ the algorithm
\cite{RSS24} empirically approximates all  $m$
roots at a rather low cost, which provides more than just  compression of $D$.
Can we extend this approach  to compression in the case where $m<d$
and the disc $D$ is sufficiently strongly
isolated? 

This is a research challenge, but we propose a similar low cost heuristic recipe which, as we conjecture, accelerates our root-finding for a large class of inputs.

Let  the disc $D=D(0,1)$ be strongly isolated and let
$i(D)\ge d^{2\nu}$ for a reasonably large $\nu$,
and then instead of Newton's  apply  Schr{\"o}der's single iteration (\ref{eqSCHR}) at sufficiently many, say 100,
equally spaced points $y$  of the circle $C:=C(0,d^{\nu})$.
Let $Z$  denote the  set of images  $z=y-m{\rm NR}(y)$. 
By extending  our study in Sec.
\ref{secmprstris},  deduce that  all these images lie in the $\alpha\Delta$-neighborhood of  the root set $X(D)$ for $\Delta:=\Delta(D)$
and $\alpha$ fast converging to 1 as $\nu$ grows large and 
hence that $\Delta(Z)\le (2\alpha+1)\Delta$.

We may run into worst case inputs where $Z$ is a singleton and $\Delta(Z)=0$, but we conjecture that  under 
a random rotation of the set of points $y$,
the ratio $\Delta(Z)/\Delta$ exceeds a positive constant whp, and then we can narrow the initial range  of BoL in Sec. \ref{sextrrk} from $[-b,0]$ to $[\log_2(\Delta(Z)),0]$
at a low cost of performing Schr{\"o}der's iteration (\ref{eqSCHR}) at reasonably  many, say 100, points. 

\medskip
\medskip

\bigskip

\medskip
 
   
\medskip



\noindent {\bf \Large Part V: Factorization, Root Isolation, Functional Iterations, and Root Radii Estimation} 
\medskip 

We organize this part of the paper as follows.
 In Secs. \ref{scprtf} and \ref{4.15} we cover some results by Sch{\"o}nhage
\cite{S82} on polynomial factorization and its links to root-finding and root-isolation, and our  lower and upper bit-complexity estimates for these problems. In Sec. \ref{snopth} we recall 
root-finders by means of
 Newton's, Weierstrass's
 \cite{W03}
(aka Durand-Kerner's \cite{D60,K66}), and Ehrlich's  of  1967 \cite{E67}, aka Aberth's  of 1973  \cite{A73}
iterations, accelerate them by means of incorporation of FMM, and initialize Newton's iterations for a black box input polynomial.
In Sec. \ref{sexclexc2} 
 we recall fast approximation of all 
 $d$ root radii  where the coefficients of $p$ are given.
   
\section{Polynomial Factorization,  Root-finding and Root  Isolation: \\Linking three Problems and their bit operation complexity}\label{scprtf} 

\subsection{Linking Factorization and Root-finding: upper bounds
on the bit operation  complexity}\label{sfctrrlnks}  
  
  
Hereafter write $|u|=\sum_{i=0}^d|u_i|$ for a polynomial 
$u(x)=\sum_{i=0}^du_ix^i$.

The root-finders of \cite{S82,P95,P02} solve Problems 0 and 0$^*$ by extending the  solution of 
 the following problem of independent 
importance:

{\bf Problem F. Approximate Polynomial Factorization:}\footnote{The solution of  Problem F can be also extended to isolation of the zeros of a square-free polynomial with integer coefficients  (see  our Remark \ref{rertisl}) and has various other applications to modern computations, listed in Sec. \ref{prbs}. For a  polynomial of degree $d$  
the known sharp upper and lower bounds on the input precision and number of bit operations for root-finding exceed 
    those for numerical factorization by a factor of $d$.} Given a  real $b$ and $d+1$  coefficients 
$p_0,p_1,\dots,p_d$  of (\ref{eqpoly}),
 compute
 $2d$ complex numbers
$u_j,v_j$ for $j=1,\dots,d$
such that
\begin{equation}\label{eqfctrp} 
|p-\prod_{j=1}^d(u_jx-v_j)|\le \epsilon |p|~
{\rm for}~\epsilon=1/2^{b}.
\end{equation}
    
Hereafter $\mathbb B_0(b,d)$  and
 $\mathbb B_F(b,d)$ denote the optimal number of bit operations required for the solution of Problems 0 and F, respectively, for an error bound $\epsilon=1/2^b$,
and next we  specify a class of polynomials whose {\em  zeros  make up no large
 clusters}
(cf. \cite{BCSS98,CS99}).
 
\begin{definition}\label{defclstr} {\bf  Polynomials with bounded root cluster sizes.} The zero set of $p$
consists of  
clusters, each made up of a constant number of zeros and  having a diameter at most $1/2^{3-b/d}$; furthermore, these clusters are 
pairwise separated 
by the distances at least  $1/d^{\tilde O(1)}$.
\end{definition} 

In  \cite{P95,P02} we    
proved the following 
near-optimal bound for Problem F.

\begin{theorem}\label{thbscfct}
   $\mathbb B_F(b,d)=\tilde O((b+d)d)$.
\end{theorem}

In Sec. \ref{sfctrrapp1}  we reduce Problem 0 to Problem F and vice versa by applying Sch{\"o}nhage's results of \cite[Cor. 4.3 and Sec. 19]{S82}:

\begin{theorem}\label{thprb1toprb2}
(i) $\mathbb   B_0(b,d)=\tilde O(\mathbb  B_F(bd,d))$ for any polynomial $p(x)$ of degree $d$, (ii)  $\mathbb   B_F(b,d)=\tilde O(\mathbb  B_0(b+d+\log(d),d))$
for any polynomial $p(x)$  of degree $d$, and
(iii) $\mathbb   B_0(b,d)=\tilde O(\mathbb  B_F2(b,d))$ for polynomials of Def. \ref {defclstr},  with bounded cluster sizes.  
\end{theorem}

By combining 
 Thms. \ref{thbscfct} 
 and \ref{thprb1toprb2} obtain the following estimates.

\begin{theorem}\label{thblncs1}
 One can solve Problem 0
by using 
(i) $\tilde O((b+d)d^2)$
bit operations for any polynomial $p$ of (\ref{eqpoly}) and
(ii)  $\tilde O((b+d)d)$ bit operations for a polynomial $p$
with bounded cluster sizes.
\end{theorem}

\begin{corollary}\label{thblncs1m}
 One can solve Problem 0 by using $\tilde O((b+d)d+(b+m)m^2)$ bit operation for $m$ of Def. \ref{defepsbt} of Sec. 
\ref{scmplests}. The bound turns into 
$\tilde O((b+d)d)$ if 
$m^2=\tilde O(d)$. 
\end{corollary}

\begin{proof}
Proceed in two steps: \\
 (i)  by using $\tilde O((b+d)d)$ bit operations solve Problem F and then \\
 (ii) apply the algorithm that supports
Thm. \ref{thblncs1}  to the  $m$th degree factor of the polynomial  
$\prod_{j=1}^d(u_jx-v_j)$
of (\ref{eqfctrp}) that shares with this polynomial all its zeros lying in the  disc $D(1+\beta)$ for positive $\beta$ of Def. \ref{defepsbt}.\footnote{We can slightly speed up stage (i) because at that stage we actually only need to approximate factorization  of the factor  $f$ of degree $m$ of $p$ that shares with $p$ its zeros in $D$.}
\end{proof}

\subsection{Linking Factorization and Root-finding: lower bounds
on the bit operation  complexity}\label{slwrbnds}

Readily verify the following information lower bounds.

\begin{observation}\label{obbd0}
To solve Problem 0 or F one must access 
$(d+1)b$ bits of input coefficients and  perform at least $0.5(d+1)b$ bit operations -- at most one operation for two input bits.
\end{observation}

 Next we deduce a
stronger lower bound for  
Problem 0.  

 \begin{observation}\label{obbd}
 Let $p(x)=x^d +p_{d-j}x^{d-j}$ for an integer 
 $j$, $1<j<d$, and 
 complex  $p_{d-j}$ such that $|p_{d-j}|=1/2^{bj}$.
 Then the polynomial $p(x)$   has the $j$ zeros $\zeta^g2^{-b}$ for $g=0,1,\dots,j-1$, which vanish in the transition to $x^d$.
   \end{observation}
   Observation \ref{obbd} implies that one must 
access the coefficient $p_{d-j}$ within 
 $2^{-jb}$ to   
distinguish the above $j$ zeros of $p(x)$ from the value 0 or from the zeros of the
 cluster  of $j$ zeros
 of the polynomials obtained from $x^d$ by means of infinitesimal perturbation  
of the coefficient $p_{d-j}$. Sum these bounds for 
$j=1,\dots,d$ and obtain  

  \begin{corollary}\label{cobd}
   One must access at least
 $0.5(d+1)db$ 
   bits of the coefficients of a monic polynomial $p$ and  perform 
at least   $0.25(d+1)db$ bit operations
    to approximate even  its single zero within $2^{-b}$ if no restriction is imposed on the number of roots in an input complex domain $\mathbb D$ and its $\beta$-neighborhood. 
   \end{corollary}
 
The  lower bounds of this subsection imply that the algorithm of \cite{P95,P02}  
 and its extension supporting Thm. \ref{thblncs1m} solve  Problems 0 for $m=d$  and F by using optimal number  of bit operations up to polylogarithmic factors.
 
Let us extend this study to Problem 0 for $m<d$.
In this case  isolate
 from one another the two sets made up of $m$ and $d-m$ roots. 

 \begin{observation}\label{obbdm}
 Let $p(x)=(x^m +p_{d-j}x^{m-j})(x-z_1)^{d-m}$ for an integer 
 $j$, $1<j<m$, 
 complex  $p_{d-j}$ such that $|p_{d-j}|=1/2^{bj}$, and sufficiently large $|z_1|$.
 Then the polynomial $p(x)$   has the $j$ zeros $\zeta^g2^{-b}$ for $g=0,1,\dots,j-1$, which vanish in the transition to $p(x)=x^m(x-z_1)^{d-m}$.
\end{observation}

Now we extend  Cor. \ref{cobd} as follows.

  \begin{corollary}\label{cobdm}
  For an integer $m$, $1\le m\le d$,
   one must access at least
 $0.5(m+1)mb+db$ 
   bits of the coefficients of a monic polynomial $p$ and  perform 
at least   $0.25(m+1)mb+0.5db$ bit operations
    to approximate within $2^{-b}$  even  its single zero lying in the $\beta$-neighborhood of a complex domain $\mathbb  D$ for a fixed constant $\beta>0$ if this neighborhood contains at most $m$  roots. 
\end{corollary}
   
\begin{proof}
Obtain the lower bounds
$0.5(m+1)mb$ and 
$0.25(m+1)mb$ as for
Cor. \ref{cobd}. Then notice that a $\delta$-perturbation of any trailing coefficient $p_i$ of $p(x)$ for $i\le d-m$ can move the roots lying in the $\beta$-neighborhood of  $\mathbb  D$ by order of $\delta$.
\end{proof}

 Our Main Theorem 3 implies that these lower bounds 
are near-optimal for $m$ of order $d$.


\subsection{Linking  Factorization  and  Root-finding: proof of Thm. \ref{thprb1toprb2}}\label{sfctrrapp1}  
  

Next we prove Thm. \ref{thprb1toprb2}, by applying the results of \cite[Cor. 4.3 and Sec. 19]{S82}.

Recall that
$\mathbb B_0(b,d)$ and
 $\mathbb B_F(b,d)$ denote the optimal number of bit operations for the solution of Problems 0 and F, respectively, for an error bound $\epsilon=1/2^b$, and furthermore, recall the lower bounds  
\begin{equation}\label{eqlbnds} 
 \mathbb B_0(b,d)\ge 0.5(d+1)b~{\rm and}~
 \mathbb B_F(b,d) \ge 0.25(d+1)db
\end{equation} 
  of 
 \cite{P95,P02} and the following result (cf. \cite[Thm. 19.1]{S82}).
 \begin{theorem}\label{thfrmfctr}
 Let $r_1=\max_{j=1}^d |z_j|\le 1$.
 Then  a solution of Problem 0
 within an error bound
 $1/2^b$ can be recovered at a dominated bit operation cost from the solution of   Problem F within an error bound $1/2^{bd+2}$, and so 
$\mathbb B_0(b,d)\le
\mathbb B_F(bd+2,d)$. 
 \end{theorem}
 
Based on \cite[Cor. 4.3]{S82} we can also readily relax the assumption  that $p(x)$ is monic.

Furthermore, in view of (\ref{eqlbnds}), relaxing the bound $r_1\le 1$ would little affect the latter estimate because we can  apply map (\ref{eqshft}), which changes the error bound of Problem 0 by at most $\log_2(\max\{\rho,1/\rho\})$, for  $\rho\ge 1/2^b$. 
 
\cite[Thm. 2.7]{S85}, which we reproduce in Sec. \ref{sfctrrapp} as Thm. \ref{thfrmfctr1},
shows an alternative way
to relaxing both assumptions that $r_1\le 1$ and  
$p(x)$ is monic.  
  
  Sch{\"o}nhage  
 in \cite[Sec. 19]{S82}
has greatly strengthened the estimate of Thm. \ref{thfrmfctr} unless 
 ``the zeros of $p$ are clustered too much".
  Namely, he proved
    (see \cite[Eqn. (19.8)]{S82}) the following result.
   \begin{lemma}\label{lefrmfctr0}
  Assume that 
 the polynomial $p(x)$
is monic,  the values $u_j=1$ for all $j$ and complex $v_1,\dots,v_d$ satisfy   
bound (\ref{eqfctrp}), and  $|z_j-v_j|\le r:=2^{2-b/d}$ for all $j$. Furthermore, write $\Delta_{j,1}:=|v_j-v_1|$ for $j=2,3,\dots,d$ and suppose that $1.25 \Delta_{j,1}\le  r$ for $j=2,3,\dots,d$
and  $\Delta_1:=\min_{j}\Delta_{j,1}> r$.
Then $|z_1-v_1|\le 2^{2-b/m}/|\Delta_1-r|^{\frac{d-m}{m}}$.
\end{lemma}

If  $1/|\Delta-r|^{\frac{d-m}{m}}=2^{\tilde O(d)}$, 
 then  Lemma \ref{lefrmfctr0} implies that $|z_1-v_1|\le 1/2^{b/m-\tilde O(d)}$,
 and we can similarly  estimate $
 |v_j-z_j|$ for all $j=1,2,\dots,d$.
 
  By extending these  estimates we strengthen  
  Thm. \ref{thfrmfctr} as follows.
 \begin{corollary}\label{cofrmfctr0}
 Under bounding cluster sizes  of Def. \ref{defclstr}, it holds that
$|z_j-v_j|\le 1/2^{b/m-\tilde O(d)},~j=1,2,\dots,d$. Hence if $m=O(1)$ is a  constant, then
 a solution of Problem 0
 within an error bound   $1/2^{b}$ can be recovered at a dominated bit operation cost from the solution of 
  Problem F with a precision $1/2^{\bar b}$
  for a sufficiently large $\bar b=\tilde O(b+d)$, and so  $\mathbb B_0(b,d)\le
 \mathbb B_F(\tilde O(b+d),d)$.
 \end{corollary}
 

For converse reduction of
 Problem F  to Problem 0   recall 
 \cite[Cor. 4.3]{S82}
  in the case of $k=2$, covering the product of two polynomials:

   
  \begin{theorem}\label{thfctrnrms}
Let $p=gh$ for three polynomials $p$ of (\ref{eqpoly}), $g$, and $h$. Then $||g||_1||h||_1\le 2^{d-1}||p||_1$.
\end{theorem}

Now consider the products
$p=gh$ for $g=x-z_j$, for 
$j=1,\dots,d$. For a fixed $j$, an error of at most $\epsilon/2^d$ for  the 
root $z_j$
contributes at most $\epsilon ||h||_1$
to overall factorization  error of Problem F. 
Therefore,
$||h||_1\le 2^{d-1} ||p||_1/||x-z_j||_1\le
2^{d-1} ||p||_1$ by virtue of Thm. 
\ref{thfctrnrms}.
 Hence the  approximation errors of all the $d$ roots  
together
contribute at most
$2^{d-1}d\epsilon$ to
the relative error 
of the solution of Problem F
(ignoring dominated impact of higher order terms), and we arrive at  
\begin{corollary}\label{co2to1}
If  $b=O(b-d-\log(d))$,
 then Problem F for an error bound $1/2^b$ can be reduced at a dominated bit operation cost to Problem 0 for an error bound $1/2^{\widehat b}$
 and a sufficiently large $\widehat b=O(b+d+\log(d))$; 
hence and so  $\mathbb B_F(b,d)=O(\mathbb B_0(b,d))$. 
\end{corollary}

It follows that the complexity estimate of  Thm. \ref{thblntm}   can be  also applied to the
 solution of Problem F of factorization  provided that $b=O(b-d-\log(d))$.   
    Then again, the bit operation complexity of this solution of  Problem F is {\em optimal} (up to polylogarithmic factors)
 under the random root model provided that  $d=\tilde O(b)$  (cf. Observation   \ref{obbd0}).


\section{Recursive Factorization, its bit complexity,  and Root Isolation} \label{4.15}

In this section we briefly recall some  algorithms and complexity estimates from \cite{S82,S85}.
We first recursively decompose a polynomial into the product of linear factors and then extend the factorization to approximation and isolation of the roots. 


\subsection{Auxiliary norm bounds}
\label{4.2}

Recall two simple lemmas. 

\begin{lemma}
  \label{pro8.2.1}
  It holds that $| u  +   v | \le | u |+ |
v |$  and $|u   v | \le |u|~|v|$
for any pair of polynomials $u=u(x)$ and $v=v(x)$.
\end{lemma}

\begin{lemma}\label{pro8.2.2}
Let $ u= \sum_{i=0} ^{d}u_i x^i$ and $| u
|_2=(\sum_{i=0}^{d}|u_i|^2)^{\frac{1}{2}}$. Then $|u|_2 \le |u| $.
\end{lemma}

The following lemma relates the norms
of a polynomial \index{polynomial!norm}
and its factors.

\begin{lemma}\label{pro8.2.6}
  If
 $p= p (x) = \prod_{j=1}^k f_j$ 
for polynomials $f_1,\dots,f_k$ and  
$\deg  p \le d$, then 
$\prod_{j=1}^k | f_j  |  \leq 2
^{d} | p |_2 \le 2^d | p |. $
\end{lemma}
 \begin{proof}  The leftmost bound was
 proved by Mignotte in \cite{M74}.
 The rightmost bound follows from Lemma
 \ref{pro8.2.2}.
\end{proof}

\begin{remark}\label{re8.2.6} 
\cite[Lemma 2.6]{S85} states with no proof 
a little stronger bound
$\prod_{j=1}^k | f_j  | \le 2^{d-1} | p | $
 under the assumptions of Lemma \ref{pro8.2.6}. 
From various factors  
of the polynomial $p(x)=x^d-1$ such as $\prod_{j=1}^{\frac{d}{2}}(x-\zeta_d^j)$ for even $d$ and $\zeta_q$ of (\ref{eqdndgnr}), one can see some limitations
on  strengthening this bound. 
\end{remark}

 \subsection{Errors and complexity} \label{4.16}

  Suppose that
  we split 
  a polynomial $ p$  into a pair of factors over an isolated circle
 and recursively apply this algorithm to the factors  until they become linear of the form $
  ux+v$; some or all of them can be repeated. Finally, we arrive at
  complete approximate factorization
\index{factorization!approximate}  
\begin{equation}\label{equ8.15.1}
 p\approx  p^{\ast}=p^{\ast}(x)
     = \prod_{j=1}^d (u_j x + v_j ) .
    \end{equation}
Next, by following \cite[Sec. 5]{S82}, we estimate the norm of the residual polynomial
\index{polynomial!norm} 
\begin{equation}
  \Delta^{\ast}  = p^{\ast}  - p.
  \label{equ8.15.2}
\end{equation}

We  begin with an auxiliary result.

\begin{theorem}\label{th8.15.1}
Let
\begin{equation}\label{equ8.15.3}
\Delta_k= | p  - f_1  \cdots f_k  | \leq
  \frac{k \epsilon | p |}{ d} ,
  \end{equation}
\begin{equation}\label{equ8.15.4}
\Delta =  | f _1  - f  g  | \leq \epsilon _k | f_1  | ,
 \end{equation}
for some non-constant polynomials
\index{polynomial} $ f_1  , \ldots , f_k  , f  $ and $
g   $ such that
\begin{equation}\label{equ8.15.5}
  \epsilon_k  d 2^d\prod_{j=1} ^k | f_j  | \leq \epsilon .
 \end{equation}
Then
\begin{equation}\label{equ8.15.6}
| \Delta_{k+1} |  = | p  - f  g  f_2  \cdots f _k  | \leq \frac{(k+1)
  \epsilon | p | }{ d} .
  \end{equation}
\end{theorem}

\begin{proof} $\Delta_{k+1} = | p  - f_1  \cdots f_k  +(f_1 -f g)
f_2  \cdots f _k  |.$  Apply Lemma \ref{pro8.2.1}  and deduce that
$\Delta_{k+1}\le \Delta_{k}+ |(f_1 -f g)
f_2  \cdots f _k  |$ and, furthermore, that
$$|(f_1-fg)f_2\cdots f_k|\leq |f_1-fg|~|f_2\cdots f_k|=\Delta 
|f_2\cdots f_k|.$$
Combine the latter inequalities and obtain
$\Delta_{k+1} \leq \Delta_k + \Delta | f_2 \cdots f_k|$.
Combine this bound with (\ref{equ8.15.3})-(\ref{equ8.15.5}) and Lemmas
\ref{pro8.2.1} and  \ref{pro8.2.6} and obtain (\ref{equ8.15.6}). 
\end{proof}

Write $ f_1 : = f$ and $ f_{k+1}  = g .$ Then (\ref{equ8.15.6})
turns into (\ref{equ8.15.3}) for $ k $ replaced by $ k + 1 .$ 
Now compute one of the factors $f_j  $ as in (\ref{equ8.15.4}),  
 apply 
Thm.  \ref{th8.15.1}, then recursively continue
splitting \index{splitting!recursive}
the polynomial \index{polynomial!splitting into factors}
$p $ into factors of smaller degrees, and finally arrive at
    factorization\index{factorization} (\ref{equ8.15.1}) with
 $$ | \Delta^{\ast} | \leq \epsilon | p |$$
for $\Delta^*$ of (\ref{equ8.15.2}). Let us call this computation 
\textit{Recursive Splitting Process} provided that it begins with $k=1$ and
$f_1 =p$ and ends with $k=d$.

\begin{theorem}\label{th8.15.2}
In order to support (\ref{equ8.15.3}) for 
$0<\epsilon\le 1$ and  $j=1, 2,
\ldots, d$ in
 the Recursive Splitting Process\index{splitting!recursive}
 it is sufficient
 to choose $ \epsilon_k $
in (\ref{equ8.15.4}) satisfying
\begin{equation}\label{equ8.15.8}
  \epsilon_k ~ d~ 2^{2d+1} \leq \epsilon ~\textrm{for~all}~ k.
 \end{equation}
 \end{theorem}

\begin{proof} Prove bound (\ref{equ8.15.3}) by
induction on $j$. Clearly,
 the bound holds for $k=1$. It remains to
deduce (\ref{equ8.15.6}) from (\ref{equ8.15.3}) and
(\ref{equ8.15.8}) for any $k$. By first applying Lemma
\ref{pro8.2.6}  and then bound (\ref{equ8.15.3}), obtain
\[
  \prod_{i=1}^k|f_i|\leq
  2^{d}|\prod_{i=1}^kf_i | \leq
  2^{d}~\Big(1+ \frac{k \epsilon}{d}\Big)~|p|.
\]
The latter bound cannot exceed $2^{d+1}|{   p}|$ for $k \leq d$, $\epsilon \leq
1$. Consequently (\ref{equ8.15.8})
 ensures (\ref{equ8.15.5}),  and
then
   (\ref{equ8.15.6})
follows by virtue of Thm. \ref{th8.15.1}. \\
\end{proof}     
\begin{remark}\label{rerecprec}
The theorem implies that we can ensure the output precision bound $b'$ by working with the precision of
$$\bar b\ge \bar b_{\rm inf}=2d+1+\log_2 d+b'$$ 
bits throughout  Recursive Splitting Process.
\end{remark}

 \subsection{Overall complexity of Recursive Factorization} \label{scstpr2}

The overall complexity of recursive factorization
is dominated by the sum of  the bounds on the complexity 
of all splittings into pairs of factors.
The algorithm of  \cite{P02} adopts 
recursive factorization  of \cite{S82} but 
decreases the overall complexity bound by
roughly a factor of $d$. This is 
due to additional techniques, which ensure 
that in every recursive splitting of a polynomial into the product of two factors the ratio of their 
   degrees is at most  11. Consequently, $O(\log(d)$ recursive splitting 
  steps are sufficient in  \cite{P02}, and so the overall complexity of root-finding is in $\tilde O(db)$ -- proportional to that 
of the first splitting. In contrast,  the degrees of a factor in any splitting in 
  \cite{S82} can be low, e.g., 1, and so there can be   up to $d-1$ splittings overall,
  each of the first $0.5d$ of them being almost as expensive as the very first one.


 \subsection{From factors to roots}\label{sfctrrapp}
  
   
  \begin{theorem}\label{thfrmfctr1}   \cite[Thm. 2.7]{S85}.
  Suppose that
   $$p=p_d\prod_{j=1}^d(x-z_j),~
  \tilde p=\tilde  p_d\prod_{j=1}^d(x-y_j),~
 |\tilde p-p|\le\epsilon |p|,~\epsilon\le 2^{-7d}$$ 
  and $$|z_j|\le 1~{\rm for}~j=1,\dots,d'\le d,~|z_j|\ge 1~{\rm for}~j=d'+1,\dots,d.$$
    Then up to reordering the roots it holds that     
 $$|y_j-z_j|<9 \epsilon^{\frac{1}{d}}~{\rm for}~j=1,\dots,d';~
 \Big|\frac{1}{y_j}-\frac{1}{z_j}\Big |<9 \epsilon^{\frac{1}{d}}~{\rm for}~j=d'+1,\dots,d.$$ 
 \end{theorem}	 

 By virtue of  Thm. \ref{thfrmfctr1} for $b'=O(bd)$ 
  we can bound the bit operation  complexity of the solution of Problem 0
  by increasing the estimate for the complexity of 
 factorization in  Sec. \ref{scstpr2} by a factor of $d$.
  
  \begin{corollary}\label{cofrmfctr} {\em The bit operation complexity of the solution of Problem 0.} 
  Given a polynomial $p$ of degree $d$ and a real $b'\ge 7d$,
 one can approximate all its $d$ zeros
 within the error bound 
  $2^{-b'}$  by using 
  $O(\mu(b'+d)d^2\log(d)(\log^2d+\log(b')))=\tilde O((b'+d)d^2)$ bit operations.  
   \end{corollary} 
   \begin{remark}\label{reislt}
  The study in this section covers
  the solution of Problem 1 for the worst case input polynomial $p$.
  If its roots are isolated from 
  each other, then the upper estimates of Thm. \ref{thfrmfctr1} and Cor. \ref{coislt}
  can be decreased by a factor of $d$
  (cf. \cite{P02}). 
\end{remark}    

\subsection{Polynomial Root Isolation}\label{sisol}
   
  Based on his results recalled in this section, Sch{\"o}nhage in  \cite[Sec. 20]{S82}               
  estimates the bit operation complexity of the following 
  well-known problem
  by reducing it to Problem F. 
  
{\bf Problem ISOL:} 
  {\em Polynomial Root Isolation.} Given a polynomial $p$ of (\ref{eqpoly}) that has  {\em integer coefficients} and is {\em square-free}, that is, has only simple roots, compute $d$ disjoint discs on the complex plane, each containing exactly one root.
    
\begin{remark}\label{rertisl}
Problem  0$^*$ can be viewed as the
 {\bf Problem of Polynomial Root Cluster Isolation}. It turns into
  Problem ISOL if all  the coefficients of $p$ are integers and the roots are simple, and so our   algorithms for
   Problem  0$^*$ as well as the Algorithms of \cite{P95,P02}  solve Problem ISOL within a record  complexity  bounds.
 \end{remark}

In \cite[Sec. 20]{S82}
Sch{\"o}nhage  reduces  Problem ISOL
to factorization Problem F.
Problem ISOL is readily solved if the zeros are approximated within $\epsilon=1/2^b <0.5 \delta$ for
$\delta$ denoting the minimal pairwise distance between the zeros of $p$ (see the known lower bounds on $\delta$ in Kurt Mahler and Maurice Mignotte \cite{M64,M74,M95}). 
 Sch{\"o}nhage has estimated that the solution of Problem F (e.g., presented in \cite{S82,K98,P94,P95a,P02}) implies isolation
 provided that 
\begin{equation}\label{eqrtisl}
b=O((\tau+\log(d)) +d)~ {\rm and}~\tau=\log_2\max_{i=0}^d |p_i|.
\end{equation}
 
  \begin{corollary}\label{coislt} {\em The bit operation complexity of polynomial root isolation.}  
  Suppose that a polynomial $p$ of  (\ref{eqpoly}) has integer coefficients and only simple roots. 
  Let $\sigma_p$ denotes its {\em root separation}, that is, the minimal distance
  between a pair of its roots. Write  
 $\epsilon:=0.4 \sigma_p$ and 
 $b':=\log_2(\frac{1}{\epsilon})$. Let $\epsilon<2^{-d}$ and let
 $m=m_{p,\epsilon}$ denote the maximal number  of the roots 
  of the polynomial $p(x)$ in 
 $\epsilon$-clusters of its roots. Then one can solve Problem ISOL  
 by using  $\tilde O(bdm)$ bit operations for
$b=\frac{b'}{m}$.
  \end{corollary}

 
 
\section{Newton's, Weierstrass's, and Ehrlich's iterations}\label{snopth}  
  

\subsection{Contents}\label{snopthcntn}
 
In the next subsection we
recall Newton's, Weierstrass's, and Ehrlich's iterations and  comment on their convergence and  acceleration by means of fast multipoint 
evaluation, which is dramatic but involves the coefficients of $p$.  
In Sec. \ref{snopth2} we  cover more limited acceleration by means of incorporation of FMM, which can be applied to a black box polynomial.  
In Secs. \ref{sdscfrall}
and \ref{sexclexc2}
we cover initialization of Newton's, Weierstrass's, and Ehrlich's iterations, required for their fast global empirical convergence. For Newton's iterations 
we only need to compute a disc that contains all zeros of a polynomial $p$ (cf. \cite{KS94,HSS01,SS17,RSS24,S23}). This is given by the well-known estimates via the coefficients, but we 
present simple novel initialization recipes  in the case of a
 black box polynomial 
 $p$. The customary initialization is more demanding for  Weierstrass's and Ehrlich's iterations but is still fast for polynomials represented with their coefficients, as we recall in  Sec. \ref{sexclexc2}, where we also discuss applications of this initialization in particular to fast real and nearly real root-finding. It is still a  challenge to extend such initialization to the case of a  black box polynomial that would be faster than our radii algorithm of
 Sec. \ref{sextrrk}.


\subsection{Definitions}\label{snopth1} 

 
Consider Newton's, Weierstrass's, and Ehrlich's  
 iterations \cite{KS94,HSS01,RSS24,SS17,W03,D60,K66,E67,A73,BR14}):
 \begin{equation}\label{eqehrab}
z\leftarrow z-N(z),~z_j\leftarrow z_j-W(z_j),~z_j\leftarrow z_j-E_{j}(z_j),~j=1,\dots,d,~{\rm where} 
\end{equation} 
\begin{equation}\label{eqnwe0}
N(x)=W(x)=E_j(x)=0~{\rm for~all}~j~{\rm if}~ p(x)=0
\end{equation}
and otherwise 
\begin{equation}\label{eqe}
N(z):=\frac{p(z)}{p'(z)},~W(x):=\frac{p(x)}{l'(x)}, ~
\frac{1}{E_{j}(x)}:=\frac{1}{N(x)}-
\sum_{i=1,i\neq j}^d\frac{1}{x-z_i} 
\end{equation}
 for $l(x):=\prod_{j=1}^d(x-z_j)$.
  
  Newton's iterations can be initiated at a single point $z$ or  at many points and  can approximate  a single root, many or all 
  $d$ roots at the cost decreasing as the number of roots decreases  \cite{KS94,HSS01,SS17,S23,MV24,RSS24}. 
  
  Weierstrass's and Ehrlich's iterations begin at $d$ points for a $d$th degree polynomial and approximate a subset of its $d$ roots about as fast and almost as slow as all  its $d$ roots.
   
  Near the roots Newton's and Weierstrass's iterations 
  converge with quadratic rate,
   Ehrlich's iterations with cubic rate
   \cite{T98}. Empirically, however, all of them 
   consistently converge very fast globally, right from the start, 
  if they are 
    initialized properly.\footnote{Formal support for fast  global  convergence  is only known for for Newton's iterations  empirically they still run faster than according to the known estimates  (see \cite{RSS24}).}
 The known fast initialization algorithms involve the  coefficients of $p$ (see \cite{BF00,BR14}).

All expressions in (\ref{eqe}) for the functions $N(x)$, 
 $W(x)$, and  $E_j(x)$ 
 can be rewritten as  ratios 
 $\frac{u(x)}{v(x)}$
 for pairs of polynomials
 $u(x)$ and $v(x)$, whose coefficients are given or readily available at a low computational  cost if the polynomial $p(x)$ is given with its coefficients (see \cite[Problems 3.6.1 and 3.6.3]{P01}). Then multipoint evaluation of such a ratio amounts  to multipoint evaluation of $u(x)$ and $v(x)$
 and division of the computed values, and one can accelerate Newton's, Weierstrass's, and Ehrlich's iterations
by a factor $d/\log^2(d)$ by means of
incorporation of the algorithms of \cite{M21,IM23}. 


\subsection{Acceleration with Fast Multipole Method (FMM)}\label{snopth2}
 
 
 The algorithms of \cite{M21,IM23} deeply involve the coefficients of $p(x)$. Next we cover 
 acceleration of Newton's, Weierstrass's, and Ehrlich's iterations  applied to black box polynomial root-finding. 
  Given  $3d$ complex values  $w_j$, $z_i$, $v_j$ such that
 $z_i\neq v_j$ for all
 $i,j=1,\dots,d$, FMM  computes the sums $\sum_{j=1}^d\frac{w_j}{x_i-z_j}$
 as well as $\sum_{j=1,j\neq i}^d\frac{w_j}{z_i-z_j}$ for all $i$
   by using $O(db)$ arithmetic operations, 
  with the precision of order  
$b$,
  versus $2d^2$  arithmetic operations in the straightforward 
  algorithm. 
 The papers \cite{P15,P17a} list the 
 recent bibliography  on FMM and related computations with Hierarchical Semiseparable (HSS) matrices. Relevant 
software libraries developed at the
Max Planck Institute for Mathematics in the Sciences can be found in HLib,  

 The acceleration based on FMM 
 is practically valuable in spite of the limitation of the working and output
 precision to the order of $\log(d)$  and the
  resulting limitation on the 
 decrease of the bit operation complexity. 
 
Next we incorporate FMM  into
 Newton's, Weierstrass's, and Ehrlich's
iterations. First  
 relate the roots of a polynomial $p(x)$ to the eigenvalues of a diagonal-plus-rank-one matrix\footnote{Hereafter  $^T$ denotes  transposition of a vector.} 
 $$A=D+{\bf v}{\bf e}^T~{\rm for}~
 D=\diag(z_j)_{j=1}^d,~
  {\bf e}=(1)_{j=1}^d,~
 ~{\bf w}=(w_j)_{j=1}^d,~w_j=W(z_j),$$
  $W(x)$ of (\ref{eqe}) and all $j$
 (cf. \cite{BGP04}).  Let $I_d$ denote
 the $d\times d$ identity matrix, apply the Sherman-Morrison-Woodbury formula, and obtain $$xI_d-A=(xI_d-D)(I_d-
(xI_d-D)^{-1}{\bf w}{\bf e}^T).$$  
Hence
\begin{equation}\label{eqep} 
\frac{1}{p_d}p(x)=\det(xI_d-A)=-\sum_{i=1}^dw_i\prod_{j\neq i}(x-z_j)+\prod_{j=1}^d(x-z_j).
\end{equation}
Indeed, the monic polynomials
of degree $d$ on both sides have the same values for $x=z_j$, $j=1,\dots,d$ (cf. \cite{BGP04})).\footnote{Gerschgorin's theorem (cf. \cite{GL13}) for the matrix
$xI_d-A$ defines neighborhoods of its eigenvalues (the roots of $p$), which   helps to  analyze
and to modify the iterations (cf. \cite{BR14}).}  
Equations (\ref{eqep}) imply  that
$$p(z_i)=-w_i\prod_{j\neq i}(z_i-z_j),~
p'(z_i)=\prod_{j\neq i}(z_i-z_j)\Big(1-
\sum_{j\neq i}\frac{w_i+w_j}{z_i-z_j}\Big),$$
  and therefore 
\begin{equation}\label{eqenw}  
  \frac{w_i}{N(z_i)}=\sum_{j\neq i}\frac{w_i+w_j}{z_i-z_j}-1.
 \end{equation}
 Hence, given the $2d$ values of $z_j$ and $w_j=W(z_j)$ for all $j$,
 one can fast approximate $N(z_j)$
 and consequently $E_j(z_j)$ for all $j$
 (cf. (\ref{eqe})) by applying FMM.
    
Every iteration but the first one begins with some values $z_j$, $w_j$ and the
updated 
 values $z_j^{\rm (new)}$ of
 (\ref{eqehrab}),    
for  $j=1,\dots,d$,
and outputs the updated values  
 $w_j^{\rm (new)}=W(z_j^{\rm (new)})$
 for all $j$. 
  Then 
  equations (\ref{eqe}) and 
  (\ref{eqenw}) define
  $N(z_j^{\rm (new)})$ and
 $E_j(z_j^{\rm (new)})$ for all $j$.
 To update the values $w_j$,
 apply \cite[equation (17)]{BR14}:  
$$w_i^{\rm new}=(z_i^{\rm new}-z_i)\prod_{j\neq i} \frac{z_i^{\rm new}-z_j}{z_i^{\rm new}-z_j^{\rm new}}\Big (\sum_{k=1}^d\frac{w_k}{z_i^{\rm new}-z_k}-1\Big),~i=1,\dots,d.$$
 In the case where 
$z_i^{\rm new}=z_i$ if and only if $i>k$ this updating is still used for $i=1,\dots,k$
but is  simplified to
$$w_i^{\rm new}=w_i\prod_{j=1}^k\frac{z_i-z_j}{z_i^{\rm new}-z_j}~{\rm for}~ 
i=k+1,\dots, d.$$

Now  fast  compute the values
$\sum_{k=1}^d\frac{w_k}{z_i^{\rm new}-z_k}$ for all $i$ by applying FMM.

Also compute the values $$R_i=(z_i^{\rm new}-z_i)\prod_{j\neq i} \frac{z_i^{\rm new}-z_j}{z_i^{\rm new}-z_j^{\rm new}}$$  
 fast  in two ways.

(i)  Use the following customary trick 
   (see \cite[Remark 3]{DGR96}):   approximate   the values
  $\ln(R_i)$ 
  for $i=1,\dots,d$ fast by applying FMM
  and then readily approximate
  $R_i=\exp(\ln(R_i))$ 
   for all $i$.
 \cite{DGR96} proves 
 numerical stability of the  reduction to FMM  of a  similar although a little more involved rational expressions. 
     
 (ii)    Observe that 
 $$\ln\Big(\frac{R_i}{z_i^{\rm new}-z_j}\Big)=\sum_{j\neq i}\ln \Big(1-\frac{z_j^{\rm new}-z_j}{z_i^{\rm new}-z_j^{\rm new}}\Big)~{\rm for}~i=1,\dots,d$$ 
 and that $$\ln\Big(1-\frac{z_j^{\rm new}-z_j}{z_i^{\rm new}-z_j^{\rm new}}\Big)\approx 
 \frac{z_j-z_j^{\rm new}}{z_i^{\rm new}-z_j^{\rm new}}~{\rm if}~|z_j^{\rm new}-z_j|\ll |z_i^{\rm new}-z_j^{\rm new}|.$$ 
 Let the latter bound hold
  for  $j\in \mathbb J_i$ and let $\mathbb J'_i$ denote the set of the remaining indices 
  $j\neq i$. Then $$R_i\approx \exp(\sigma_i)\prod_{j\in\mathbb J'_i} \Big (1-\frac{z_j^{\rm new}-z_j}{z_i^{\rm new}-z_j^{\rm new}}\Big ),~
\sigma_i=\sum_{j\in  \mathbb J_i}\frac{z_j-z_j^{\rm new}}{z_i^{\rm new}-z_j^{\rm new}},~i=1,\dots,d.$$
By applying FMM 
 approximate $\sigma_i$ for all $i$ and  
 compute the values 
     $\prod_{j\in\mathbb J'} \frac{z_i^{\rm new}-z_j}{z_i^{\rm new}-z_j^{\rm new}}$ for all $i$ in at most 
$4\sum_{i=1}^d~|\mathbb J_i'|-d$
     arithmetic operations for $|\mathbb J'_i|$ denoting the cardinality of the set 
     $\mathbb J'_i$.  Compare the latter cost bound with order of  $d^2$ arithmetic operations in the  straightforward algorithm.
           
 Clearly  the above
 updating computation is simpler and more stable numerically
 for $W(z_j)$  versus
 $N(z_j)$ and even more so versus $E_j(z_j)$, for
 $j=1,\dots,d$.
 
 Let us comment on the changes where one incorporates FMM into 
  Newton's, Weierstrass's, and Ehrlich's  
 iterations implicitly 
 applied to the wild factor $f(x)$ of $p(x)$.
The same expression 
 (\ref{eqimpldfl}) defines Ehrlich's   iterations for both 
$f(x)$ and $p(x)$  by virtue of Thm. \ref{thratio}.  Applying
 Newton's  iterations to a wild factor one additionally subtract the sum $\sum_{j=1}^{d-w}\frac{1}{x-z_j}$ (cf. (\ref{eqimpldfl})) and can apply FMM to accelerate  numerical computation of this sum.
 Applying Weierstrass's iterations to a wild factor one performs
additional division by $t(x)$  and decreases
 the degree of the divisor $l'(x)$ 
 from $d-1$ to $w-1$. 


\subsection{MPSolve and its  competition with Newton's iterations}\label{saplimpl}
     
   Since its creation in 2000 the package MPSolve
   by  Bini and  Fiorentino \cite{BF00}, 
   revised in 2014 by Bini and  Robol \cite{BR14}, has been remaining user's choice
 for approximation of all
  $d$  complex  roots of a polynomial $p$. It implements Ehrlich's iterations, which  
 empirically  converge very fast  globally, right from the start,
 under proper initialization involving the coefficients of $p$.
 Such convergence has been consistently observed in  decades-long worldwide applications of these  iterations as well as in the case of some other root-finders by means of functional iterations, notably Weierstrass's  iterations.
 
Unsuccessful attempts of supporting this empirical behavior  formally (e.g., see \cite{AKY02,HKI10}), however,  have  recently ended  with  surprising proofs that   Weierstrass's and  Ehrlich's iterations diverge in the case of inputs that lie on some  complex manifolds,  being an open domain in the case of Weierstrass's  iterations (see Reinke, Schleicher, and M. Stoll
   \cite{RSSa,Ra}).   
    
 Serious  competition to MPSolve came in 2001 from the package EigenSolve by  Fortune \cite{F02}, but the later version of MPSolve \cite{BR14} combines the benefits of both packages. An implementations of Newton's iterations for approximation of all  roots by   Schleicher and  R. Stoll in \cite{SS17}  competed with MPSolve in 2017, but then MPSolve has regained  upper hand by means of incorporation of FMM. As we showed in  the previous subsection, however, such implementation is simpler 
 and can be more efficient for Newton's than Ehrlich's
 iterations. The Ehrlich-Newton competition seems to resume in 2024 (see \cite{B24,RSS24,MV24}).
  
 
\subsection{Computation of a disc  covering all zeros of a black box polynomial}\label{sdscfrall}

Customary initialization of Ehrlich's and Weierstrass's iterations involves 
approximation of all root radii, for which we can $d$ times apply our algorithm of Sec. \ref{sextrrk}, 
 greatly accelerated in the next section if $p(x)$ is given with its coefficients. In that case we immediately estimate the largest root radius $r_1$ 
to initialize Newton's  
iterations of \cite{SS17,S23,RSS24}.
In this subsection we 
obtain such estimates
for a black box polynomial.

\subsubsection{Introductory comments}
Given the coefficients of a polynomial $p$ of a degree $d$, we can readily compute its largest and smallest root radii $r_1$  and $r_d$, respectively (cf. \cite{F916,H74}).  
Next we compute  a range $[\rho_-,\rho_+]$ that  brackets all root radii or, equivalently, compute an upper estimate for the largest  root radius $r_1$ and  a low estimate for the smallest  root radius  $r_d$
of a black box polynomial $p(x)$ of a degree $d$. 

 Its solution is required, e.g., in order to initialize  algorithms  for all $d$ roots  by means of  functional iterations.
 
 Suppose that we have an algorithm that computes a disc $D(0,\rho_-)$ containing no zeros of a black box polynomial $p(x)$; clearly,
 $\rho_-\le r_d$. By  applying this algorithm to the reverse polynomial  $p_{\rm rev}(x)$ we  obtain  the upper bound
$1/\rho_{-,{\rm rev}}$ on $r_1$.

It remains to describe computation of the disc $D(0,\rho_-)$. Empirically we can do this by applying the algorithms of \cite{P22b,GPS23}, but next we study formally  some alternative algorithms.

To obtain a range $[\rho_-,\rho_+]$ that  brackets all root radii, we need  upper estimate for the  root radii $r_1$ and  $r_d$, but next we only  estimate $r_d$ because we can extend them to $r_1$ by estimating $r_d$ for the reverse polynomial  $p_{\rm rev}(x)$.
We obtain a low bound on $r_d$ by computing a disc that contains no zeros of a black box polynomial $p(x)$. 

 In the rest of this section $p(x)$ denotes a black box polynomial.

\subsubsection{Probabilistic solution} 

Seek a disc $D(c,\rho)$  containing no zeros of $p(x)$ in its interior, where we can choose any positive $\rho$ . 
Sample a random $c$
in the unit disc $D(0,1)$ under the uniform probability distribution
in $D(0,1)$. 

The disc has area $\pi$, 
while  the union of the $d$ discs $D(x_j,\rho)$, for $j=1,2,\dots,d$, has    
 area at most $\pi d\rho^2$.
 Therefore, $c$ lies in that union 
  with a probability  at most
  $d\rho^2$,
  which we can decrease 
  at will by decreasing $\rho$. Unless $c$ lies in such a union,
  we obtain a desired disc $D(c,\rho)$.
  
  By applying our e/i tests of Secs. \ref{seirndalg} -- \ref{smdfd}
  to this disc, we  can verify correctness of its choice.
   
   \subsubsection{Deterministic solution}  
   
We can  deterministically compute a desired disc $D(0,\rho_-)$ at a higher computational cost, based on a soft e/i test. 

Namely,  apply  1-test to 
$d+1$ triples of parameters  
$c_i$, $\theta_i$, and $\rho_i$, for
$i=0,1,\dots,d$, such that the discs
$D(c_i,\theta_i\rho_i)$
have no overlap pairwise.
Then, clearly,  at most 
$d$ discs contain roots because $p(x)$ 
has only $d$ zeros, and so at least one of the $d+1$ discs contains no roots.                                                        
  Deterministic e/i tests
  of Secs. \ref{sbscei} and \ref{smdfd} involve
  evaluation of $p(x)$ at
$2d+2$ points. We obtain a desired disc $D$ by applying these tests to at most $d+1$ discs and thus computing $p(x)$ at
at most $2(d+1)^2$ points.
 
  
\section{Fast approximation of the $d$ root radii with applications}\label{sexclexc2} 
   
  
 The customary initialization recipes of
 Weierstrass's and Ehrlich's iterations 
 involve the $d$ root radii at the origin
 and consequently require to use the coefficients of $p$. To initialize
 Ehrlich's iterations  MPSolve  first applies Bini's heuristic recipe of \cite{B96}, reused in  \cite{BF00,BR14}, for the approximation of the $d$ root radii.
Given such approximations MPSolve  defines $k$ concentric annuli 
$A(c,\rho_i,\rho_i')$ and $k$ integers $m_i$ such that
$\#(A(c,\rho_i,\rho_i'))=m_i$  for  $i=1,\dots,k$, $k\le d$,  and  $\sum_{i=1}^km_i=d$. Then MPSolve initializes Ehrlich's iterations at $d$ points $z_j$, $j=1,2,\dots,d$,  exactly
 $m_i$
of them being equally spaced on the circle $C(c,(\rho_i+\rho_i')/2)$.
This heuristic initialization  recipe turns out  to be  highly efficient.  

The algorithm supporting the next theorem also approximates the  $d$ root radii. It appeared much earlier than Bini's, is deterministic, formally supported, and faster than Bini's. The algorithm has  been devised  and briefly analyzed in \cite{S82} for fast approximation of a single root radius and then readily extended
in \cite[Sec. 4]{P87} (cf. also \cite[Sec. 5]{P89}) 
 to approximation of all root radii of a polynomial. As in \cite{B96,BF00,BR14} it uses all $d$  coefficients of $p$. A variation of this algorithm in \cite{IP21} supports the following theorem.

  
\begin{theorem}\label{thrr}  \cite[Prop. 3]{IP21}.   
 Given the coefficients of a polynomial $p=p(x)$, one can approximate all the
 $d$ {\em root radii} $|x_1|,\dots,|x_d|$  within the relative error bound 
$4d$ at a Boolean cost in $O(d\log(||p||_{\infty}))$. 
\end{theorem} 

\cite{IP21} extends the supporting algorithm to approximation of all root radii within a positive relative error bound $\Delta$
by performing  
$\lceil\log_2\log_{1+\Delta}(2d)\rceil$ DLG root-squaring iterations 
(\ref{eqdnd}).

Combine 
Thm. \ref{thrr} with estimates for the Boolean cost of performing these iterations implicit in the proof of 
\cite[Cor. 14.3]{S82} and obtain

 
\begin{corollary}\label{corr}   
 Given the coefficients of a polynomial $p=p(x)$ and a positive $\Delta$, one can approximate all the
 $d$ {\em root radii} $|x_1|,\dots,|x_d|$  within a relative error bound 
$\Delta$ in $\frac{1}{d^{O(1)}}$ at a Boolean cost in $O(d\log(||p||_{\infty})+d^2\log^2(d))$. 
\end{corollary} 
The algorithm  supporting the corollary  
 and said to be 
{\bf Alg. \ref{thrr}a}  computes at a low cost $d$ {\em suspect annuli} 
$A(0,r_{j,-},r_{j,+})$  for $j=1,\dots,d$
(some may  
overlap),  whose union $\mathbb U$ contains all $d$ roots. 
The algorithm was the point of departure
for amazing progress 
in \cite{IM23}, which extends the  ingenious  algorithm
of \cite{M21} for multipoint polynomial evaluation.  

We conclude this subsection with a list of {\bf some other benefits of fast approximation of the $d$ root radii.}

(i) In a subdivision step one can skip testing exclusion and discard a  
square unless it intersects the set $\mathbb U$. 
 
(ii) The search for real roots can  be narrowed to the intersection $\mathbb I$ of $\mathbb U$ with the real axis, which consists of $2m$ segments of small length  
(see Fig. \ref{figrr}).
 \cite{PZ15} has tested a real root-finder that approximates the roots in  $\mathbb I$ by initiating Newton's iterations at the centers of the segments. According to these tests  the algorithm is highly efficient for 
  the approximation of all well-isolated
 real roots.  One can extend Newton's iterations at the centers of the segments by applying also  our root-radii algorithm of Sec. \ref{sextrrk}.
  By complementing  the root radii approach with some advanced novel techniques, the paper \cite{IP21} has   noticeably accelerated the real root-finder \cite{KRS16}, which is currently   user's choice. 
 
(iii) Initialization based on  Alg. \ref{thrr}a should improve global convergence 
of Ehrlich's, Weierstrass's, and other 
functional iterations for root-finding
as well as of Kirrinnis's factorization algorithm of \cite{K98}. 

(iv) Suppose that Alg. \ref{thrr}a has computed
an annulus $A(0,\rho_-,\rho_+)$ containing a small 
number $u$ of the zeros of $p$ and isolated from its external zeros of $p$. 
Then we can extend the root-finding recipes of \cite{S82} from a disc to an annulus as follows. 
(a) Readily 
compute the first $u$ power sums of the zeros of $p$ lying in 
the discs $D(0,\rho_+)$ and $D(0,\rho_-)$ and then (b) their differences 
equal to the power sums
of the zeros of $p$ lying in the annuls  $A(0,\rho_-,\rho_+)$.
(c) By solving the system of Newton's identities,  compute the coefficients of the factor $f$ of the $u$th degree that shares with $p$ its root set in the annulus (cf. \cite[Sec. 12]{S82}). 
(d) Approximate the zeros of $f$, at a lower cost since $u$ is small.

We can apply this recipe to every annulus containing a small 
number of zeros of $p$  and isolated from its external zeros and finally approximate the remaining wild zeros of $p$ by applying  implicit deflation
(cf. (\ref{eqimpldfl})).
 
\begin{figure} 
\begin{center}
\begin{tikzpicture}
\draw[->,thick] (-4,0)--(4,0) node[right]{$x$};
\draw[->,thick] (0,-3)--(0,3) node[above]{$y$};
\draw[black] (0,0) circle (2.5cm);
\draw[black] (0,0) circle (2.6cm);
\draw[black] (0,0) circle (1.5cm);
\draw[black] (0,0) circle (1.6cm);
\draw[->] (3, 2.5) node[above]{may contain roots} -- (0.77, 1.34);
\draw[->] (3, 2.5) -- (2.24, 1.13);
\draw[->] (3, 2.5) -- (2.55, 0);
\draw[->] (-3, 2.5) node[above]{may contain roots} -- (-1.55, 0);
\draw[->] (-3, 2.5) -- (-1.80, 1.80);
\draw[->] (-3, 2.5) -- (-2.55, 0);\end{tikzpicture}
\end{center}
\caption{Two approximate root radii define two narrow suspect annuli and four suspect segments of the real axis.}
\label{figrr}
\end{figure} 

 
\medskip
 



\end{document}